\documentclass{amsart}
\usepackage[utf8]{inputenc}

\usepackage{amssymb}
\usepackage{amsmath}
\usepackage{amsfonts}
\usepackage{tikz-cd}

\tikzset{
curvarr/.style={
  to path={ -- ([xshift=2ex]\tikztostart.east)
    |- (#1) [near end]\tikztonodes
    -| ([xshift=-2ex]\tikztotarget.west)
    -- (\tikztotarget)}
  }
}

\tikzset{
  curvedlink/.style={
    to path={
      let \p1=(\tikztostart.east), \p2=(\tikztotarget.west),
      \n1= {abs(\y2-\y1)/4} in
      (\p1) arc(90:-90:\n1) -- ([yshift=2*\n1]\p2) arc (90:270:\n1)
    },
  }
}

\usepackage{amsthm}
\usepackage{enumitem}
\usepackage{mathrsfs}
\usetikzlibrary{decorations.pathmorphing}
\usepackage[utf8]{inputenc}
\usepackage[T1]{fontenc}
\usepackage{indentfirst}
\usepackage{xcolor}
\usepackage{mathtools}
\usepackage{bm}
\usepackage{stmaryrd}
\usepackage{lineno}
\usepackage{hyperref}
\hypersetup{
    colorlinks,
    citecolor=black,
    filecolor=black,
    linkcolor=black,
    urlcolor=black
}

\newtheorem{theorem}{Theorem}[section]
\newtheorem{lemma}[theorem]{Lemma}

\newtheorem{conjecture}[theorem]{Conjecture}
\newtheorem{proposition}[theorem]{Proposition}
\newtheorem{corollary}[theorem]{Corollary}

\theoremstyle{definition}
\newtheorem{define}[theorem]{Definition}
\newtheorem{example}[theorem]{Example}

\newtheorem{convention}[theorem]{Convention}

\theoremstyle{remark}
\newtheorem{remark}[theorem]{Remark}

\numberwithin{equation}{section}

\DeclareMathOperator{\Hom}{Hom}
\DeclareMathOperator{\Ext}{Ext}

\DeclareMathOperator{\Der}{Der}

\DeclareMathOperator{\ann}{ann}

\DeclareMathOperator{\Div}{Div}

\DeclareMathOperator{\Var}{V}

\DeclareMathOperator{\depth}{depth}
\DeclareMathOperator{\Spec}{Spec}

\DeclareMathOperator{\reg}{reg}
\DeclareMathOperator{\rank}{rank}
\DeclareMathOperator{\qi}{q.i.}
\DeclareMathOperator{\deRham}{DR}
\DeclareMathOperator{\anal}{an}
\DeclareMathOperator{\derivedR}{\mathbf{R}}
\DeclareMathOperator{\localSystem}{L}
\DeclareMathOperator{\Exp}{Exp}
\DeclareMathOperator{\Cech}{\check{C}ech}
\DeclareMathOperator{\alg}{alg}

\usepackage [english]{babel}
\usepackage [autostyle, english = american]{csquotes}
\MakeOuterQuote{"}

\author{Daniel Bath}
\title[Arrangements: (un)Twisted Logarithmic Comparison Theorems]{Hyperplane Arrangements Satisfy (un)Twisted Logarithmic Comparison Theorems, Applications to $\mathscr{D}_{X}$-modules}
\address{Department of Mathematics, KU Leuven, Leuven, Belgium.}
\email{dan.bath@kuleuven.be}
\thanks{The author is supported by FWO grant \#G097819N}
\begin{document}
\sloppy
\maketitle
\begin{abstract}
For a reduced hyperplane arrangement we prove the analytic Twisted Logarithmic Comparison Theorem, subject to mild combinatorial arithmetic conditions on the weights defining the twist. This gives a quasi-isomorphism between the twisted logarithmic de Rham complex and the twisted meromorphic de Rham complex. The latter computes the cohomology of the arrangement's complement with coefficients from the corresponding rank one local system. We also prove the algebraic variant (when the arrangement is central), and the analytic and algebraic (untwisted) Logarithmic Comparison Theorems. The last item positively resolves an old conjecture of Terao. We also prove that: every nontrivial rank one local system on the complement can be computed via these Twisted Logarithmic Comparison Theorems; these computations are explicit finite dimensional linear algebra. Finally, we give some $\mathscr{D}_{X}$-module applications: for example, we give a sharp restriction on the codimension one components of the multivariate Bernstein--Sato ideal attached to an arbitrary factorization of an arrangement. The bound corresponds to (and, in the univariate case, gives an independent proof of) M. Saito's result that the roots of the Bernstein--Sato polynomial of a non-smooth, central, reduced arrangement live in $(-2 + 1/d, 0).$
\end{abstract}

\tableofcontents

\section{Introduction}

Throughout $f=f_{1} \cdots f_{d} \in R = \mathbb{C}[x_{1}, \dots, x_{n}]$ is a defining equation, factored into linears, of a $\deg(f) = d$ reduced hyperplane arrangement $\mathscr{A}$ in $X = \mathbb{C}^{n}$ and with complement $U = X \setminus \mathscr{A}$. By default, we regard $X$ as a Stein manifold with analytic structure sheaf $\mathscr{O}_{X}$. Along with the de Rham complex $(\Omega_{X}^{\bullet}, d)$ there is the \emph{meromorphic de Rham complex}
\[
(\Omega_{X}^{\bullet}(\star \mathscr{A}), d) = 0 \to \Omega_{X}^{0}(\star \mathscr{A}) \xrightarrow[]{d} \cdots \xrightarrow[]{d} \Omega_{X}^{n}(\star \mathscr{A}) \to 0 
\]
where $d$ is the exterior derivative and $(\star \mathscr{A})$ means we allow poles of arbitrary order along $\mathscr{A}$. K. Saito \cite{SaitoLogarithmicForms} introduced the \emph{logarithmic de Rham complex} $(\Omega_{X}^{\bullet}(\log \mathscr{A}), d)$ as the subcomplex of the meromorphic de Rham complex characterized by 
\[
\Omega_{X}^{j}(\log \mathscr{A}) = \{ \eta \in \Omega_{X}^{j}(\mathscr{A}) \mid d(\eta) \in \Omega_{X}^{j+1}(\mathscr{A}) \}.
\]
That is, both $\eta$ and $d(\eta)$ have poles of order at most one along $\mathscr{A}$. Note that these are complexes of sheaves of $\mathscr{O}_{X}$-modules, neither of which depends on the choice of defining equation $f$ of $\mathscr{A}$. There are algebraic analogues: the algebraic logarithmic de Rham complex $\Omega_{R}^{\bullet}(\log \mathscr{A})$; the algebraic rational de Rham complex $\Omega_{R}^{\bullet}(\star \mathscr{A})$. These complexes of $R$-modules have similar definitions to the analytic ones.

Our original motivation was to answer a well studied conjecture (Conjecture 3.1 of \cite{TeraoLCTforArrangementsConjecture}) of Terao from 1977:

\begin{conjecture} \label{conjecture -  algebraic lct for arrangements} \text{(Terao's LCT Conjecture)} The Algebraic Logarithmic Comparison Theorem holds for reduced hyperplane arrangements. That is, the natural inclusion of the algebraic logarithmic de Rham complex into the algebraic rational de Rham complex is a quasi-isomorphism:
\[
(\Omega_{R}^{\bullet}(\log \mathscr{A}), d) \xhookrightarrow{\qi} (\Omega_{R}^{\bullet}(\star \mathscr{A}), d) \quad (\simeq H^{\bullet}(U, \mathbb{C}_{U})).
\]
\end{conjecture}
\noindent Here $\mathbb{C}_{U}$ is the the rank one $\mathbb{C}$-valued constant sheaf on $U$. 

That $H^{\bullet}(\Omega_{R}^{\bullet}(\star \mathscr{A}), d) \simeq H^{\bullet}(U, \mathbb{C}_{U})$ is a consequence of Grothendieck's Comparison Theorem \cite{GrothendieckOnTheDeRham}. So Conjecture \ref{conjecture -  algebraic lct for arrangements} promises that logarithmic data of $\mathscr{A}$ determines the cohomology of the complement via the conjectured quasi-isomorphism.

We prove Conjecture \ref{conjecture - algebraic lct for arrangements} as a corollary to more general result: the Twisted Logarithmic Comparison Theorem. This promises that twisting the differential of the logarithmic de Rham complex by a certain logarithmic form computes the cohomology of the complement $U$ in the corresponding rank one local system. 

To explain this, let $\lambda_{1}, \dots, \lambda_{d} \in \mathbb{C}$ be weights, one for each hyperplane $\Var(f_{k})$ of $\mathscr{A}$, and let the weights define a logarithmic one form
\[
\omega = \sum_{k} \lambda_{k} \frac{df_{k}}{f_{k}} = \sum_{k} \lambda_{k} d \log f_{k}.
\]
Using $\omega$ we twist the exterior derivative for both the logarithmic and meromorphic de Rham complex (along with the algebraic analogues): we replace the differential $d(-)$ with the twisted differential $\nabla_{\omega}(-) = d(-)+ \omega \wedge (-).$ 

On the local system side, it is well known that all rank one local systems on the complement $U = X \setminus \mathscr{A}$ are parameterized by torus points $\boldsymbol{\beta} \in (\mathbb{C}^{\star})^{d}$. Essentially, each $\beta_{k}$ encodes the monodromy around the hyperplane $\Var(f_{k})$. Denote $\localSystem_{\boldsymbol{\beta}}$ the local system corresponding to $\boldsymbol{\beta}$. By Deligne and Grothendieck's algebraic de Rham Theorems (see the summary of the arrangement case in \cite{OrlikHypergeometricIntegralsAndArrangements}), the algebraic twisted de Rham complex computes the cohomology on $U$ with one of these local systems:
\[
H^{\bullet}(\Omega_{R}^{\bullet}(\star \mathscr{A}), \nabla_{\omega}) \simeq H^{\bullet}(U, \localSystem_{\Exp(\boldsymbol{\lambda})}).
\]
\noindent $\localSystem_{\Exp(\boldsymbol{\lambda})}$ corresponds to the torus point obtained by replacing each component $\lambda_{k}$ of the weights $\boldsymbol{\lambda} = (\lambda_{1}, \dots, \lambda_{d})$ with $e^{2 \pi i \lambda_{k}}$.

Our two main results are: the analytic Twisted Logarithmic Comparison Theorem; the algebraic Twisted Logarithmic Comparison Theorem. In both cases, we have very mild combinatorial arithmetic restrictions on the weights defining the twist. Recall $\mathscr{L}(A)$ is the intersection lattice of $\mathscr{A}$; its constituents are called edges. The majority of work is in proving our analytic Theorem \ref{thm - intro, analytic twisted lct}:

\begin{theorem} \label{thm - intro, analytic twisted lct}
(Analytic Twisted Logarithmic Comparison Theorem) Let $f = f_{1} \cdots f_{d} \in R$ cut out a reduced hyperplane arrangement $\mathscr{A}$. Suppose that $\lambda_{1}, \dots, \lambda_{d} \in \mathbb{C}$ are weights such that, for each edge $E \in \mathscr{L}(\mathscr{A})$,
\begin{equation} \label{eqn - intro, analytic twisted lct weight condition}
\sum_{\{1 \leq k \leq d \mid E \subseteq \Var(f_{k})\}} \lambda_{k} \notin \mathbb{Z}_{\geq \min\{2, \rank(E)\}}.
\end{equation}
Let $\omega = \sum_{k} \lambda_{k} \frac{df_{k}}{f_{k}}$ be
the logarithmic one form determined by the $\{\lambda_{k}\}$. Then the analytic Twisted Logarithmic Comparison Theorem with respect to $\{\lambda_{k}\}$ holds:
\begin{equation*} 
(\Omega_{X}^{\bullet}(\log \mathscr{A}), \nabla_{\omega}) \xhookrightarrow{\qi} (\Omega_{X}^{\bullet}(\star \mathscr{A}), \nabla_{\omega}) \quad (= \derivedR j_{\star} \localSystem_{\Exp(\boldsymbol{\lambda})}).
\end{equation*}
\end{theorem}
\noindent The statement $``(= \derivedR j_{\star} \localSystem_{\Exp(\boldsymbol{\lambda})})''$ means there is equality in the derived category (of perverse sheaves), where $j: U \xhookrightarrow{} X$ is the inclusion.

From this we deduce the algebraic Theorem \ref{thm - global algebraic twisted lct with homogeneous subcomplex q.i}, though here impose centrality:

\begin{theorem} \label{thm - intro, global algebraic twisted lct with homogeneous subcomplex q.i} (Algebraic Twisted Logarithmic Comparison Theorem)
Let $f = f_{1} \cdots f_{d}$ cut out a central, reduced hyperplane arrangement $\mathscr{A}$ and let $\lambda_{1}, \dots, \lambda_{d} \in \mathbb{C}$ be weights such that, for each edge $E$,
\[
\sum_{\{1 \leq k \leq d \mid E \subseteq \Var(f_{k})\}} \lambda_{k} \notin \mathbb{Z}_{\geq \min\{2, \rank(E)\}}.
\]
Furthermore, let $\omega = \sum_{k} \lambda_{k} df_{k} / f_{k}$. Then the algebraic Twisted Logarithmic Comparison Theorem holds along with an additional quasi-isomorphism:
\[
(\Omega_{R}^{\bullet}(\log \mathscr{A})_{-\iota_{E}(\omega)}, \nabla_{\omega}) \xhookrightarrow{\qi} (\Omega_{R}^{\bullet}(\log \mathscr{A}), \nabla_{\omega}) \xhookrightarrow{\qi} (\Omega_{R}^{\bullet}(\star \mathscr{A}), \nabla_{\omega}) \quad (\simeq  H^{\bullet}(U, \localSystem_{\Exp(\boldsymbol{\lambda})})).
\]
Here $(\simeq  H^{\bullet}(U, \localSystem_{\Exp(\boldsymbol{\lambda})}))$ means there are isomorphisms on the level of cohomology and $(\Omega_{R}^{\bullet}(\log \mathscr{A})_{-\iota_{E}(\omega)}, \nabla_{\omega})$ is the homogeneous subcomplex of degree $- \iota_{E}(\omega) = -( \lambda_{1} + \cdots + \lambda_{d})$, which is a complex of finite dimensional $\mathbb{C}$-vector spaces.
\end{theorem}

The combinatorial arithmetic conditions \eqref{eqn - intro, analytic twisted lct weight condition} on the intersection lattice are relatively benign. (And are closely related to data about Bernstein--Sato ideals, cf. Theorem \ref{thm - intro, bounding codim one components of BS zero loci}, Section 4.) We justify this assessment in Proposition \ref{prop - all local systems can be computed}:

\begin{proposition} \label{prop - intro, all local systems can be computed}
For an arbitrary rank one local system $\localSystem_{\boldsymbol{\beta}}$ on $U$, there exists a $\boldsymbol{\lambda} \in \Exp^{-1}(\boldsymbol{\beta})$ such that all the aforementioned Twisted Logarithmic Comparison Theorems hold with respect to the weights $\boldsymbol{\lambda} = (\lambda_{1}, \dots, \lambda_{d}).$
\end{proposition}

These Twisted Logarithmic Comparison Theorems can be considered as an extension of the results in \cite{CohomologyofLocalSystemsOnTheComplementOfHyperplanes}. Therein projective hyperplane arrangements are considered, but using a Leray spectral sequence there is a precise cohomological relationship between local systems on the complement of a central affine arrangement and the corresponding local system on the complement of the associated projective arrangement, cf. Theorem 5.2 of \cite{DimcaHyperplaneArrangements}. In \cite{CohomologyofLocalSystemsOnTheComplementOfHyperplanes}, the authors prove that the twisted Orlik--Solomon algebra, i.e. the twisted Brieskorn algebra, with differential wedging $\omega = \sum \lambda_{k} df_{k}/f_{k}$, computes the cohomology of the associated rank one local system on the projective complement provided that: (1) the sum $\lambda_{1} + \cdots + \lambda_{k} = 0$; (2) a similar condition to \eqref{eqn - intro, analytic twisted lct weight condition} holds on the residues of $\omega$ at dense edges. The first requirement means these twisted Orlik--Solomon algebra can only detect local systems belong to non-torsion translated components of the characteristic variety, i.e. resonance varieties, cf. Example 4.1 of \cite{SuciuTranslatedTori} or Section 3 of \cite{CohenTriplesofArrangements}. This is a feature of twisted Orlik--Solomon algebras: they are blind to the local systems populating the torsion translated components by \cite{CohenSuciuTangent}. 

In contrast, any local system on the complement of an arrangement can be detected via these Twisted Logarithmic Comparison Theorems, after an appropriate twist, essentially because there is no stricture like $\lambda_{1} + \cdots + \lambda_{d} = 0$. So Theorem \ref{thm - intro, analytic twisted lct} and Theorem \ref{thm - intro, global algebraic twisted lct with homogeneous subcomplex q.i} extend the results of \cite{CohomologyofLocalSystemsOnTheComplementOfHyperplanes} to all rank one local systems, at the cost of replacing the Orlik--Solomon algebra with the logarithmic de Rham complex. To demonstrate this, in Example \ref{ex - deleted B3 arrangement} we use our results to compute the Betti numbers of $H^{\bullet}(U, \localSystem_{\boldsymbol{\beta}})$ for a special local system on the complement of the deleted $B_{3}$ arrangement. This system $\localSystem_{\boldsymbol{\beta}}$ is invisible to Orlik--Solomon methods.

From Theorem \ref{thm - intro, analytic twisted lct} we quickly conclude the (untwisted) analytic Logarithmic Comparison Theorem holds in Corollary \ref{cor - analytic untwisted lct}:

\begin{corollary} \label{cor - intro, analytic untwisted lct} (Analytic Logarithmic Comparison Theorem)
For a reduced hyperplane arrangement $\mathscr{A}$, the analytic Logarithmic Comparison Theorem holds:
\[
(\Omega_{X}^{\bullet}(\log \mathscr{A}), d) \xhookrightarrow{\qi} (\Omega_{X}^{\bullet}(\star \mathscr{A}), d) \quad (= \derivedR j_{\star} \mathbb{C}_{U}).
\]
\end{corollary}

We also prove the algebraic (untwisted) Logarithmic Comparison Theorem in Corollary \ref{cor - algebraic untwisted lct}, positively resolving Terao's Conjecture \ref{conjecture -  algebraic lct for arrangements}. Here there is no centrality assumption and $A^{\bullet}(\mathscr{A})$ is the Orlik--Solomon algebra.

\begin{corollary} \label{cor - intro, algebraic untwisted lct}
(Algebraic Logarithmic Comparison Theorem) Let $\mathscr{A}$ be a reduced hyperplane arrangement. Then the algebraic Logarithmic Comparison Theorem holds giving a sequence of quasi-isomorphisms
\begin{equation*} 
A^{\bullet}(\mathscr{A}) \xhookrightarrow{\qi} (\Omega_{R}^{\bullet}(\log \mathscr{A}), d) \xhookrightarrow{\qi} (\Omega_{R}^{\bullet}( \star \mathscr{A}), d) \quad (\simeq H^{\bullet}(U, \mathbb{C}_{U})).
\end{equation*}
Moreover, $A^{j}(\mathscr{A}) = H^{j}(\Omega_{R}^{\bullet}(\log \mathscr{A}), d)$.
\end{corollary}

Previously Conjecture \ref{conjecture -  algebraic lct for arrangements} has been confirmed in special cases of arrangements and other divisors: necessary and sufficient conditions for quasi-homogeneous divisors with isolated singularities \cite{MondLogarithmicDifferentialFormsAnd}; strongly quasi-homogeneous free divisors \cite{CohomologyComplementFree} (in this setting, see also \cite{MorenoMacarroLogarithmic} for a less explicit twisted version than what we find and \cite{MAcarroDuality} for the current state of the art); tame arrangements and arrangements of rank at most 4 \cite{YuzvinskyWiensLTCTameArrangements}. Both free and tame impose homological restrictions on the logarithmic forms: free means $\Omega_{R}^{1}(\log \mathscr{A})$ is a free $R$-module; tame means the projective dimension of $\Omega_{R}^{j}(\log \mathscr{A})$ is at most $j$.

Our approach to the analytic Twisted Logarithmic Comparison Theorem (Theorem \ref{thm - intro, analytic twisted lct}) inhabits the architecture of Castro Jim{\'e}nez, Mond, and Narv{\'a}ez Macarro's inductive proof of the analytic (untwisted) Logarithmic Comparison Theorem for strongly quasi-homogeneous free divisors \cite{CohomologyComplementFree}, with the induction happening on the rank of edges $E \in \mathscr{L}(\mathscr{A})$ (or equivalently on the dimension of logarithmic strata as in loc.\ cit.\, see \cite{SaitoLogarithmicForms}). Within this inductive scheme, the obstacle is understanding a spectral sequence with first page
\begin{equation} \label{eqn - intro, first page of the spectral sequence}
^{\prime \prime} E_{1}^{p,q} = H^{q}(V \setminus 0, \Omega_{X}^{p}(\log \mathscr{A})),
\end{equation}
where $V \ni 0$ is a small Stein open, $\mathscr{A}$ is central, and this page's vertical differential is induced by $\nabla_{\omega}$. The goal is to show only certain entries survive on the second page. Homological restrictions on the logarithmic $j$-forms limit the complexity of this first page: in \cite{CohomologyComplementFree}, under freeness, only two columns may not vanish; in \cite{YuzvinskyWiensLTCTameArrangements}, under tameness, the problematic entries are confined to two diagonal lines. With these simplifications, both \cite{CohomologyComplementFree} and \cite{YuzvinskyWiensLTCTameArrangements} can then use graded data to show the required entries on the second page vanish (though \cite{YuzvinskyWiensLTCTameArrangements} crucially invokes Brieskorn's Theorem at this step).

Without any homological assumptions the first page \eqref{eqn - intro, first page of the spectral sequence} is arbitrarily daedal. However, the first page's data is governed by the complexes of modules of analytic \v{C}ech cohomology
\begin{equation} \label{eqn - intro, analytic second page to global twisted Cech}
(H_{\Cech}^{t}(\Omega_{X}^{\bullet}(\log \mathscr{A})), \nabla_{\omega})
\end{equation}
attached to the open cover $\{D(x_{i})\}$ of $X \setminus 0$. In subsection 2.3 we show the cohomology of \eqref{eqn - intro, analytic second page to global twisted Cech} is determined by a ``homogeneous'' subcomplex in the sense of Laurent expansions. Thus, to show the necessary entries of our spectral sequence's second page vanish, it suffices to show that this ``homogeneous'' subcomplex is the zero complex. Corollary \ref{cor - specific weights at origin, acyclic complex of analytic Cech cohomology} uses a GAGA argument to relate this vanishing to the vanishing of certain graded components of the algebraic local cohomology: in notation, to $H_{\mathfrak{m}}^{t}(\Omega_{R}^{\bullet}(\log \mathscr{A}))_{\ell}$. We show the appropriate components are zero in Theorem \ref{thm - computing CM regularity of log forms}, where we bound the Castelnuovo--Mumford regularity of $\Omega_{R}^{j}(\log \mathscr{A})$: 

\begin{theorem} \label{thm - intro, computing CM regularity of log forms}
For $\mathscr{A}$ a central, essential, reduced hyperplane arrangement, the Castelnuovo--Mumford regularity of $\Omega_{R}^{j}(\log \mathscr{A})$ is bounded by:
\[
\reg(\Omega_{R}^{j}(\log \mathscr{A})) \leq 0 \quad \text{for} \quad 1 \leq j \leq \rank(\mathscr{A}) - 1.
\]
When $j = 0$ or $j=\rank(\mathscr{A})$, the module of logarithmic zero or $\rank(\mathscr{A})$-forms are free $R$-modules and so only $H_{\mathfrak{m}}^{\rank{\mathscr{A}}}(-) \neq 0$. In these cases, $\reg(\Omega_{R}^{0}(\log \mathscr{A})) = 0$ and $\reg(\Omega_{R}^{\rank(\mathscr{A})}(\log \mathscr{A})) = \rank(\mathscr{A}) - \deg(\mathscr{A})$. 
\end{theorem}

We obtain this using linear approximation methods due to Derksen and Sidman \cite{DerksenSidman-CMRegularyByApproximation}. Therein they bound the regularity of the logarithmic derivations of an arrangement (and M. Saito improved this in \cite{SaitoDegenerationOfPoleOrderSpectralSequence4Variables}), but there is nontrivial difficulty in applying their techniques to the case of logarithmic differential forms.

\vspace{5mm}

We conclude with some applications to classical questions for $\mathscr{D}_{X}$-modules. This helps demystify the combinatorial arithmetic conditions \eqref{eqn - intro, analytic twisted lct weight condition}. Here $\mathscr{D}_{X}$ is the sheaf of $\mathbb{C}$-linear analytic differential operators. One expects the analytic Logarithmic Comparison Theorem to inform certain $\mathscr{D}_{X}$-constructions, especially Bernstein--Sato polynomials. See the surveys \cite{TorelliLogarithmicComparisionTheoremAnd}, \cite{NarvaezMacarroLinearityConditions}. For some free divisors there is an intrinsic $\mathscr{D}_{X}$-theoretic formulation of the Logarithmic Comparison Theorem as developed in \cite{MorenoMacarroLogarithmic}; see also \cite{MAcarroDuality} for current developments. 

The existing literature primarily focuses on the untwisted, univariate case, and thus questions like: does the cyclic $\mathscr{D}_{X}$-module generated by $f^{-1}$ equal (as a submodule) the $\mathscr{D}_{X}$-module $\mathscr{O}_{X}(\star \mathscr{A})$ (the module with poles of arbitrary order along $\mathscr{A})$? This corresponds to $-1$ being the smallest $\mathbb{Z}$-root of the Bernstein--Sato polynomial of $f$, which was proved for arrangements by Leykin in \cite{WaltherGeneric}.

We ask a more subtle question: for a fixed $\boldsymbol{\lambda} \in \mathbb{C}^{d}$, does the cyclic $\mathscr{D}_{X}$-module generated by $f^{\boldsymbol{\lambda}} = f_{1}^{\lambda_{1}} \cdots f_{d}^{\lambda_{d}}$ equal (as a submodule) the $\mathscr{D}_{X}$-module $\mathscr{O}_{X}(\star f^{\lambda}) = \cup_{\textbf{p} \in \mathbb{N}^{d}} \mathscr{D}_{X} f^{\boldsymbol{\lambda} - \textbf{p}}$? Using ideas of Torelli \cite{TorelliLogarithmicComparisionTheoremAnd}, in Theorem \ref{thm - twisted LCT implies generating level of meromorphic specialization} we give an answer, contingent on our familiar combinatorial arithmetic properties on $\boldsymbol{\lambda}$:

\begin{theorem} \label{thm - intro, twisted LCT implies generating level of meromorphic specialization}

Let $f = f_{1} \cdots f_{d}$ cut out a central, reduced hyperplane arrangement and let $\boldsymbol{\lambda} = (\lambda_{1}, \dots, \lambda_{d}) \in \mathbb{C}^{d}$ be weights such that for each edge $E$
\[
\sum_{\{1 \leq k \leq d \mid E \subseteq \Var(f_{k})\}} \lambda_{k} \notin \mathbb{Z}_{\geq \min\{2, \rank(E)\}}.
\]
Then 
\[
\mathscr{D}_{X} f^{-\textbf{1} + \boldsymbol{\lambda}} = \mathscr{O}_{X}(\star f^{\boldsymbol{\lambda}}).
\]
\end{theorem}

As in the case of $f^{-1}$ one hopes there is an equivalent interpretation in terms of Bernstein--Sato ideals. However, unlike the univariate case ($f^{-1}$, Bernstein--Sato polynomials), the multivariate case ($f^{\boldsymbol{\lambda}}$, Bernstein--Sato ideals) and the associated $\mathscr{D}_{X}[s_{1}, \dots, s_{d}]$-modules are more nuanced and the precise connection remains unclear. These issues are central to the approach of \cite{MaisonobeFiltrationRelative} and \cite{ZeroLociI}, where methods and theory were developed for the multivariate setting.

Using the generically Cohen--Macaulay strategy developed in \cite{ZeroLociI} as well as other intricate techniques, in Theorem \ref{thm - bounding codim one components of BS zero loci} we are able to give a very good bound (i.e. sharp, see the formula for a generic arrangement in Theorem 3.23 of \cite{Bath3}) for the codimension one components of the zero locus of the Bernstein--Sato ideal. Here we consider arbitrary factorizations $F = (f_{1}, \dots, f_{r})$ of $f$.

\begin{theorem} \label{thm - intro, bounding codim one components of BS zero loci}
For $F = (f_{1}, \dots, f_{r})$ an arbitrary factorization of a reduced, central arrangement $f$, the codimension one components of the Bernstein--Sato ideal attached to $F$ has the following restriction:
\[
Z_{r-1}(B_{F,0}) \subseteq \bigcup_{\substack{E \in \mathscr{L}(\mathscr{A}) \\ E \text{ \normalfont dense}}} \bigcup_{v=0}^{Q_{E}} \left\{ \sum\limits_{\{1 \leq k \leq r \mid E \subseteq \Var(f_{k})\}} d_{E,k} s_{k} + \rank(E) + v = 0 \right\}
\]
where $d_{E,k}$ is as in Definition \ref{def - notation for factorization of arrangements along edges} and 
\[
Q_{E} = 
\begin{cases}
2 d_{E} - \rank(E) - \min\{2, \rank(E)\} \text{ for } F \text{ a factorization into linears}; \\
2 d_{E} - \rank(E) - \min\{2, \rank(E)\} \text{ for $F$ any factorization } \& \text{ } E = \{0\}; \\
2 d_{E} - \rank(E) - 1 \text{ for } F \text{ not a factorization into linears }\& \text{ } E \neq \{0\}.
\end{cases}
\]
In particular, when $F = (f)$ is the trivial factorization and $f$ is not smooth, the roots of the Bernstein--Sato polynomial are contained in $(-2 + 1/d, 0).$
\end{theorem}
\noindent To clarify notation: $Z_{r-1}(B_{F,0})$ denotes the codimension one components of the Bernstein--Sato ideal attached to $F$ at the origin, $d_{E}$ equals the number of hyperplanes containing the edge $E$, and $d_{E,k}$ equals the number of hyperplanes containing the edge $E$ that themselves are contained in $\Var(f_{k})$, cf. Definition \ref{def - notation for factorization of arrangements along edges}.

In the simpler case of Bernstein--Sato polynomials, Theorem \ref{thm - intro, bounding codim one components of BS zero loci} gives an entirely new proof of M. Saito's result \cite{SaitoArrangements} that the roots of the Bernstein--Sato polynomial of a non-smooth central arrangement lie in $(-2+1/d, 0)$.

\vspace{5mm}
Here is the paper's structure. Section 2 introduces paper-wide notation, develops tools for studying complexes of algebraic local cohomology of logarithmic forms (subsection 2.2, subsection 2.5), studying the analytic analogue of complexes of \v{C}ech cohomology of logarithmic forms (subsection 2.3), converting algebraic results to analytic ones (subsection 2.6), and bounding the Castelnuovo--Mumford regularity (subsection 2.4). Section 3 uses these results to: prove our four (un)Twisted Logarithmic Comparison Theorems (two analytic, two algebraic); prove that the theorems apply to all rank one local systems (Proposition \ref{prop - all local systems can be computed}); compute Example \ref{ex - deleted B3 arrangement}. Section 4 gives $\mathscr{D}_{X}$-module applications, including a bound on the codimension one components of the Bernstein--Sato ideal of an arbitrary factorization of our arrangement.

We thank Guillem Blanco, Nero Budur, Francisco Castro Jim{\'e}nez, Luis Narv{\'a}ez Macarro, and Robin van der Veer for combinations of helpful conversations, comments, and inspiring our interest in the subject. We are very grateful to Avi Steiner for detailed comments, including identifying a gap of justification in a previous version. And we sincerely appreciate Uli Walther's general support as well as his generosity in fielding an untold number of technical questions.

\section{Logarithmic Forms and Local Cohomology}

In the first subsection we set up necessary notation and background knowledge used throughout the paper. The rest of the section is devoted to studying the local cohomology of logarithmic forms in the algebraic setting as well as a sort of analytic analogue: the \v{C}ech cohomology of the logarithmic forms with respect to the cover $\{D(x_{i})\}$ of $X \setminus 0$. While we require knowledge of both contexts, the analytic analogue is vital. On the other hand, the techniques in the algebraic setting inspire those for the analytic one.

In subsection 2.2 we build up tools for studying complexes of such algebraic local cohomology modules with twisted differentials, focusing on establishing useful contracting homotopies; in subsection 2.3 we morally repeat the strategy in the analytic \v{C}ech case, again obtaining a useful contracting homotopy; in subsection 2.4 we bound the Castelnuovo--Mumford regularity of logarithmic forms in the algebraic case; in subsection 2.5 we use this bound to obtain an acyclicity criterion in the algebraic case; in subsection 2.6 we use a GAGA argument to get a similar acyclicity criterion in the analytic setting. The last result powers our proof of the analytic (un)Twisted Logarithmic Comparison Theorem. 

\subsection{Terminology and Background} \text{ }

We will use $f=f_{1} \cdots f_{d} \in R = \mathbb{C}[x_{1}, \dots, x_{n}]$ to denote a defining equation of a reduced hyperplane arrangement $\mathscr{A}$ of degree $d$ (i.e. $\deg(f) = d)$. We use $X$ to denote $\mathbb{C}^{n}$ and will, unless otherwise stated, consider $X$ as a Stein manifold with analytic structure sheaf $\mathscr{O}_{X}$. The complement $U = X \setminus \mathscr{A}$ of our arrangement includes into $X$ via $j: U \xhookrightarrow{} X$. As is standard, we denote the constant $\mathbb{C}$-valued rank one system on $U$ by $\mathbb{C}_{U}$.

Let $\mathscr{L}(\mathscr{A})$ be the intersection lattice of the hyperplane arrangement $\mathscr{A}$; its members are called \emph{edges} (sometimes flats in the literature). We say: the \emph{rank} of $\mathscr{A}$, denoted $\rank(\mathscr{A})$, is the rank of the lattice $\mathscr{L}(\mathscr{A})$; $\mathscr{A}$ is \emph{reduced} when $f$ is reduced; $\mathscr{A}$ is \emph{central} when $\{0\} \subseteq \cap_{E \in \mathscr{L}(\mathscr{A})} E$, i.e. $f$ is homogeneous; $\mathscr{A}$ is \emph{essential} when $\{0\} = \cap_{E \in \mathscr{L}(\mathscr{A})} E$; $\mathscr{A}$ is \emph{decomposable} when, after a possible coordinate change, there are polynomials $g$ and $h$ in disjoint, nonempty variable sets such that $f = gh$; an edge $E \in \mathscr{L}(\mathscr{A})$ is \emph{dense} when the subarrangement of all hyperplanes containing $E$ is indecomposable; the rank of $E \in \mathscr{L}(\mathscr{A})$ is the rank of the subarrangement of hyperplanes containing $E$. Alternatively, the rank of $A$ or $E$ is the codimension of the corresponding subspace in $X$.

We recall the types of de Rham complexes appearing in the Introduction.

\begin{define} \label{def - meromorphic and log de rham complex} Let $(\Omega_{X}^{\bullet}, d)$ be the canonical de Rham complex. The \emph{meromorphic de Rham complex} is 
\begin{equation*} 
(\Omega_{X}^{\bullet}(\star \mathscr{A}), d) = 0 \to \Omega_{X}^{0}(\star \mathscr{A}) \xrightarrow[]{d} \cdots \xrightarrow[]{d} \Omega_{X}^{n}(\star \mathscr{A}) \to 0, 
\end{equation*}
where $d$ is the exterior derivative and $\Omega_{X}^{j}(\star \mathscr{A})$ is the localization of $\Omega_{X}^{j}$ along $f$, i.e. $(\star \mathscr{A})$ means we have poles of arbitrary order along $\mathscr{A}$. The \emph{logarithmic de Rham complex} $(\Omega_{X}^{\bullet}(\log \mathscr{A}), d)$ is the subcomplex of the meromorphic de Rham complex characterized by 
\begin{equation} \label{eqn - def of log forms in terms of differentials}
\Omega_{X}^{j}(\log \mathscr{A}) = \{ \eta \in \Omega_{X}^{j}(\mathscr{A}) \mid d(\eta) \in \Omega_{X}^{j+1}(\mathscr{A}) \},
\end{equation}
that is, its the largest subcomplex where both $\eta$ and $d(\eta)$ have poles of order at most one along $\mathscr{A}$. Both are complexes of $\mathscr{O}_{X}$-modules with $\mathbb{C}$-linear differentials.

There are global algebraic analogues: the algebraic \emph{rational de Rham complex} $\Omega_{R}^{\bullet}(\star \mathscr{A})$; the algebraic \emph{logarithmic de Rham complex} $\Omega_{R}^{\bullet}(\log \mathscr{A})$. These are complexes of $R$-modules with $\mathbb{C}$-linear differentials and have entirely similar definitions: $\Omega_{R}^{j}(\star \mathscr{A}) = \Omega_{R}^{j}[f^{-1}]$; $\Omega_{R}^{j}(\log \mathscr{A})$ satisfies its version of \eqref{eqn - def of log forms in terms of differentials}.

There are natural inclusion of complexes:
\begin{align} \label{eqn - def inclusion log de Rham complex into meromorphic}
(\Omega_{X}^{\bullet}(\log \mathscr{A}), d) \xhookrightarrow{} (\Omega_{X}^{\bullet}(\star \mathscr{A}), d); \\
(\Omega_{R}^{\bullet}(\log \mathscr{A}), d) \xhookrightarrow{} (\Omega_{R}^{\bullet}(\star \mathscr{A}), d). \nonumber
\end{align}
We say $\mathscr{A}$ satisfies the \emph{analytic Logarithmic Comparison Theorem} when first inclusion \eqref{eqn - def inclusion log de Rham complex into meromorphic} is a quasi-isomorphism; it satisfies the \emph{algebraic Logarithmic Comparison Theorem} when the second is a quasi-isomorphism.
\end{define}

A dual object to $\Omega_{R}^{1}(\log \mathscr{A})$ are the \emph{logarithmic derivations}
\[
\Der_{R}(-\log \mathscr{A}) = \{\delta \in \Der_{R} \mid (\delta \bullet f) \in R \cdot f\}.
\]
These can be defined analytically as well. We sometimes replace ``$\log \mathscr{A}$'' with ``$\log f$'' when working with a defining equation $f$ of $\mathscr{A}$. Note that none of these constructions depend on the choice of $f$ (nor coordinate systems).

\begin{remark} \label{rmk - Kunneth formula for log diff forms}
Suppose $\mathscr{A}$ is a product (i.e. decomposable): $\mathscr{A} = \mathscr{B} \oplus \mathscr{C}$. This means the defining equation may be written as $f = g h$ where $g \in B = \mathbb{C}[x_{1}, \dots, x_{i}]$ cuts out $\mathscr{B}$ and $h \in C=\mathbb{C}[x_{i+1}, \dots, x_{n}]$ cuts out $\mathscr{C}$ (changing coordinates as necessary). Then the logarithmic differential forms satisfy a Kunneth formula, see, for example, Lemma 3.10 of \cite{CriticalPointsandResonanceVarieties}:
\[
\Omega_{R}^{k}(\log \mathscr{A}) \simeq \bigoplus_{i+j = k} \Omega_{B}^{i}(\log \mathscr{B}) \otimes_{\mathbb{C}} \Omega_{C}^{j}(\log \mathscr{C}).
\]
The same type of Kunneth formula holds on the sheaf level and for $\Der_{R}(-\log \mathscr{A})$, cf. Lemma 1.2 of \cite{SaitoDegenerationOfPoleOrderSpectralSequence4Variables}.
\end{remark}

Recall that the \emph{Orlik--Solomon} algebra $A^{\bullet}(\mathscr{A})$ can be identified by the $\mathbb{C}$-algebra generated by the differential forms $\{d f_{k} / f_{k}\}$, one for each hyperplane $f_{k}$. Under this characterization it is called the \emph{Brieskorn algebra}. Brieskorn's Theorem \cite{BrieskornTheorem} combined with Grothendieck's Comparison Theorem \cite{GrothendieckOnTheDeRham} gives a quasi-isomorphism
\[
A^{\bullet}(\mathscr{A}) \xhookrightarrow{\qi} (\Omega_{R}^{\bullet}(\star \mathscr{A}), d)  \quad (\simeq  H^{\bullet}(U, \mathbb{C}_{U})),
\]
where ``$(\simeq  H^{\bullet}(U, \mathbb{C}_{U}))$'' means the cohomology objects are isomorphic, cf. Section 5.4 \cite{OrlikTeraoBook}.

We turn to twisted de Rham complexes and nontrivial rank one local systems.

\begin{define} \label{def - twisted meromoprhic and log de rham}
For $\lambda_{1}, \dots, \lambda_{d} \in \mathbb{C}$, consider the logarithmic one form
\[
\omega = \sum_{k} \lambda_{k} \frac{d f_{k}}{f_{k}} = \sum_{k} \lambda_{k} d \log f_{k}.
\]
This induces a \emph{twisted meromorphic de Rham complex}
\[
(\Omega_{X}^{\bullet}(\star \mathscr{A}), \nabla_{\omega}) = 0 \to \Omega_{X}^{0}(\star \mathscr{A}) \xrightarrow[]{\nabla_{\omega}} \Omega_{X}^{1}(\star \mathscr{A}) \xrightarrow[]{\nabla_{\omega}} \cdots \xrightarrow[]{\nabla_{\omega}} \Omega_{X}^{n}(\star \mathscr{A}) \to 0.
\]
where, for a meromorphic $j$-form $\eta$, 
\[
\nabla_{\omega}(\eta) = d(\eta) + \omega_{\lambda} \wedge \eta.
\]
Since the logarithmic de Rham complex is closed under exterior products, $\nabla_{\omega}$ also induces the \emph{analytic twisted logarithmic de Rham complex}
\[
(\Omega_{X}^{\bullet}(\log \mathscr{A}), \nabla_{\omega}) = 0 \to \Omega_{X}^{0}(\log \mathscr{A}) \xrightarrow[]{\nabla_{\omega}} \Omega_{X}^{1}(\log \mathscr{A}) \xrightarrow[]{\nabla_{\omega}} \cdots \xrightarrow[]{\nabla_{\omega}} \Omega_{X}^{n}(\log \mathscr{A}) \to 0.
\]
Because $\omega \in \Omega_{R}^{1}(\log \mathscr{A})$, we have global algebraic analogues: the \emph{twisted rational de Rham complex} $(\Omega_{R}^{\bullet}(\star \mathscr{A}), \nabla_{\omega})$; the \emph{algebraic twisted logarithmic de Rham complex} $(\Omega_{R}^{\bullet}(\log \mathscr{A}), \nabla_{\omega}).$ We sometimes drop ``analytic'' or ``algebraic'' when the context is clear.

Analogizing the untwisted case, we say $\mathscr{A}$ satisfies the \emph{analytic Twisted Logarithmic Comparison Theorem} when the natural inclusion of complexes is a quasi-isomorphism, that is when:
\[
(\Omega_{X}^{\bullet}(\log \mathscr{A}), \nabla_{\omega}) \xhookrightarrow{\qi} (\Omega_{X}^{\bullet}(\star \mathscr{A}), \nabla_{\omega}).
\]
$\mathscr{A}$ satisfies the \emph{algebraic Twisted Logarithmic Comparison Theorem} when the corresponding inclusion of algebraic complexes is a quasi-isomorphism.
\end{define}

As the meromorphic de Rham complex relates to cohomology with constant coefficients, the twisted meromorphic de Rham complex relates to cohomology with a nontrivial rank one local system. 

\begin{define}
Recall $U = X \setminus \mathscr{A}$ and $j: U \xhookrightarrow{} X$ is the inclusion. It is well known that the $\mathbb{C}$-valued rank one local systems on $U$ are in one to one correspondence with the torus points
\[
\Hom(\pi_{1}(U, \mathbb{Z}), \mathbb{C}^{\star}) = (\mathbb{C}^{\star})^{d}.
\]
For a point $\boldsymbol{\beta}$ in the torus $(\mathbb{C}^{\star})^{r}$ we define $\localSystem_{\boldsymbol{\beta}}$ to be the corresponding rank one local system on $U$. We recall
\[
\derivedR j_{\star} \localSystem_{\boldsymbol{\beta}}
\]
is an object in the derived category of bounded $\mathbb{C}_{X}$-complexes. In fact, it is a perverse sheaf. 

Now let $\Exp : \mathbb{C}^{d} \to (\mathbb{C}^{\star})^{d}$ be the exponential map:
\[
\boldsymbol{\lambda} = (\lambda_{1}, \dots, \lambda_{d}) \mapsto \Exp(\boldsymbol{\lambda}) = (e^{2 \pi i \lambda_{1}}, \dots, e^{2 \pi i \lambda_{d}}).
\]
Then every point $\boldsymbol{\lambda} \in \mathbb{C}^{d}$, along with its integer translates, corresponds to the rank one local system on $U$ we've called $\localSystem_{\Exp(\boldsymbol{\lambda})}.$
\end{define}

Similar to the untwisted case, by a combination of Deligne's and Grothendieck's algebraic de Rham Theorems we have isomorphisms in cohomology:
\[
H^{\bullet}(\Omega_{R}^{\bullet}(\star \mathscr{A}), \nabla_{\omega}) \simeq H^{\bullet}(U, \localSystem_{\Exp(\boldsymbol{\lambda})}).
\]
See \cite{OrlikHypergeometricIntegralsAndArrangements} for a discussion of this in the arrangement case. In the analytic setting we have a well-known identification in the derived category:
\[
(\Omega_{X}(\star f), \nabla_{\omega}) = \derivedR j_{\star} \localSystem_{\Exp(\boldsymbol{\lambda})}.
\]
We give details on this in Proposition \ref{prop - twisted meromorphic de rham quasi-iso to local system}. So the analytic twisted meromorphic de Rham complex has all the data of the derived direct image of the local system in question. An analytic/algebraic Twisted Logarithmic Comparison Theorem guarantees this data is determined by only logarithmic content.

\begin{convention} \label{convention - grading}
Unless otherwise stated, $R = \mathbb{C}[x_{1}, \dots, x_{n}]$ and is canonically graded with $x_{i}$ given weight one. Denote the irrelevant ideal $R \cdot (x_{1}, \dots, x_{n})$ of $R$ by $\mathfrak{m}$. In general we allow $R$-modules to be $\mathbb{Z}$-graded. We grade $\Omega_{R}^{j}$ by giving $dx_{i}$ weight one, we grade $R[x_{I}^{-1}]$ (the localization of $R$ at each $x_{i}$ for $i \in I \subseteq [n]$) by giving $x_{i}^{-1}$ weight one; for a graded $R$-module $M$ we grade $M[x_{I}^{-1}]$ similarly. We also grade $\Omega_{R}^{j}(\star \mathscr{A})$ and $\Omega_{R}^{j}(\log \mathscr{A})$ by giving $dx_{i}$ weight one and $\frac{1}{f}$ weight $- \deg(f)$; we grade $\Der_{R}(-\log f)$ by giving $\partial_{i}$ weight $-1$.
\end{convention}

\subsection{Complexes of Local Cohomology Modules of Logarithmic Forms} \text{ }

In this section we work algebraically, developing some technical tools for studying both twisted logarithmic de Rham complexes and twisted complexes of local cohomology of logarithmic forms. By the latter, we mean a complex of local cohomology modules induced by the twisted logarithmic de Rham complex. Much of the time is spent showing this latter complex is well-defined and the correct notion of a twisted Lie derivative obeys a nice formula. The primary objective is Proposition \ref{prop - homogeneous subcomplex quasi iso}, which reduces computing cohomology of these complexes to computing the cohomology of a particular subcomplex of finite dimensional $\mathbb{C}$-vector spaces.

In subsection 2.3, we recycle these ideas in an analytic setting.

\begin{convention}
In this subsection $\mathscr{A}$ is always assumed to be central with a homogeneous defining equation $f$.
\end{convention}

First let us recall basic facts about local cohomology modules:

\begin{define} \label{def - local cohomology}
Let $M$ be a finite $R$-module, $Q \subseteq R$ an ideal. Consider the functor
\[
\Gamma_{Q}^{0}(M) = \{m \in M \mid Q^{\ell} m = 0 \text{ for some } \ell \in \mathbb{N} \}.
\]
Then $H_{Q}^{t}(M)$, the \emph{$t^{\text{th}}$ local cohomology module of $M$ with respect to $Q$}, is the $t^{\text{th}}$ derived functor of $\Gamma_{Q}(M)$.
\end{define}

\begin{remark} \text{ }
\begin{enumerate}[label=(\alph*)]
    \item We usually work with $R$ canonically graded, $M$ a $\mathbb{Z}$ or $\mathbb{N}$-graded module, $\mathfrak{m} = R \cdot (x_{1} , \dots, x_{n})$, and the local cohomology modules $H_{\mathfrak{m}}^{t}(M)$. In this graded setting there are well known (non)vanishing theorems:
        \begin{enumerate}[label=(\arabic*)]
            \item $H_{\mathfrak{m}}^{t}(M) = 0 \text{ for } t < \depth(M) \text{ or } t > \dim M$;
             \item $\text{If } t = \depth(M) \text{ or } t = \dim M \text{, then } H_{\mathfrak{m}}^{t}(M) \neq 0$.
        \end{enumerate}
    Consequently, $H_{\mathfrak{m}}^{t}(R)$ is nonzero only when $t = n$.
    \item There is a practical way to compute local cohomology using the algebraic \v{C}ech complex. Let $q_{1}, \dots, q_{p}$ generate the ideal $Q \subseteq R$. The algebraic \v{C}ech complex is
    \[
    C(q_{1}, \dots, q_{p}; M) = 0 \to M \to \cdots \to \bigoplus_{\substack{I \subseteq [p] \\ \mid I \mid = t}} M[q_{I}^{-1}] \to \cdots \to M[q_{[p]}^{-1}] \to 0
    \]
    where the differential $\epsilon: \oplus_{\substack{I \subseteq [p] \\ \mid I \mid = t}} M[q_{I}^{-1}] \to \oplus_{\substack{I \subseteq [p] \\ \mid I \mid = t+1}} M[q_{I}^{-1}]$ takes
    \[
    M[q_{I}^{-1}] \ni m_{I} \mapsto \sum_{\ell \notin I} (-1)^{\sigma_{I}(\ell)} m_{I \cup \{\ell\}} \in \sum_{\ell \not in I} M[q_{I \cup \{\ell\}}^{-1}].
    \]
    By $m_{I \cup \{\ell\}}$ we mean the image of $m_{I}$ in the localization $M[q_{I}^{-1}] \to M[q_{I \cup \{\ell\}}^{-1}]$; by $\sigma_{I}(\ell)$ we mean the number of elements in $I$ less than $\ell$. There is natural identification
    \[
    H_{Q}^{t}(M) = H^{t}(C(q_{1}, \dots, q_{p}; M)).
    \]
    \item If $Q$ is homogeneous, $M$ is $\mathbb{Z}$-graded, and we grade $M[x_{I}^{-1}]$ naturally by granting $x_{i}^{-1}$ weight one, then $H_{Q}^{t}(M)$ is naturally $\mathbb{Z}$-graded by the above.
\end{enumerate}
\end{remark}

We want to construct a very large commutative diagram of horizontal and vertical complexes, visualized in $\mathbb{Z}^{2}$, that we eventually name $N^{\bullet, \bullet}$. This diagram will let us understand the interplay between local cohomology modules of logarithmic differential forms and the twisted differential $\nabla_{\omega}$.

Picture $(\Omega_{R}^{\bullet}(\log \mathscr{A}), \nabla_{\omega})$ as a column with differentials pointed upwards. In fact, put $(\Omega_{R}^{\bullet}(\log \mathscr{A}), \nabla_{\omega})$ on the $y$-axis. For each $\Omega_{R}^{j}(\log \mathscr{A})$, consider the \v{C}ech complex $C(x_{1}, \dots, x_{n}; \Omega_{R}^{j}(\log \mathscr{A}))$ that computes the local cohomology modules $H_{\mathfrak{m}}^{\bullet}(\Omega_{R}^{j}(\log \mathscr{A}))$. Position each $C(x_{1}, \dots, x_{n}; \Omega_{R}^{j}(\log \mathscr{A}))$ as a horizontal complex, with differentials oriented rightwards, located on the line $y = j$. Shift the \v{C}ech complex so that its $\Omega_{R}^{j}(\log \mathscr{A})$ term overlaps with the corresponding object of $(\Omega_{R}^{\bullet}(\log \mathscr{A}), \nabla_{\omega})$. 

The map $\nabla_{\omega}$ naturally lifts to a map between the objects of the \v{C}ech complex. We want to verify this induces a chain map of \v{C}ech complexes, that is, we must confirm the squares
\begin{equation} \label{eqn - tikzcd diagram, comm square of local cohomology log de rham complex}
\begin{tikzcd}
\bigoplus\limits_{\mid I \mid = t} \Omega_{R}^{j+1}(\log \mathscr{A})[x_{I}^{-1}] \rar{\epsilon}
    & \bigoplus\limits_{\mid I^{\prime} \mid = t+1} \Omega_{R}^{j+1}(\log \mathscr{A})[x_{I^{\prime}}^{-1}] \\
\bigoplus\limits_{\mid I \mid = t} \Omega_{R}^{j}(\log \mathscr{A})[x_{I}^{-1}] \uar{\nabla_{\omega}} \rar{\epsilon}
    & \bigoplus\limits_{\mid I^{\prime} \mid = t+1} \Omega_{R}^{j}(\log \mathscr{A})[x_{I^{\prime}}^{-1}] \uar{\nabla_{\omega}}
\end{tikzcd}
\end{equation}
commute. Here, for $\eta_{I} \in \Omega_{R}^{j}(\log \mathscr{A})$ and $Q_{I} \in R[x_{I}^{-1}]$, the horiziontal map is defined by
\begin{equation} \label{eqn - Cech complex map on j log diff forms}
\epsilon(Q_{I} \eta_{I}) = \sum_{\ell \notin I} (-1)^{\sigma_{I}(\ell)} Q_{I \cup \{\ell\}} \eta_{I} \in \bigoplus\limits_{ \ell \notin I } \Omega_{R}^{j}(\log \mathscr{A})[x_{I \cup \{\ell\}}^{-1}]
\end{equation}
where $Q_{I \cup \{\ell\}}\eta_{I}$ denotes the image of $Q_{I}\eta_{I} \in  \Omega_{R}^{j}(\log \mathscr{A})[x_{I}^{-1}]$ in the subsequent localization $\Omega_{R}^{j}(\log \mathscr{A})[x_{I \cup \{\ell\}}^{-1}]$ and $\sigma_{I}(\ell)$ is the number of elements in $I$ less than $\ell$. To suppress potential confusion, the vertical $\nabla_{\omega}$ are defined by 
\begin{equation} \label{eqn - nabla omega definition on Cech complex}
\nabla_{\omega}(Q_{I} \eta_{I}) = d(Q_{I} \eta_{I}) + \omega_{\lambda} \wedge Q_{I} \eta_{I} = d(Q_{I}) \wedge \eta_{I} + Q_{I} \wedge d(\eta_{I}) + Q_{I} \omega_{\lambda} \wedge \eta_{I}. 
\end{equation}
Note that $d(Q_{I}) \wedge \eta_{I} \in \Omega_{R}^{j+1}(\log \mathscr{A})[x_{I}^{-1}]$ since $d(Q_{I}) \in \Omega_{R}^{1}[x_{I}^{-1}]$ and $\Omega_{R}^{1} \wedge \Omega_{R}^{j}(\log \mathscr{A}) \subseteq \Omega_{R}^{j+1}(\log \mathscr{A})$, which together entail that $d(Q_{I}) \wedge \eta_{I} \in \Omega_{R}^{1}[x_{I}^{-1}] \wedge \Omega_{R}^{j}(\log \mathscr{A}) \subseteq \Omega_{R}^{j+1}(\log \mathscr{A})[x_{I}^{-1}]$. 

That \eqref{eqn - tikzcd diagram, comm square of local cohomology log de rham complex} commutes follows from the definitions:
\begin{align*}
\nabla_{\omega}(\epsilon(Q_{I} \eta_{I}) &= \nabla_{\omega}(\sum_{\ell \notin I}  (-1)^{\sigma_{I}(\ell)} Q_{I \cup \{\ell\}} \eta_{I}) \\
    &= \sum_{\ell \notin I}(-1)^{\sigma_{I}(\ell)}( (d(Q_{I \cup \{\ell\}}) \wedge \eta_{I} + Q_{I\cup \{\ell\}} \wedge d(\eta_{I}) + Q_{I \cup \{\ell\}} \omega \wedge \eta_{I}); \\
\epsilon(\nabla_{\omega}(Q_{I} \eta_{I}) &= \epsilon( d(Q_{I}) \wedge \eta_{I} + Q_{I} \wedge d(\eta_{I}) + Q_{I} \omega \wedge \eta_{I}) \\
    &= \epsilon(d(Q_{I}) \wedge \eta_{I}) + \sum_{\ell \notin I} (-1)^{\sigma_{I}(\ell)} (Q_{I \cup \{\ell\}} \wedge d(\eta_{I}) + Q_{I \cup \{\ell\}} \omega \wedge \eta_{I}).
\end{align*}
We are done once we confirm $\sum_{\ell \notin I} (-1)^{\sigma_{I}(\ell)} d(Q_{I \cup \{\ell\}}) \wedge \eta_{I} = \epsilon(d(Q_{I}) \wedge \eta_{I})$. Let $Q_{I} = \frac{a}{x_{I}^{p}}$. Then 
\begin{align*}
\epsilon(d(Q_{I}) \wedge \eta_{I}) &= \epsilon( (\sum_{1 \leq i \leq n} \frac{(\partial_{i} \bullet a) dx_{i}}{x_{I}^{p}} - \frac{(\partial_{i} \bullet x_{I}^{p}) a dx_{i}}{x_{I}^{2p}}) \wedge \eta_{I}) \\
    & = \epsilon( \sum_{1 \leq i \leq n} (\frac{(\partial_{i} \bullet a)} {x_{I}^{p}} - \frac{(\partial_{i} \bullet x_{I}^{p}) a} {x_{I}^{2p}}) dx_{i}  \wedge \eta_{I} \\
    &= \sum_{\ell \notin I} (-1)^{\sigma_{I}(\ell)} d(Q_{I \cup \{\ell\}}) \wedge \eta_{I}
\end{align*}
since the image of $\partial_{i} \bullet Q_{I}$ in $R[x_{I \cup \{\ell\}}^{-1}]$ agrees with $\partial_{i}$ applied to $Q_{I\cup \{\ell\}} \in R[x_{I \cup \{\ell\}}^{-1}]$. 

We summarize our construction of $N^{\bullet, \bullet}$:

\begin{proposition} \label{prop - definition of N diagram, Cech twisted log de rham lattice}
There is a commutative diagram $N^{\bullet, \bullet}$, positioned in $\mathbb{Z}^{2}$, where each row and column is a chain complex, and each square looks like: 
\begin{equation} \label{eqn - tikzcd diagram, comm square of local cohomology log de rham complex with names}
\begin{tikzcd}
N^{j+1, t} \rar[equal]
    &  \bigoplus\limits_{\mid I \mid = t} \Omega_{R}^{j+1}(\log \mathscr{A})[x_{I}^{-1}] \rar{\epsilon}
    & \bigoplus\limits_{\mid I^{\prime} \mid = t+1} \Omega_{R}^{j+1}(\log \mathscr{A})[x_{I^{\prime}}^{-1}] \rar[equal]
        & N^{j+1, t+1} \\
N^{j, t} \rar[equal]
    & \bigoplus\limits_{\mid I \mid = t} \Omega_{R}^{j}(\log \mathscr{A})[x_{I}^{-1}] \uar{\nabla_{\omega}} \rar{\epsilon}
    & \bigoplus\limits_{\mid I^{\prime} \mid = t+1} \Omega_{R}^{j}(\log \mathscr{A})[x_{I^{\prime}}^{-1}] \uar{\nabla_{\omega}} \rar[equal]
        & N^{j, t+1}.
\end{tikzcd}
\end{equation}
The horizontal maps are the natural $R$-linear maps induced by the \v{C}ech complex $C(x_{1}, \dots, x_{n}; \Omega_{R}^{j}(\log \mathscr{A}))$ (see \eqref{eqn - Cech complex map on j log diff forms}), the vertical maps are the natural $\mathbb{C}$-linear maps induced by $(\Omega_{R}^{\bullet}, \nabla_{\omega})$ (see \eqref{eqn - nabla omega definition on Cech complex}),  and each square commutes (see the computations following \eqref{eqn - tikzcd diagram, comm square of local cohomology log de rham complex}). 
\end{proposition}

We now construct a well behaved map $N^{j+1, t} \to N^{j,t}$ that is $R$-linear.

\begin{define} \label{def - defining contraction along Euler in many contexts}
Let $E = \sum_{i} x_{i} \partial_{i} \in \Der_{X}(-\log \mathscr{A})$ be the Euler operator and recall the classical \emph{contraction} $\iota_{E}: \Omega_{R}^{j+1} \to \Omega_{R}^{j}$ along $E$. Since $E$ is logarithmic, this induces a contraction on the logarithmic differential forms (cf. 1.6 \cite{SaitoLogarithmicForms})
\[
\iota_{E} : \Omega_{R}^{j+1}(\log \mathscr{A}) \to \Omega_{R}^{j}(\log \mathscr{A})
\]
as well as a contraction on the meromorphic differential forms
\[
\iota_{E} : \Omega_{R}^{j+1}(\star \mathscr{A}) \to \Omega_{R}^{j}(\star \mathscr{A}).
\]
It also induces a contraction map
\[
\iota_{E}: \bigoplus\limits_{\mid I \mid = t} \Omega_{R}^{j+1}(\log \mathscr{A})[x_{I}^{-1}] \to \bigoplus\limits_{\mid I \mid = t} \Omega_{R}^{j}(\log \mathscr{A})[x_{I}^{-1}]
\]
given by, for $\eta_{I} \in \Omega_{R}^{j+1}(\log \mathscr{A})$,
\[
\iota_{E}(Q_{I} \eta_{I}) = Q_{I} \iota_{E}(\eta_{I}).
\]
\end{define}

\begin{proposition} \label{prop - comm square of N diagram with contraction}
Let $E$ be the Euler operator and let $N^{\bullet, \bullet}$ as in Proposition \ref{prop - definition of N diagram, Cech twisted log de rham lattice}. Then the diagram
\[
\begin{tikzcd}
N^{j+1, t} \rar{\epsilon} \dar{\iota_{E}}
    & N^{j+1, t+1} \dar{\iota_{E}} \\
N^{j, t} \rar{\epsilon}
    & N^{j, t+1}
\end{tikzcd}
\]
commutes.
\end{proposition}

\begin{proof}
Let $Q_{I} \in R[x_{I}^{-1}]$ and $\eta_{I} \in \Omega_{R}^{j+1}(\log \mathscr{A})$ so that $Q_{I} \eta_{I} \in \Omega_{R}^{j+1}(\log \mathscr{A})[x_{I}^{-1}]$. Then
\begin{align*}
\epsilon(\iota_{E}(Q_{I} \eta_{I})) = \epsilon(Q_{I}( \iota_{E}(\eta_{I})))     &= \sum_{\ell \notin I} (-1)^{\sigma_{I}(\ell)} Q_{I \cup \{\ell\}} \iota_{E}(\eta_{I}) \\
    &= \iota_{E}(\sum_{\ell \notin I} (-1)^{\sigma_{I}(\ell)} Q_{I \cup \{\ell\}} \eta_{I}) = \iota_{E}(\epsilon (Q_{I}\eta_{I})). 
\end{align*}
\end{proof}

We now begin to reap some fruits of our mostly formal labors. From our first set of commutative diagrams we have constructed a complex of $\mathbb{C}$-modules; from our second set, a complex of $R$-modules:

\begin{proposition} \label{prop- contraction and twisted log de rham on local cohomology well defined}
The complex
\begin{align} \label{eqn - twisted log de rham complex on local cohomology}
(H_{\mathfrak{m}}^{t}(\Omega_{R}^{\bullet}(\log \mathscr{A})), \nabla_{\omega}) = 0 \to H_{\mathfrak{m}}^{t}(\Omega_{R}^{0}(\log \mathscr{A})) &\xrightarrow[]{\nabla_{\omega}} H_{\mathfrak{m}}^{t}(\Omega_{R}^{1}(\log \mathscr{A})) \xrightarrow[]{\nabla_{\omega}} \cdots \\
    &\xrightarrow[]{\nabla_{\omega}} H_{\mathfrak{m}}^{t}(\Omega_{R}^{n}(\log \mathscr{A}) \to 0 \nonumber
\end{align}
with $\mathbb{C}$-linear differential $\nabla_{\omega}$, as described in \eqref{eqn - nabla omega definition on Cech complex}, is well defined. Additionally, for $E$ the Euler operator, the complex
\begin{align} \label{eqn - contraction along Euler on local cohomology}
(H_{\mathfrak{m}}^{t}(\Omega_{R}^{\bullet}(\log \mathscr{A})), \iota_{E}) = 0 \to H_{\mathfrak{m}}^{t}(\Omega_{R}^{n}(\log \mathscr{A})) &\xrightarrow[]{\iota_{E}} H_{\mathfrak{m}}^{t}(\Omega_{R}^{n-1}(\log \mathscr{A})) \xrightarrow[]{\iota_{E}} \cdots \\
    &\xrightarrow[]{\iota_{E}} H_{\mathfrak{m}}^{t}(\Omega_{R}^{0}(\log \mathscr{A}) \to 0 \nonumber
\end{align}
with $R$-linear differential induced by $\iota_{E}$, as in Definition \ref{def - defining contraction along Euler in many contexts}, is well-defined.
\end{proposition}

\begin{proof}
That the squares \eqref{eqn - tikzcd diagram, comm square of local cohomology log de rham complex with names} from Proposition \ref{prop - definition of N diagram, Cech twisted log de rham lattice} commute amounts to saying $\nabla_{\omega}$ gives a chain map from $N^{j, \bullet} \to N^{j+1, \bullet}$ for all $j$. So the induced maps on cohomology 
\[
H_{\mathfrak{m}}^{t}(\Omega_{R}^{j}(\log \mathscr{A})) \xrightarrow[]{\nabla_{\omega}} H_{\mathfrak{m}}^{t}(\Omega_{R}^{j+1}(\log \mathscr{A}))
\]
are well-defined. Since $\nabla_{\omega}^{2} = 0$, certainly $(H_{\mathfrak{m}}^{t}(\Omega_{R}^{\bullet}(\log \mathscr{A})), \nabla_{\omega})$ is a complex.

The case of $(H_{\mathfrak{m}}^{t}(\Omega_{R}^{\bullet}(\log \mathscr{A})), \iota_{E})$ is similar, now using Proposition \ref{prop - comm square of N diagram with contraction}.
\end{proof}

Thus we can discuss the cohomology of the twisted differential $\nabla_{\omega}$ on local cohomology modules of logarithmic forms. Grade $R$ and $\Omega_{R}^{\bullet}(\log \mathscr{A})$ by the edict of the the Euler derivation $\sum x_{i} \partial_{i}$; that is, give $x$ and $dx_{i}$ weight one, $\frac{1}{f}$ weight $- \deg(f)$. Local cohomology $H_{\mathfrak{m}}^{t}(\Omega_{R}^{j}(\log \mathscr{A}))$ inherits a natural $\mathbb{Z}$-grading by making $x^{-1}$ have weight $-1$. Exterior differentiation preserves degree as does the exterior product with $\omega$ (since $df_{k}/f_{k}$ has weight zero). Therefore our complexes decompose into their homogeneous subcomplexes:
\[
(\Omega_{R}^{\bullet}(\log \mathscr{A}), \nabla_{\omega}) = \bigoplus_{q \in \mathbb{Z}} \enspace (\Omega_{R}^{\bullet}(\log \mathscr{A})_{q}, \nabla_{\omega})
\]
\[
(H_{\mathfrak{m}}^{t}(\Omega_{R}^{\bullet}(\log \mathscr{A})), \nabla_{\omega}) = \bigoplus_{q \in \mathbb{Z}} (H_{\mathfrak{m}}^{t}(\Omega_{R}^{\bullet}(\log \mathscr{A}))_{q}, \nabla_{\omega}).
\]
Clearly cohomology is well-behaved with respect to this direct sum decomposition.

Following Lemma 2.1 of \cite{CohomologyComplementFree} we can extend a nice formula involving Lie derivative and homogeneous elements to the twisted and local cohomological cases:

\begin{lemma} \label{lemma - twisted lie derivative computation on homogeneous}
Suppose that $E$ is the Euler operator and that $R$, $\Omega_{R}^{j}(\log \mathscr{A})$, and $H_{\mathfrak{m}}^{t}(\log \mathscr{A})$ are graded accordingly, i.e. with $x_{i}$ and $dx_{i}$ given weight one. The twisted Lie derivative with respect to $E$ on $\Omega_{R}^{\bullet}(\log \mathscr{A})$
\begin{equation} \label{eqn - def of twisted lie derivative, chi}
L_{\omega}(\eta) = \nabla_{\omega}(\iota_{E} (\eta)) + \iota_{E}((\nabla_{\omega})(\eta))
\end{equation}
satisfies, for $\eta$ a homogeneous logarithmic $j$-form,
\begin{equation} \label{eqn - twisted lie derivative on log forms, homogeneous}
L_{\omega}(\eta) = (\deg (\eta) + \iota_{E}(\omega))\eta = (\deg (\eta) + \sum \lambda_{k}) \eta.
\end{equation}
where $\omega = \sum \lambda_{k} d \log f_{k}.$ The induced twisted Lie derivative on a homogeneous class $[\zeta] \in H_{\mathfrak{m}}^{t}(\Omega_{R}^{j}(\log \mathscr{A}))$ satisfies
\begin{equation} \label{eqn - twisted lie derivative on local cohomology of log forms, homogeneous}
L_{\omega}([\zeta]) = (\deg ([\zeta]) + \iota_{E}(\omega))[\zeta] = (\deg ([\zeta]) + \sum \lambda_{k})[\zeta].
\end{equation}
\end{lemma}

\begin{proof}
We first do the case of logarithmic forms. The untwisted statement is due to Naruki according to Lemma 2.1 of \cite{CohomologyComplementFree} and is a straightforward computation. Since contraction and exterior differentiation are additive, we need only to confirm that
\begin{align*}
\omega \wedge \iota_{E}(\eta) + \iota_{E}(\omega \wedge \eta) &= \omega \wedge \iota_{E}(\eta) + \iota_{E}(\omega) \wedge \eta - \omega \wedge \iota_{E}(\eta)  \\
    &= \iota_{E}(\omega) \eta 
    = (\sum  \lambda_{k} \iota_{E}(d \log f_{k}) )\eta = (\sum \lambda_{k}) \eta.
\end{align*}
This gives \eqref{eqn - twisted lie derivative on log forms, homogeneous}. As for \eqref{eqn - twisted lie derivative on local cohomology of log forms, homogeneous}, first note that we have confirmed $\iota_{E}$ and $\nabla_{\omega}$ are well-defined maps on local cohomology and are defined by taking representatives. So the second formula follows from the first because: Naruki's computation for the untwisted case extends to $\Omega_{R}^{j}(\log \mathscr{A})[x_{I}^{-1}]$; we can compute the twisted Lie derivative on a homogeneous representative of our homogeneous class of $H_{\mathfrak{m}}^{t}(\Omega_{R}^{j}(\log \mathscr{A}))$.
\end{proof}

\begin{convention}
A quick computation verifies $\iota_{E}(\omega) = \lambda_{1} + \cdots + \lambda_{d}$. We use $\iota_{E}(\omega)$ throughout the paper, instead of writing this sum.
\end{convention}

Here is the workman's result of this subsection: it reduces computing cohomology of these twisted complexes to a finite dimensional linear algebra problem. Indeed, only a particular homogeneous subcomplex needs to be considered. Because $R$ is $\mathbb{N}$-graded, this homogeneous subcomplex is a finite dimensional $\mathbb{C}$-module. Again, this lemma tracks with ideas from \cite{CohomologyComplementFree}, but we have extended the result to the both the twisted case and to the case of local cohomology.

\begin{proposition} \label{prop - homogeneous subcomplex quasi iso}
Suppose that $E$ is the Euler operator and that $R$, $\Omega_{R}^{j}(\log \mathscr{A})$, $H^{i}_{\mathfrak{m}}(\Omega_{R}^{j}(\log \mathscr{A}))$ are graded accordingly, i.e. with $x_{i}$ and $dx_{i}$ given weight one. Then $(\Omega_{R}^{\bullet}(\log \mathscr{A}), \nabla_{\omega})$ and $(H_{\mathfrak{m}}^{t}(\Omega_{R}^{\bullet}(\log \mathscr{A})), \nabla_{\omega})$ have homogeneous differentials and decompose as direct sum of their homogeneous subcomplexes. Moreover, the following homogeneous subcomplex inclusions are quasi-isomorphisms:
\[
(\Omega_{R}^{\bullet}(\log \mathscr{A})_{ - \iota_{E}(\omega)}, \nabla_{\omega}) \xhookrightarrow{\text{q.i}} (\Omega_{R}^{\bullet}(\log \mathscr{A}), \nabla_{\omega});
\]
\[
(H_{\mathfrak{m}}^{t}(\Omega_{R}^{\bullet}(\log \mathscr{A}))_{ - \iota_{E}(\omega)}, \nabla_{\omega}) \xhookrightarrow{\text{q.i}} (H_{\mathfrak{m}}^{t}(\Omega_{R}^{\bullet}(\log \mathscr{A})), \nabla_{\omega}).
\]
(Here $(-)_{-\iota_{E}(\omega)}$ denotes the $-\iota_{E}(\omega)$-homogeneous subspace of $(-)$.)
\end{proposition}

\begin{proof}
Clearly $d$ is degree preserving map. Since $\omega = \sum \lambda_{k} \frac{df_{k}}{f_{k}}$ is homogeneous of grade zero, $\nabla_{\omega}$ is also a degree preserving map. That all these complexes decompose into direct sums of their homogeneous parts follows. 

As for the quasi-isomorphisms, we first do the case of the logarithmic de Rham complex. Thanks to the direct sum decomposition (which certainly extends to cohomology) it is enough to prove that, for $q \neq - \iota_{E}(\omega)$, the degree $q$ subcomplex $(\Omega_{R}^{\bullet}(\log \mathscr{A})_{q}, \nabla_{\omega})$ is acyclic. Thanks to \eqref{eqn - twisted lie derivative on log forms, homogeneous} from Lemma \ref{lemma - twisted lie derivative computation on homogeneous}, and the fact $\iota_{E}$ is itself degree preserving, $\iota_{E}$ defines a chain homotopy between 
\[
(\Omega_{R}^{\bullet}(\log \mathscr{A})_{q}, \nabla_{\omega}) \xrightarrow{(q + \iota_{E}(\omega)) \cdot} (\Omega_{R}^{\bullet}(\log \mathscr{A})_{q}, \nabla_{\omega})
\]
and the zero map
\[
(\Omega_{R}^{\bullet}(\log \mathscr{A})_{q}, \nabla_{\omega}) \xrightarrow[]{0 \cdot} (\Omega_{R}^{\bullet}(\log \mathscr{A})_{q}, \nabla_{\omega}).
\]
Since $q + \iota_{E}(\omega) \neq 0$, the first map is a quasi-isomorphism and so $(\Omega_{R}^{\bullet}(\log \mathscr{A})_{q}, \nabla_{\omega})$ is acyclic. The result for the logarithmic de Rham complex follows.

As for the local cohomology complexes, the same argument works: $\iota_{E}$ defines a chain homotopy between similar maps, now using \eqref{eqn - twisted lie derivative on local cohomology of log forms, homogeneous} from Lemma \ref{lemma - twisted lie derivative computation on homogeneous}.
\end{proof}

\subsection{Complexes of Analytic \v{C}ech Cohomology of Logarithmic Forms} \enspace

The purpose of this subsection is to repeat the constructions of the previous subsection in the appropriate analytic setting. Specifically, we want an analytic version of Proposition \ref{prop - homogeneous subcomplex quasi iso}. This involves working with the coherent sheaves $\Omega_{X}^{j}(\log \mathscr{A})$ of $\mathscr{O}_{X}$-modules. Often we will replace $``\log \mathscr{A}''$ with $``\log f''$. In this subsection we will always assume $\mathscr{A}$ is central and so $f$ is homogeneous.

Instead of working with the algebraic \v{C}ech complex attached to the maximal ideal $\mathfrak{m}$ (the irrelevant ideal of $R$), we will work with the sheaf-theoretic analytic \v{C}ech complex attached to the open cover $\{D(x_{i})\} = \{ \{x_{i} \neq 0 \}\}_{1 \leq i \leq n}$ of $X \setminus 0$. For an index set $I \subseteq [n]$, let $x_{I}$ denote $\prod_{i \in I} x_{i}$ and let $D(x_{I}) = \{x_{I} \neq 0\}$ the open set in $X = \mathbb{C}^{n}$. Since $D(x_{I})$ is a Rienhardt (i.e. multicircular) domain, every analytic function on $D(x_{I})$ has a unique Laurent expansion at $0$ that converges absolutely on $D(x_{I})$ and uniformly on compact sets. This Laurent expansion is a formal power series in $\mathbb{C}[[x_{1}^{\pm 1}, \dots, x_{n}^{\pm 1}]]$, but using absolute convergence one can check that every monomial $c_{\boldsymbol{\alpha}}x^{\boldsymbol{\alpha}}$ with some $\alpha_{k} < 0$ for $k \notin I$ has $c_{\boldsymbol{\alpha}} = 0$. When considering an analytic function on $D(x_{I})$ we will identify it with this Laurent expansion, and so its associated element of the formal power series ring $\mathbb{C}[[x_{1}, \dots, x_{n}, x_{i_{1}}^{-1}, \dots, x_{i_{t}}^{-1}]]$, where $I = \{i_{1}, \dots, i_{t}\}$.

In particular, an element $\eta_{I} \in \Gamma(D(x_{I}), \Omega_{X}^{j}(\log f))$ is of the form 
\begin{equation} \label{eqn - our expression for log j-forms on distinguished opens}
\eta_{I} = \sum_{\substack{J \subseteq [n] \\ |J| =j}} \frac{ g_{I,J} dx_{J}}{f}
\end{equation}
where $g_{I,J} \in \mathbb{C}[[x_{1}, \dots, x_{n}, x_{i_{1}}^{-1}, \dots, x_{i_{t}}^{-1}]]$ is the convergent Laurent expansion (at $0$) on $D(x_{I})$ and $dx_{J} = \prod_{j \in J} dx_{j}$. (We order the $j \in J$ in increasing order as is the norm.) Since each $g_{I,J}$ is canonically identified with its Laurent expansion and because this series converges absolutely we may write 
\[
g_{I,J} = \sum_{\textbf{p} \in \mathbb{Z}^{n}} c_{\textbf{p}} x_{1}^{p_{1}} \cdots x_{i}^{p_{n}} = \sum_{p \in \mathbb{Z}^{n}} g_{I,j,p}
\]
where 
\[
g_{I,J,p} = \sum_{\substack{\textbf{p} \in \mathbb{Z}^{n} \\ |\textbf{p}| = p}} c_{\textbf{p}} x^{\textbf{p}} = \sum_{\substack{\textbf{p} \in \mathbb{Z}^{n} \\ |\textbf{p}| = p}} c_{\textbf{p}} x_{1}^{p_{1}} \cdots x_{n}^{p_{n}}
\]
and it is understood that $\textbf{p}$ is restricted to integral vectors whose potential negative components correspond to entries from $I$. This decomposition of $g_{I,J}$ into the infinite sum $\sum_{p \in \mathbb{Z}} g_{I,J,p}$ is a decomposition of an analytic function into homogeneous parts, under the standard weighting giving $x_{i}$ weight $1$ and $x_{i}^{-1}$ weight $-1$. We extend this weighting to $\Gamma(D(x_{I}), \Omega_{X}^{j}(\log f))$ by, as before, weighting $\deg(1/f) = -\deg(f)$ and $\deg(dx_{J}) = |J|$. Then every element $\eta_{I} \in \Gamma(D(x_{I}), \Omega_{X}^{j}(\log f))$ has a (infinite) decomposition into homogeneous parts
\begin{equation} \label{eqn - log j-form decomposition into homogeneous parts}
\eta_{I} = \sum_{p \in \mathbb{Z}} \eta_{I,p} \quad \text{ where } \quad \eta_{I,p} = \sum_{\substack{J \subseteq [n] \\ |J| =j}} \frac{ g_{I,J,p+\deg(f)-|J|} dx_{J}}{f} 
\end{equation}
is a homogeneous (infinte sum) of weight $p$.

\begin{remark} \enspace

\begin{enumerate}[label=(\alph*)] \label{rmk - general analytic laurent nonsense}
    \item If $g_{I} = \sum_{p \in \mathbb{Z}} g_{I,p}$ is the homogeneous decomposition of a Laurent expansion (at $0$) of a holomorphic function on $D(x_{I})$, then each $g_{I,p}$ is also the Laurent expansion (at $0$) of a holomorphic function on $D(x_{I})$. This follows from the fact $g_{I}$ is absolutely convergent and so not only can the series terms be rearranged, but subseries $g_{I,p}$ is also absolutely convergent on all of $D(x_{I})$. Also note that, by comparing homogeneous terms, the (partial) derivatives of $g_{I,p}$ are recovered from the (partial) derivatives of $g_{I}$.
    \item If $\eta_{I} = \sum_{p \in \mathbb{Z}} \eta_{I,p}$ is the homogeneous decomposition of a logarithmic $j$-form from $\Gamma(D(x_{I}), \Omega_{X}^{j}(\log \mathscr{A}))$, then so is each $\eta_{p}$. The previous item shows $\eta_{p} \in \frac{1}{f} \Gamma(D(x_{I}), \Omega_{X}^{j}(\log \mathscr{A}))$. So we need only verify $d(\eta_{p})$ also has a pole of order at most one along $f$. We know $d(\eta)$ has a pole of order at most one along $f$, which equates to 
    $d(f) \wedge \sum_{p \in \mathbb{Z}} \sum_{\substack{J \subseteq [n] \\ |J| =j}}  g_{I,J,p+\deg(f)-|J|} dx_{J}$ belonging to $f \cdot \Gamma(D(x_{I}), \Omega_{X}^{j+1})$. Because $\eta_{I}$ is absolutely convergent, we may compute this wedge product term-by-term; because $d(f)$ is homogeneous of degree $\deg(f)$, wedging with $d(f)$ sends a homogeneous series of degree $k$ to a homogeneous series of degree $k+\deg(f)$. By comparing the homogeneous decomposition of the result of wedging with $d(f)$ (and using the fact the function $0$ has a unique Laurent expansion (at $0$) on $D(x_{I})$), we conclude $d(f) \wedge \sum_{\substack{J \subseteq [n] \\ |J| =j}}  g_{I,J,p+\deg(f)-|J|} dx_{J}$ also belongs to $f \cdot \Gamma(D(x_{I}), \Omega_{X}^{j+1})$.
\end{enumerate}
\end{remark}

We explicitly study the sheaf-theoretic \v{C}ech complex of $\Omega_{X}^{j}(\log \mathscr{A})$ attached to the open cover $\{D(x_{i})\}$ of $X \setminus 0$. We name this 
\[
(C^{\bullet}(\{D(x_{i})\}, \Omega_{X}^{j}(\log \mathscr{A})), \epsilon)
\]
where $\epsilon$ is the standard \v{C}ech differential (and defined similarly to the algebraic \v{C}ech differential). We almost always omit the $\epsilon$ for aesthetic reasons.
When working with a particular element from the \v{C}ech complex we use the notation
\[
\eta = \bigoplus_{\substack{I \subseteq [n] \\ |I| = t}} \eta_{I} = \bigoplus_{\substack{I \subseteq [n] \\ |I| = t}}\sum_{p \in \mathbb{Z}} \eta_{I,p} \in \bigoplus_{\substack{I \subseteq [n] \\ |I| = t}} \Gamma(D(x_{I}), \Omega_{X}^{j}(\log f)) 
\]
where $\eta_{I}$ is as in \eqref{eqn - our expression for log j-forms on distinguished opens} and $\eta_{I,p}$ is as in \eqref{eqn - log j-form decomposition into homogeneous parts}.

Similar arguments as in the previous subsection show that the differentials $\nabla_{\omega}$ and $\iota_{E}$ commute with the \v{C}ech differentials, yielding commutative diagrams similar to those of Proposition \ref{prop - definition of N diagram, Cech twisted log de rham lattice} and Proposition \ref{prop - comm square of N diagram with contraction}. Therefore, as in Proposition \ref{prop- contraction and twisted log de rham on local cohomology well defined} these induce well-defined complexes on \v{C}ech cohomology. We give notation for this:

\begin{define} \label{def - analytic twisted complexes of Cech cohomology}
For the open cover $\{D(x_{i})\}$ of $X \setminus 0$, denote the $t^{\text{th}}$ \v{C}ech cohomology of $\Omega_{X}^{j}(\log \mathscr{A})$ with respect to this cover by
\[
H_{\Cech}^{t}(\Omega_{X}^{j}(\log \mathscr{A})) = H^{t}(C^{\bullet}(\{D(x_{i})\}, \Omega_{X}^{j}(\log \mathscr{A}))).
\]
The resultant complex on \v{C}ech cohomology induced by $\nabla_{\omega}$ is named
\[
(H_{\Cech}^{t}(\Omega_{X}^{\bullet}(\log \mathscr{A})), \nabla_{\omega}) = \enspace \to H_{\Cech}^{t}(\Omega_{X}^{j}(\log \mathscr{A})) \xrightarrow[]{\nabla_{\omega}} H_{\Cech}^{t}(\Omega_{X}^{j+1}(\log \mathscr{A})) \to 
\]
and the resultant complex on \v{C}ech cohomolgy induced by $\iota_{E}$ is named
\[
(H_{\Cech}^{t}(\Omega_{X}^{\bullet}(\log \mathscr{A})), \iota_{E}) =  \enspace \to H_{\Cech}^{t}(\Omega_{X}^{j+1}(\log \mathscr{A})) \xrightarrow[]{\iota_{E}} H_{\Cech}^{t}(\Omega_{X}^{j}(\log \mathscr{A})) \to 
\]
\end{define}

It is clear that the \v{C}ech differential sends homogeneous elements of weight $k$ in $C^{t}(\{D(x_{i})\}, \Omega_{X}^{j}(\log \mathscr{A})$ to homogeneous elements of weight $k$ in $C^{t+1}(\{D(x_{i})\}, \Omega_{X}^{j}(\log \mathscr{A})$. Here, by homogeneous, we mean an element $\eta = \bigoplus_{I} \eta_{I} \in \bigoplus_{I} \Gamma(D(x_{I}), \Omega_{X}^{j}(\log \mathscr{A}))$ whose homogeneous decomposition (cf. \eqref{eqn - log j-form decomposition into homogeneous parts}) satisfies $\bigoplus_{I} \eta_{I} = \bigoplus_{I} \sum_{p \in \mathbb{Z}} \eta_{I,p} = \bigoplus_{I} \eta_{I,k}$, that is, $\eta$ is homogeneous of weight $k$ when, for each $I$, $\eta_{I,m} = 0$ for all $m \neq k$. And since $\iota_{E}$ and $\nabla_{\omega}$ also preserve the weight of homogeneous elements, we are authorized to make the following defintion:

\begin{define} \label{def - homogeneous subcomplex of analytic twisted complexes of Cech cohomology}
We denote by 
\[
(C^{\bullet}(\{D(x_{i})\}, \Omega_{X}^{j}(\log \mathscr{A}))_{m}, \epsilon)
\]
the subcomplex of the \v{C}ech complex $(C^{\bullet}(\{D(x_{i})\}, \Omega_{X}^{j}(\log \mathscr{A})), \epsilon)$ comprised solely of homogeneous terms of weight $m$, where by homogeneous we mean in the Laurent series sense as described above. We denote this subcomplex's cohomology by \[
H_{\Cech}^{t}(\Omega_{X}^{\bullet}(\log \mathscr{A})_{m}) = H^{t}((C^{\bullet}(\{D(x_{i})\}, \Omega_{X}^{j}(\log \mathscr{A}))_{m}, \epsilon))
\]
We also have a map of complexes
\begin{equation} \label{eqn - analytic Cech complex cohomology twisted inclusion}
(H_{\Cech}^{t}(\Omega_{X}^{\bullet}(\log \mathscr{A})_{m}), \nabla_{\omega}) \to (H_{\Cech}^{t}(\Omega_{X}^{\bullet}(\log \mathscr{A})), \nabla_{\omega})
\end{equation}
which is induced by including $C^{\bullet}(\{D(x_{i})\}, \Omega_{X}^{j}(\log \mathscr{A})_{m})$ into the whole \v{C}ech complex, and is well-defined because $\nabla_{\omega}$ commutes with the \v{C}ech differential $\epsilon$ and both preserve weights. 
\end{define}

With this set-up we can now state our intended generalization of Proposition \ref{prop - homogeneous subcomplex quasi iso}: we intend to show that with $m = - \iota_{E}(\omega)$, the map \eqref{eqn - analytic Cech complex cohomology twisted inclusion} of complexes induces a surjection on the level of cohomology. While we still use twisted Lie derivative formulas as in Lemma \ref{lemma - twisted lie derivative computation on homogeneous}, unlike the algebraic case their application is not immediate. We track a similar procedure as in \cite{CohomologyComplementFree}, but whereas they work directly with analytic logarithmic de Rham complexes we work a level deeper with their \v{C}ech cohomology, making our version more technically involved. The key tool is the following map:

\begin{define}
Define a map 
\[
\phi: \bigoplus_{\substack{I \subseteq [n] \\ |I| = t}} \Gamma(D(x_{I}), \Omega_{X}^{j}(\log f)) \to \bigoplus_{\substack{I \subseteq [n] \\ |I| = t}} \Gamma(D(x_{I}), \Omega_{X}^{j}(\log f))
\]
by sending each $I$-component $\eta_{I}$ of $\eta$ to
\[
\phi(\eta_{I}) = \phi(\sum_{p} \eta_{I,p}) = \sum_{ p \neq - \iota_{E}(w) } \frac{\eta_{I, p}}{p + \iota_{E}(\omega)}.
\]
\end{define}

\begin{proposition} \label{prop - assignment is well-defined}
The assignment $\phi: \Gamma(D(x_{I}),\Omega_{X}^{j}(\log f)) \to \Gamma(D(x_{I}),\Omega_{X}^{j}(\log f))$ is well-defined. 
\end{proposition}

\begin{proof}
It is enough to verify the claim for $\phi(\eta_{I})$. We first show that $\phi(\eta_{I}) \in \frac{1}{f}\Gamma(D(x_{I}), \Omega_{X}^{j})$, which is essentially a statement about the appropriate Laurent series converging absolutely. By Remark \ref{rmk - general analytic laurent nonsense}, and the absolute convergence of Laurent expansions, the $-\iota_{E}(\omega)$-homogeneous piece $\eta_{I, -\iota_{E}(\omega)}$ is itself in $\frac{1}{f}\Gamma(D(x_{I}), \Omega_{X}^{j})$. And $\phi(\eta_{I}) = \phi(\eta_{I} - \eta_{I,-\iota_{E}(\omega)})$. To test absolute convergence, we apply the modulus to an evaluation of the formal Laurent series $\phi(\eta_{I} - \eta_{I,-\iota_{E}(\omega)})$ at some point. Under $\phi$, every (nonzero) homogeneous $\eta_{I, p}$ of $\eta_{I} - \eta_{I, -\iota_{E}(\omega)}$ gets sent to $(1/p+\iota_{E}(\omega)) \eta_{I,p}$. After applying the modulus, and using the fact that $|1/(p+\iota_{E}(\omega)| \to 0$ as $|p| \to \infty$, we see that each $\mid (1/p+\iota_{E}(\omega)) \eta_{I,p}\mid$ is smaller than $\mid \eta_{I,p}\mid$, provided $p$ is sufficiently far from $0$. Since finite sums of convergent series converge, $\phi(\eta_{I})$ does live in $\frac{1}{f}\Gamma(D(x_{I}), \Omega_{X}^{j})$. 

Now we confirm that $\phi(\eta_{I}) \in \Gamma(D(x_{I}), \Omega_{X}^{j}(\log f))$ by showing $df/f \wedge \phi(\eta_{I}) \in \frac{1}{f} \Gamma(D(x_{I}), \Omega_{X}^{j+1}).$ Using Remark \ref{rmk - general analytic laurent nonsense} again, $\eta_{I, -\iota_{E}(\omega)}$ is logarithmic, and so $\eta_{I} - \eta_{I, -\iota_{E}(\omega)}$ is also logarithmic. Consequently, $df/f \wedge (\eta_{I} - \eta_{I, -\iota_{E}(\omega)}) \in \frac{1}{f} \Gamma(D(x_{I}, \Omega_{X}^{j+1}).$ But because $df/f$ is homogeneous of weight zero, as formal Laurent series we have the equalities
\[
df/f \wedge \phi(\eta_{I}) = df/f \wedge \phi(\eta_{I} - \eta_{I, -\iota_{E}(\omega)}) = \phi(df/f \wedge (\eta_{I} - \eta_{I, -\iota_{E}(\omega)})).
\]
And so the same argument about absolute convergence from the first paragraph, now applied to $\phi(df/f \wedge (\eta_{I} - \eta_{I, -\iota_{E}(\omega)}))$, shows $df/f \wedge \phi(\eta_{I}) \in \frac{1}{f} \Gamma(D(x_{I}), \Omega_{X}^{j+1})$ as required.
\end{proof}

\begin{proposition} \label{prop - assignment preserves Cech cycles and boundaries}
The assignment $\phi$ preserves cycles and boundaries of the \v{C}ech complex $C^{\bullet}(\{D(x_{i})\}, \Omega_{X}^{j}(\log \mathscr{A})).$
\end{proposition}

\begin{proof}
Suppose $\eta = \bigoplus_{I} \eta_{I} \in \bigoplus_{I} \Gamma(D(x_{I}), \Omega_{X}^{j}(\log f))$ is a $t$-cycle (so $|I| = t)$. This means that for each $M \subseteq [n]$ of cardinality $t+1$ we have, in $\Gamma(D(x_{M}), \Omega_{X}^{j}(\log f))$, the equalities
\begin{align*}
\sum_{m \in M} (-1)^{\sigma_{M \setminus \{m\}}(m)} \eta_{M \setminus \{m\}}     &= \sum_{m \in M} (-1)^{\sigma_{M \setminus \{m\}}(m)} \sum_{p} \eta_{M \setminus \{m\}, p} \\
    & = \sum_{p} \sum_{m \in M} (-1)^{\sigma_{M \setminus \{m\}}(m)} \eta_{M \setminus \{m\}, p} \\
    & = 0.
\end{align*}
(Here $\eta_{M \setminus \{m\}}$ is naturally regarded in $\Gamma(D(x_{M}),\Omega_{X}^{j}(\log f))$ by restriction.) We see that, for each individual $p$,
\begin{equation} \label{eqn - each indiviual p vanishes, verifying cycle}
\sum_{m \in M} (-1)^{\sigma_{M \setminus \{m\}}(m)} \eta_{M \setminus \{m\}, p} = 0.
\end{equation}
Moreover, the $M$-component of $\epsilon(\phi(\eta))$ is
\begin{align*}
\sum_{m \in M} (-1)^{\sigma_{M \setminus \{m\}}(m)} \phi(\eta_{M \setminus \{m\}}) 
    &= \sum_{m \in M} (-1)^{\sigma_{M \setminus \{m\}}(m)} \sum_{ p \neq - \iota_{E}(\omega)}\frac{\eta_{M \setminus \{m\}}}{p + \iota_{E}(\omega)} \\
    &= \sum_{p \neq - \iota_{E}(\omega)} \sum_{m \in M} (-1)^{\sigma_{M \setminus \{m\}}(m)} \frac{\eta_{M \setminus \{m\}, p}}{p + \iota_{E}(\omega)}.
\end{align*}
By \eqref{eqn - each indiviual p vanishes, verifying cycle}, we conclude the above equation equals zero and $\epsilon(\phi(\eta)) = 0$ as well.

The case of boundaries is similar. If $\eta$ is a $t$-boundary, that means there is a $\zeta$ such that $\epsilon(\zeta) = \eta$. Arguing as above, the $I$-component of $\epsilon(\zeta)$ and $\eta$ agree at each homogeneous piece: $\epsilon(\zeta)_{I, p} = \eta_{I,p}$ for all $I$ of cardinality $t$ and all $p \in \mathbb{Z}$. As $\phi(\zeta)$ and $\phi(\eta)$ both scale each homogeneous term in a fixed way, $\epsilon(\phi(\zeta)) = \phi(\eta)$.
\end{proof}

\begin{lemma} \label{lemma - cohomological surjections of analytic homogeneous twisted Cech complex to total twisted Cech complex}
When $\mathscr{A}$ is central, the map of complexes \eqref{eqn - analytic Cech complex cohomology twisted inclusion} with $m = -\iota_{E}(\omega)$, that is,
\[
(H_{\Cech}^{t}(\Omega_{X}^{\bullet}(\log \mathscr{A})_{-\iota_{E}(\omega)}), \nabla_{\omega}) \to (H_{\Cech}^{t}(\Omega_{X}^{\bullet}(\log \mathscr{A})), \nabla_{\omega}),
\]
induces surjections on the level of cohomology for all $0 \leq j \leq n$:
\[
H^{j}(H_{\Cech}^{t}(\Omega_{X}^{\bullet}(\log \mathscr{A})_{-\iota_{E}(\omega)}), \nabla_{\omega}) \twoheadrightarrow H^{j}(H_{\Cech}^{t}(\Omega_{X}^{\bullet}(\log \mathscr{A})), \nabla_{\omega}).
\]
\end{lemma}
\begin{proof}
Consider
\[
[\bigoplus_{I} \eta_{I}] \in \ker [H^{t}(C(\{D(x_{i})\}, \Omega_{X}^{j}(\log f))) \xrightarrow[]{\nabla_{\omega}} H^{t}(C(\{D(x_{i})\}, \Omega_{X}^{j}(\log f)))].
\]
We apply $L_{\omega} = \nabla_{\omega}(\iota_{E}) + \iota_{E}(\nabla_{\omega})$, the twisted Lie derivative, to $[\bigoplus_{I} \eta_{I}] = [\bigoplus_{I} \sum_{p} \eta_{I,p}]$. Since this is done by applying $L_{\omega}$ to a representative and since we may do so term-by-term thanks to absolute convergence, we may use Lemma \ref{lemma - twisted lie derivative computation on homogeneous} to compute
\[
[\bigoplus_{I} \sum_{p \in \mathbb{Z}} (p + \iota_{E}(\omega))\eta_{I,p}] = L_{\omega}([\bigoplus_{I} \eta_{I}]) = \nabla_{\omega}(\iota_{E}([\bigoplus_{I} \eta_{I}])) = \nabla_{\omega}([\bigoplus_{I} \iota_{E}(\eta_{I})]).
\]
That is, $\bigoplus_{I} \sum_{p \in \mathbb{Z}} (p + \iota_{E}(\omega))\eta_{I,p} - \nabla_{\omega}(\bigoplus_{I} \iota_{E}(\eta_{I}))$ is a \v{C}ech $t$-boundary. By Proposition \ref{prop - assignment preserves Cech cycles and boundaries}, if we apply $\phi$ to this \v{C}ech $t$-boundary we get another \v{C}ech $t$-boundary. Since $\nabla_{\omega}$ preserves weights, $\phi$ and $\nabla_{\omega}$ commute. So applying $\phi$ to our \v{C}ech $t$-boundary gives the \v{C}ech $t$-boundary
\[
\bigoplus_{I} \sum_{p \neq - \iota_{E}(\omega)} \eta_{I,p} - \nabla_{\omega}(\bigoplus_{I} \phi(\iota_{E}(\eta_{I}))).
\]
With some mild rearranging, we have an equality of \v{C}ech cohomology classes
\[
[\bigoplus_{I} \sum_{p \in \mathbb{Z}} \eta_{I,p} - \bigoplus_{I} \eta_{I, -\iota_{E}(\omega)}] = [\nabla_{\omega}(\bigoplus_{I} \phi(\iota_{E}(\eta_{I})))].
\]
(Here we use the fact that $\bigoplus_{I} \sum_{p \in \mathbb{Z}} \eta_{I,p}$ was assumed to be a \v{C}ech $t$-cycle and so, because the \v{C}ech complex preserves weight, each of its homogeneous parts is also a \v{C}ech $t$-cycle.)

This means that we have the following equality of cohomology classes in $H^{j}(H_{\Cech}^{t}(\Omega_{X}^{\bullet}(\log \mathscr{A})), \nabla_{\omega})$:
\begin{equation} \label{equality of twisted cohomology of Cech cohomology classes, homogenous part}
[[\bigoplus_{I} \sum_{p \in \mathbb{Z}} \eta_{I,p}]] = [[\bigoplus_{I} \eta_{I, -\iota_{E}(\omega)}]].
\end{equation}
Since the map \eqref{eqn - analytic Cech complex cohomology twisted inclusion} explicitly obeys
\[
H_{\Cech}^{t}(\Omega_{X}^{j}(\log \mathscr{A})_{-\iota_{E}(\omega)}) \ni [\bigoplus_{I} \eta_{I, -\iota_{E}(\omega)}] \mapsto [\bigoplus_{I} \eta_{I, -\iota_{E}(\omega)}] \in H_{\Cech}^{t}(\Omega_{X}^{j}(\log \mathscr{A})),
\]
we have, thanks to \eqref{equality of twisted cohomology of Cech cohomology classes, homogenous part}, the induced identification
\begin{align*}
H^{j}(H_{\Cech}^{t}(\Omega_{X}^{\bullet}(\log \mathscr{A})_{-\iota_{E}(\omega)}), \nabla_{\omega}) 
    & \ni [[\bigoplus_{I} \eta_{I,-\iota_{E}(\omega)}]] \\
    &\mapsto [[\bigoplus_{I} \eta_{I}]] \in H^{j}(H_{\Cech}^{t}(\Omega_{X}^{\bullet}(\log \mathscr{A})), \nabla_{\omega}).
\end{align*}
The proves the desired cohomological surjections.
\end{proof}

As a final comment, we note that a similar strategy applies to the cohomology analytic twisted logarithmic de Rham complex directly. That is, the stalk at the origin of the twisted analytic logarithmic (resp. analytic) de Rham complex is quasi-isomorphic to a particular homogeneous subcomplex. Again, we mean homogeneous in the sense of the Laurent series expansion (at $0$). As this does not involve \v{C}ech complexes, this is essentially the same procedure as \cite{CohomologyComplementFree}.

\begin{lemma} \label{lemma - twisted analytic log qi to homogeneous subcomplex}
When $\mathscr{A}$ is central, the following inclusions are quasi-isomorphisms:
\begin{align*}
(\Omega_{X,0}^{\bullet}(\log \mathscr{A})_{-\iota_{E}(\omega)}, \nabla_{\omega}) \xhookrightarrow{\qi} (\Omega_{X,0}^{\bullet}(\log \mathscr{A}), \nabla_{\omega}); \\
(\Omega_{X,0}^{\bullet}(\star \mathscr{A})_{-\iota_{E}(\omega)}, \nabla_{\omega}) \xhookrightarrow{\qi} (\Omega_{X,0}^{\bullet}(\star \mathscr{A}), \nabla_{\omega})
\end{align*}
\end{lemma}

\begin{proof}
We just do part involving the twisted logarithmic de Rham complexes, as the meromorphic case is similar. Here, we consider $\phi$ as a map from $\Omega_{X,0}^{j}(\log \mathscr{A}) \xrightarrow[]{} \Omega_{X,0}^{j}(\log \mathscr{A})$ with an entirely similar definition as before. Let $\eta$ be a logarithmic $j$-form that is a $\nabla_{\omega}$-cycle. One one hand, $L_{\omega}(\eta) = \nabla_{\omega}(\sum_{p} \eta_{p}) = \sum_{p} (p + \iota_{E}(\omega)) \eta_{p}$; on the other hand, $L_{\omega}(\eta) = \nabla_{\omega}(\iota_{E}(\eta))$. Then $\nabla_{\omega}(\phi(\iota_{E}(\eta))) = \eta - \eta_{-\iota_{E}(\omega)}$. So the class $[\eta]$ equals the class $[\eta_{-\iota_{E}(\omega)}]$.
\end{proof}

\subsection{Castelnuovo--Mumford Regularity of Logarithmic Forms} \text{ }

We return to the algebraic setting. In light of Proposition \ref{prop - homogeneous subcomplex quasi iso}, it is useful to understand which graded components of $H_{\mathfrak{m}}^{t}(\Omega_{R}^{j}(\log \mathscr{A}))$ are zero and which are not. For if particular graded components vanish, Proposition \ref{prop - homogeneous subcomplex quasi iso} lets us conclude the twisted de Rham complex on local cohomology modules is acyclic. Instead of trying to answer this question for each pair $(t, j) \in [n]^{2}$, it is much easier to fix $j$ and let $t$ vary. Even still, it is easier to package the data for $\{(0,j), (1,j), \dots, (n, j)\}$ together as in the following classical definition:
\begin{define} \label{def - CM regularity}
Let $R = \mathbb{C}[x_{1}, \dots, x_{n}]$ be graded canonically (so $\deg(x_{i}) = 1)$, $\mathfrak{m}$ the irrelevant ideal, and $M$ a graded module. We allow $M$ to be $\mathbb{Z}$-graded. The \emph{Castelnuovo--Mumford regularity} $\reg(M)$ is
\[
\reg(M) = \max\limits_{t} \{ \max \deg (H_{\mathfrak{m}}^{t}(M)) + t \}, 
\]
where $\max \deg H_{\mathfrak{m}}^{t}(M)$ is the largest $u \in \mathbb{Z}$ such that $H_{\mathfrak{m}}^{t}(M)_{u} \neq 0$, is $\infty$ if no such maximum exists, and is $-\infty$ if $H_{\mathfrak{m}}^{t}(M) = 0$. In particular, the Castelnuovo--Mumford regularity of the zero module is $-\infty$.
\end{define}

The purpose of this subsection is to find a good upper bound on the Castelnuovo--Mumford regularity of $\Omega_{R}^{j}(\log \mathscr{A})$, with the eventual aim of showing $(H_{\mathfrak{m}}^{t}( \Omega_{R}^{\bullet}(\log \mathscr{A})), \nabla_{\omega})$ (and later its analytic counterpart $(H_{\Cech}^{t}(\Omega_{X}^{\bullet}(\log \mathscr{A})), \nabla_{\omega})$) is acyclic via Proposition \ref{prop - homogeneous subcomplex quasi iso} (subject to relatively benign conditions on $\omega$).

\begin{convention}
We use the suffix $-(m)$, with $m \in \mathbb{Z}$, to denote a twist in the grading. Explicitly, the $u$-homogeneous part of $R(m)$ is: $R(m)_{u} = R_{u + m}$. If an isomorphism is written as $M \simeq R(m)$ then the isomorphism honors the grading. Another convention: we use just $\reg(M)$ without specifying the ring in question, hoping context rectifies any confusion.
\end{convention}

\begin{remark} \label{rmk - general/useful facts about regularity} \text{ }
\begin{enumerate}[label=(\alph*)]
    \item Since $R$ is free, $H_{\mathfrak{m}}^{t}(R) = 0$ for all $t \neq n$. Using the \v{C}ech complex to compute local cohomology, 
    \[
    H_{\mathfrak{m}}^{n}(R) \simeq \frac{R[x_{1}^{-1}, \dots, x_{n}^{-1}]}{\sum_{1 \leq i \leq n} R[x_{1}^{-1}, \dots, x_{i-1}^{-1}, x_{i+1}^{-1}, \dots, x_{n}^{-1}]}
    \]
    and so $H_{\mathfrak{m}}^{n}(R)_{-n} = \mathbb{C} \cdot x_{1}^{-1} \cdots x_{n}^{-1}$ and higher homogeneous subspaces vanish. So $\reg(R) = 0$.
    \item If $M$ and $N$ are graded $R$-modules such that $M(\ell) = N$, then $H_{\mathfrak{m}}^{t}(M)(\ell) = H_{\mathfrak{m}}^{t}(N)$ and $\reg(N) + \ell = \reg(M).$
    \item Because the \v{C}ech complex definition of local cohomology implies direct sums commutes with local cohomology, we have:
    \begin{align*}
    \reg(M \oplus N) &= \max\limits_{t} \{ \max\{ \max \deg (H_{\mathfrak{m}}^{t}(M))), \max \deg (H_{\mathfrak{m}}^{t}(N))\}+ t \} \\
        &= \max\{ \reg(M), \reg(N)\}.
    \end{align*}
    \item Suppose $A$ and $B$ are polynomial $k$-algebras that are $\mathbb{N}$-graded. Let $N$ (resp. $M$) be a graded $A$ (resp. $B$) module. If $A_{0}$ and $B_{0}$ are finitely generated over $k$ and the irrelevant ideal of $A$ (resp. $B$) is $A_{> 0}$ (resp. $B_{>0}$), one can define graded local cohomology and regularity similarly to Definition \ref{def - CM regularity}. Then the Kunneth formula for local cohomology modules leads to a Kunneth formula for regularity:
    \[
    \reg(N \otimes_{k} M) = \reg(N) + \reg(M)
    \]
    where $\reg(N \otimes_{k} M)$ considers $N \otimes_{k} M$ as a $A \otimes_{k} B$-module, $\reg(N)$ considers $N$ as a $A$-module, and $\reg(M)$ considers $M$ as a $B$-module. See Lemma 1.5(5) of \cite{SymondsCMRegularityCohomologyRingofGroup} for details; Proposition 2.4 of \cite{BensonRegularityConjectureCohomologyFiniteGroup} for a different proof.
\end{enumerate}
\end{remark}

In \cite{DerksenSidman-CMRegularyByApproximation}, Derksen and Sidman gave a method to approximate Castelnuovo--Mumford regularity which is particularly potent for objects attached to hyperplane arrangements thanks to their amenability to addition/subtraction techniques. It is:

\begin{corollary} \label{cor - DerksenSidman CM approximation result} \normalfont{(Corollary 3.7 of \cite{DerksenSidman-CMRegularyByApproximation})} Let $F = \oplus R(0)$ be a finite, free $R$-module generated in degree zero. Suppose that $M \subseteq F$ is a graded $R$-module and that we have modules $M_{1}, \dots, M_{\ell} \subseteq F$ and ideals $I_{1}, \dots, I_{\ell} \subseteq R$ such that 
\[
I_{i} \cdot M_{i} \subseteq M \subseteq M_{i} \qquad \forall i.
\]
If $I_{1} + \cdots + I_{\ell} = \mathfrak{m}$ and $\reg(M_{i}) \leq r - 1$ for some $r \geq 2$ and all $i$, then $\reg(M) \leq r$.
\end{corollary}

\noindent We have emphasized $F$ is generated in degree zero to be in the situation of loc.\ cit.\

Derksen and Sidman use this result to get an upper bound on the regularity of $\Der_{R}(-\log \mathscr{A})$. In Proposition 1.3 of \cite{SaitoDegenerationOfPoleOrderSpectralSequence4Variables}, M. Saito showed their argument essentially proved (i.e. with only a minor adjustment) a stronger bound: for a central, essential arrangement $A$,
\[
\reg(\Der_{R}(-\log \mathscr{A})) \leq \deg(\mathscr{A}) - \rank(\mathscr{A}).
\]
(We use the convention $\deg(\partial_{i}) = -1$.) Taking M. Saito's improvement as a model, we estimate the regularity of the logarithmic $j$-forms $\Omega_{R}^{j}(\log \mathscr{A})$ along $\mathscr{A}$. Our situation is more involved because: (1) one cannot apply Corollary 3.7 of \cite{DerksenSidman-CMRegularyByApproximation} to $\Omega_{R}^{j}(\log \mathscr{A})$ itself; (2) the inductive scheme mandates relating $(j-1)$-logarithmic forms to $j$-logarithmic forms; (3) the case of $\Omega_{R}^{\rank(\mathscr{A}) - 1}(\log \mathscr{A})$ is bespoke.

\begin{example} \label{ex - CM regularity for SNC}
Consider $\mathscr{A}$ cut out by $f = x_{1} \cdots x_{n} \in R$, i.e. a simple normal crossing divisor. It is easy to check that $\Omega_{R}^{0}(\log \mathscr{A}) = R$ and $\Omega_{R}^{1}(\log \mathscr{A}) = R \cdot \frac{dx_{1}}{x_{1}} \oplus \cdots \oplus R \cdot \frac{dx_{n}}{x_{n}} \simeq R \oplus \cdots \oplus R$ (for example, use K. Saito's freeness criterion for logarithmic $1$-forms), where the last isomorphism is graded. Since $f$ is free, $\Omega_{R}^{j}(\log \mathscr{A}) = \wedge^{j} \Omega_{R}^{1}(\log \mathscr{A}) \simeq \oplus^{\binom{n}{j}} R$, with the last isomorphism graded. So for simple normal crossing divisors in $\mathbb{C}[x_{1}, \dots, x_{n}]$, 
\[
\reg(\Omega_{R}^{j}(\log \mathscr{A})) = 0 \quad \text{for all} \quad 0 \leq j \leq n.
\]
\end{example}

With this example in hand, here is our theorem bounding Castelnuovo--Mumford regularity of logarithmic $j$-forms:

\begin{theorem} \label{thm - computing CM regularity of log forms}
For $\mathscr{A}$ a central, essential, reduced hyperplane arrangement, the Castelnuovo--Mumford regularity of $\Omega_{R}^{j}(\log \mathscr{A})$ is bounded by:
\[
\reg(\Omega_{R}^{j}(\log \mathscr{A})) \leq 0 \quad \text{for} \quad 1 \leq j \leq \rank(\mathscr{A}) - 1.
\]
When $j = 0$ or $j=\rank(\mathscr{A})$, the module of logarithmic zero or $\rank(\mathscr{A})$-forms are free $R$-modules and so only $H_{\mathfrak{m}}^{\rank{\mathscr{A}}}(-) \neq 0$. In these cases, $\reg(\Omega_{R}^{0}(\log \mathscr{A})) = 0$ and $\reg(\Omega_{R}^{\rank(\mathscr{A})}(\log \mathscr{A})) = \rank(\mathscr{A}) - \deg(\mathscr{A})$. 
\end{theorem}

\begin{proof}
Throughout, let $f = f_{1} \cdots f_{d}$ be a defining equation for $\mathscr{A}$ and, if not explicitly changed, $n = \rank(\mathscr{A})$, $d = \deg(\mathscr{A})$.

Step 0: We handle the straightforwards cases $j=0$ or $j=\rank(\mathscr{A})$. All we must show is freeness and then compute regularity. It is well known (and easy to verify using \eqref{eqn - log j-forms identified with L object} below) that $\Omega_{R}^{0}(\log \mathscr{A}) = R$. Consequently, $\reg(\Omega_{R}^{0}(\log \mathscr{A})) = 0.$ Now consider $j=\rank(\mathscr{A})$. If $\eta \in \frac{1}{f} \Omega_{R}^{\rank(\mathscr{A})}$, then $d(\eta) = 0 \in \frac{1}{f}\Omega_{R}^{\rank(\mathscr{A})+1} = 0$. So: $\Omega_{R}^{\rank(\mathscr{A})}(\log \mathscr{A}) = \frac{1}{f} \Omega_{R}^{\rank(\mathscr{A})} = R(\deg(\mathscr{A}) - \rank(\mathscr{A}))$; $\reg(\Omega_{R}^{\rank(\mathscr{A})}(\log \mathscr{A})) = \rank(\mathscr{A}) - \deg(\mathscr{A})$.

Step 1: To deal with $1 \leq j \leq \rank(\mathscr{A}) - 1$, we change our target in order to apply Corollary 3.7 of \cite{DerksenSidman-CMRegularyByApproximation} (transcribed as Corollary \ref{cor - DerksenSidman CM approximation result}). Observe: the logarithmic $j$-forms can be characterized by: $\sum_{\mid I \mid = j} \frac{a_{I}}{f} dx_{I} \in \Omega_{R}^{j}(\log \mathscr{A})$ if and only if $d(\sum_{\mid I \mid = j} \frac{a_{I}}{f} dx_{I}) \in \frac{1}{f} \Omega_{R}^{j+1}$ if and only if $\sum_{\mid I \mid = j} \frac{a_{I}}{f^{2}} d(f) \wedge dx_{I} \in \frac{1}{f} \cdot \Omega_{R}^{j+1}$ if and only if $\sum_{\mid I \mid = j} \frac{a_{I}}{f} d(f) \wedge dx_{I} \in \Omega_{R}^{j+1}$ if and only if $\sum_{\mid I \mid = j} a_{I} d(f) \wedge dx_{I} \in f \cdot \Omega_{R}^{j+1}$. Since $R$ is a UFD and $f = f_{1} \cdots f_{d}$ is a factorization into irreducibles, we conclude:
\begin{equation} \label{eqn - log j-forms identified with L object}
\sum_{\mid I \mid = j} \frac{a_{I}}{f} dx_{I} \in \Omega_{R}^{j}(\log \mathscr{A}) \iff \sum_{\mid I \mid = j} a_{I} d(f_{k}) \wedge dx_{I}  \in f_{k} \cdot \Omega_{R}^{j+1} \text{ for all } f_{k}.
\end{equation}

Define 
\[
L^{j}(\mathscr{A}) = \{ \eta \in \Omega_{R}^{j} \mid d(f_{k}) \wedge \eta \in f_{k} \cdot \Omega_{R}^{j+1} \text{ for all } f_{k}\}
\]
and note that 
\begin{equation} \label{eqn - L^j equals intersection of L^j(each hyperplane)}
L^{j}(\mathscr{A}) = \bigcap_{k} L^{j}(\Var(f_{k})).
\end{equation}
The equivalence \eqref{eqn - log j-forms identified with L object} tells us the map $\sum_{\mid I \mid = j} \frac{a_{I}}{f} dx_{I} \mapsto \sum_{\mid I \mid = j} a_{I} dx_{I}$ gives an isomorphism $\Omega_{R}^{j}(\log \mathscr{A}) \xrightarrow[]{\simeq} L^{j}(\mathscr{A})$ that sends homogeneous elements of degree $u$ to homogeneous elements of degree $u + \deg(\mathscr{A})$. Therefore
\[
\Omega_{R}^{j}(\log \mathscr{A})(-\deg(\mathscr{A})) \simeq L^{j}(\mathscr{A})
\]
Consequently, $H_{\mathfrak{m}}^{t}(\Omega_{R}^{j}(\log \mathscr{A}))(-\deg(\mathscr{A})) \simeq H_{\mathfrak{m}}^{t}(L^{j}(\mathscr{A}))$ and 
\begin{equation} \label{eqn - reg for log j-forms = reg L-j stuff - d}
\reg(\Omega_{R}^{j}(\log \mathscr{A})) = \reg(L^{j}(\mathscr{A})) - d.
\end{equation}
Therefore our theorem is equivalent to $\reg(L^{j}(\mathscr{A})) \leq \deg(\mathscr{A})$, making the following our new objective:
\[
\textbf{The Claim}: \reg(L^{j}(\mathscr{A})) \leq \deg(\mathscr{A}) \text{ for } 1 \leq j \leq \rank(\mathscr{A}) - 1.
\]

Step 2: We first prove \textbf{The Claim} for $L^{1}(\mathscr{A})$ by induction on the rank and the degree of our arrangement. If $\rank(\mathscr{A}) = 2$ with $d = \deg(\mathscr{A})$, then the arrangement is automatically free and by standard considerations, $\Omega_{R}^{1}(\log \mathscr{A}) \simeq R \delta_{1} \oplus R \delta_{2}$ where $\delta_{1}$ has degree $2-d$ and $\delta_{2}$ has degree $0$. (Note that the determinant of the Saito matrix corresponding to $\delta_{1}, \delta_{2}$ equals, up to unit, $\frac{1}{f}dx_{1}dx_{2}$, and one logarithmic form, $\delta_{1}$, corresponds to the Euler operator, whereas the other, $\delta_{2}$, corresponds to $(\partial_{1} \bullet f)\partial_{2} - (\partial_{2} \bullet f)\partial_{1}$.)  So $\Omega_{R}^{1}(\log \mathscr{A}) \simeq R(d-2) \oplus R$ and $\reg( \Omega_{R}^{1}(\log \mathscr{A})) = \max \{ 2-d, 0 \} = 0$, since, by essentiality, $d \geq 2$. So $\reg(L^{1}(\mathscr{A})) = d$.

So assume the rank of $\mathscr{A}$ is at least $3$, the degree of $\mathscr{A}$ is $d+1$ and the claim holds for all essential, central arrangements of smaller degree or rank. (The base case of a simple normal crossing divisor with rank($\mathscr{A}$) hyperplanes is handled in Example \ref{ex - CM regularity for SNC}.) By essentiality, we may find $n$ hyperplanes from $\mathscr{A}$ cutting out the origin. 

\emph{Case 2.1}: There exists $n$ hyperplanes $f_{1}, \dots, f_{n}$ (relabelling as necessary) from $\mathscr{A}$ such that: (1) $\mathfrak{m} = (f_{1}, \dots, f_{n})$; (2) deleting any $\Var(f_{i})$ from $\mathscr{A}$ produces an essential arrangement. Let $\mathscr{A}_{i} = \mathscr{A} \setminus \Var(f_{i})$ be the result post deletion. By \eqref{eqn - L^j equals intersection of L^j(each hyperplane)}, $L^{1}(\mathscr{A}) \subseteq L^{1}(\mathscr{A}_{i}).$ And using \eqref{eqn - log j-forms identified with L object} and \eqref{eqn - L^j equals intersection of L^j(each hyperplane)}, we confirm that $f_{i} \cdot L^{1}(\mathscr{A}_{i}) \subseteq L^{1}(\mathscr{A})$. So we use the inductive hypothesis to get $\reg(L^{1}(\mathscr{A}_{i})) \leq d$ and $\reg(L^{1}(\mathscr{A}_{i})(1)) \leq d -1.$ As $L^{1}(\mathscr{A}) \subseteq R dx_{1} \oplus \cdots \oplus R dx_{n} = R(-1) \oplus \cdots \oplus R(-1)$, if we shift degrees and consider $L^{1}(\mathscr{A})(1)$, $L^{1}(\mathscr{A}_{i})(1) \subseteq R(0) \oplus \cdots \oplus R(0)$ we may invoke Corollary 3.7 of \cite{DerksenSidman-CMRegularyByApproximation} (see Corollary \ref{cor - DerksenSidman CM approximation result}) to get $\reg(L^{1}(\mathscr{A})(1)) \leq d$ or equivalently $\reg(L^{1}(\mathscr{A})) \leq d + 1$, finishing the case.

\emph{Case 2.2}: There exists a hyperplane $f_{i}$ such that the deleted arrangement $\mathscr{A} \setminus \Var(f_{i})$ is not-essential. Then, after possibly changing coordinates, $\mathscr{A}$ has a defining equation $x_{n} g$ where $\Var(g) = \mathscr{A}_{i}$, $g \in S = \mathbb{C}[x_{1}, \dots, x_{n-1}]$, and $x_{n}$ corresponds to the deleted hyperplane $f_{i}$. By the Kunneth formula for products, see Remark \ref{rmk - Kunneth formula for log diff forms}, 
\[
\Omega_{R}^{1}(\log \mathscr{A}) = (\Omega_{S}^{1}(\log \Var(g)) \otimes_{\mathbb{C}} \mathbb{C}[x_{n}]) \oplus (S \otimes_{\mathbb{C}} \Omega_{\mathbb{C}[x_{n}]}^{1}(\log \Var(x_{n})))
\]
and by the Kunneth formula for local cohomology and regularity, see Remark \ref{rmk - general/useful facts about regularity}, 
\[
\reg(\Omega_{R}^{1}(\log \mathscr{A})) = \max\{ \reg(\Omega_{S}^{1}(\Var(g))) + \reg(\mathbb{C}[x_{n}]), \reg(S) + \reg(\Omega_{\mathbb{C}[x_{n}]}^{1}(\Var(x_{n}))\}.
\]
Because $\Omega_{\mathbb{C}[x_{n}]}^{1}(\Var(x_{n})) = \mathbb{C}[x_{n}]\frac{dx_{n}}{x_{n}} \simeq \mathbb{C}[x_{n}]$ (a graded isomorphism), the only regularity term of the four summands that is potentially nonzero is $\reg(\Omega_{S}^{1}(\Var(g)))$. So $\reg(\Omega_{R}^{1}(\log \mathscr{A})) \leq \max\{\reg(\Omega_{S}^{1}(\Var(g))), 0\}$. Since $\mathscr{A}$ is essential of rank at least three, $\Var(g)$ must be essential of rank at least two (and of one less degree), and so we can use the induction hypothesis: $\reg(L^{1}(\Var(g)) \leq d$ and, by \eqref{eqn - reg for log j-forms = reg L-j stuff - d}, $\reg(\Omega_{S}^{1}(\Var(g))) \leq 0$. Thus $\reg(\Omega_{R}^{1}(\Var(f)) \leq 0$ and $\reg(L^{1}(\mathscr{A})) = d+1$ completing this case and terminating the inductive step. 

Step 3: Now we prove \textbf{The Claim} holds for $L^{j}(\mathscr{A})$, for any $1 \leq j \leq \rank(\mathscr{A}) - 1$, by induction on $\deg(\mathscr{A})$ and $\rank(\mathscr{A})$, utilizing Step 2. In particular, we may assume $\rank (\mathscr{A}) \geq 3$ and $j \geq 2$. The minimal value $\deg \mathscr{A}$ can obtain while remaining essential is $\rank(\mathscr{A})$ in which case Example \ref{ex - CM regularity for SNC} applies. So assume that the claim holds for any central, essential $\mathscr{B}$ whose degree is less than $\deg(\mathscr{A})$ or whose rank is smaller than $\rank(\mathscr{A})$. Now we do three cases, the first two similar to those from Step 2.

\emph{Case 3.1}: There exist $n$ hyperplanes $f_{1}, \dots, f_{n}$ (re-labelling as necessary) from $\mathscr{A}$ such that: (1) these $n$ hyperplanes generate $\mathfrak{m}$; (2) the arrangement $\mathscr{A}_{i} = \mathscr{A} \setminus \Var(f_{i})$ is essential. Since each $\mathscr{A}_{i}$ is promised to be essential, $\rank(\mathscr{A}_{i}) = \rank(\mathscr{A})$ and so the $\mathscr{A}_{i}$ fall into our inductive scheme: $\deg(\mathscr{A}_{i}) = \deg(\mathscr{A}) - 1$ and $j \leq \rank(\mathscr{A}) - 1 = \rank(\mathscr{A}_{i}) - 1$. So we know that $\reg(L^{j}(\mathscr{A}_{i})) \leq \deg(\mathscr{A}_{i}) = \deg(\mathscr{A}) - 1$. Or, as we will prefer, $\reg(L^{j}(\mathscr{A}_{i})(j)) \leq \deg(\mathscr{A}) - 1 - j.$ Using \eqref{eqn - log j-forms identified with L object} and \eqref{eqn - L^j equals intersection of L^j(each hyperplane)}, we again have
\[
f_{i} \cdot L^{j}(\mathscr{A}_{i}) \subseteq L^{j}(\mathscr{A}) \subseteq L^{j}(\mathscr{A}_{i}). 
\]
Now $L^{j}(\mathscr{A}) \subseteq \oplus_{\mid I \mid = j} R dx_{I} = \oplus_{\mid I \mid = j} R(-j).$ Shift degrees as in Case 2.1, so that $L^{j}(\mathscr{A})(j)$, $L^{j}(\mathscr{A}_{i})(j) \subseteq \oplus_{\mid I \mid = j} R$. Now we can appeal to Corollary 3.7 of \cite{DerksenSidman-CMRegularyByApproximation} (see Corollary \ref{cor - DerksenSidman CM approximation result}) to find that $\reg(L^{j}(\mathscr{A})(j)) \leq \deg(\mathscr{A}) - j$, that is $\reg(L^{j}(\mathscr{A})) \leq \deg(\mathscr{A})$, as required.

\emph{Case 3.2}: Both: (1) there is a hyperplane $f_{i}$ where $\rank(\mathscr{A} \setminus \Var(f_{i})) < \rank(\mathscr{A})$; (2) $j < \rank(\mathscr{A}) - 1 = n-1$. Here we may, changing coordinates as needed, assume $f = x_{n}g$ where $g \in S = \mathbb{C}[x_{1}, \dots, x_{n-1}]$. Since $\mathscr{A}$ is essential, $\Var(g)$ is itself an essential arrangement whose rank is $\rank(\mathscr{A}) - 1$. And $2 \leq j \leq \rank(\mathscr{A}) - 2 = \rank(\Var(g)) - 1$. So $\Var(g)$ falls into the inductive set-up and $\reg(L^{q}(\Var(g))) \leq \deg(\Var(g)) = \deg(\mathscr{A}) - 1$ for all $1 \leq q \leq j$. We may use the Kunneth formula for products (Remark \ref{rmk - Kunneth formula for log diff forms}) to find
\[
\Omega_{R}^{j}(\log \mathscr{A}) \simeq \Omega_{S}^{j}(\Var(g)) \otimes_{\mathbb{C}} \mathbb{C}[x_{n}] \oplus \Omega_{S}^{j-1}(\Var(g)) \otimes_{\mathbb{C}} \Omega_{\mathbb{C}[x_{n}]}^{1}(\Var(x_{n}))
\]
and, using the fact $\Omega_{\mathbb{C}[x_{n}]}^{1}(\Var(x_{n})) \simeq \mathbb{C}[x_{n}]$ (graded), the induced Kunneth formula on regularity (Remark \ref{rmk - general/useful facts about regularity}) gives
\[
\reg(\Omega_{R}^{j}(\log \mathscr{A})) = \max\{ \reg(\Omega_{S}^{j}(\Var(g))) + \reg(\mathbb{C}[x_{n}]), \reg(\Omega_{S}^{j-1}(\Var(g))) + \reg(\mathbb{C}[x_{n}])\}.
\]
By the inductive set-up, $\reg(\Omega_{S}^{j}(\Var(g))) \leq 0$ and $\reg(\Omega_{S}^{j-1}(\Var(g))) \leq 0$. So $\reg(\Omega_{R}^{j}(\log \mathscr{A})) \leq 0$ and $\reg(L^{j}(\mathscr{A})) \leq \deg(\mathscr{A})$. 

\emph{Case 3.3}: j = $\rank(\mathscr{A}) -1 = n-1.$ Write $d = \deg(\mathscr{A})$. Here, with $\widehat{dx_{i}} = dx_{1} \wedge \cdots \wedge dx_{i-1} \wedge dx_{i+1} \wedge dx_{n}$ and $dx = dx_{1} \wedge \cdots \wedge dx_{n}$, the equivalence \eqref{eqn - log j-forms identified with L object} is enriched:
\begin{align*}
\sum_{i} \frac{a_{i}}{f} \widehat{dx_{i}} \in \Omega_{R}^{n-1}(\log \mathscr{A}) &\iff \sum_{i} a_{i} d(f) \widehat{dx_{i}} = (\sum_{i} (-1)^{i-1} a_{i} \partial_{i} \bullet f)) dx \in R \cdot f dx \\
    & \iff (\sum_{i} (-1)^{i-1} a_{i} \partial_{i}) \bullet f \in R \cdot f \\
    & \iff (\sum_{i} (-1)^{i-1} a_{i} \partial_{i}) \in \Der_{R}(-\log \mathscr{A}).
\end{align*}
In other words, $\Omega_{R}^{n-1}(\log \mathscr{A}) \ni \sum_{i} a_{i} \widehat{dx_{i}} \mapsto (\sum_{i} (-1)^{i-1} a_{i} \partial_{i}) \in \Der_{R}(-\log \mathscr{A})$ is an isomorphism sending homogeneous elements of degree $u$ to homogeneous elements of degree $u + d - n$. (We use the convention that $\partial_{i}$ has degree $-1$ as in \cite{SaitoDegenerationOfPoleOrderSpectralSequence4Variables}.) Therefore
\[
\Omega_{R}^{n-1}(\log \mathscr{A})(n - d) \simeq \Der_{R}(\log \mathscr{A})
\]
and $\reg(\Omega_{R}^{n-1}(\log \mathscr{A})) = \reg(\Der_{R}(\log \mathscr{A})) + n - d$. By Proposition 1.3 of \cite{SaitoDegenerationOfPoleOrderSpectralSequence4Variables}, $\reg(\Der_{R}(\log \mathscr{A})) \leq d - n$ and so $\reg(\Omega_{R}^{n-1}(\log \mathscr{A})) \leq 0$. Therefore $\reg(L^{n-1}(\mathscr{A})) \leq \deg(\mathscr{A})$ completing this case, the induction, and the proof.
\end{proof}

\subsection{Acyclic Complexes of Local Cohomology Modules} \text{ }

We are now in position to utilize our work on Castelnuovo--Mumford regularity and prove that, when $\omega = \sum \lambda_{k} df_{k} / f_{k}$ and $\{\lambda_{k}\}$ satisfy some combinatorial numerical restrictions, the complexes of local cohomology modules of logarithmic de Rham modules, with differential induced by $\nabla_{\omega}$, is acyclic. 

\begin{corollary} \label{cor - specific weights at origin, acyclic complex of local cohomology modules}
Let $f = f_{1} \cdots f_{d} \in R = \mathbb{C}[x_{1}, \dots, x_{n}]$ be a central, essential, and reduced hyperplane arrangement, $\lambda_{1}, \dots, \lambda_{d} \in \mathbb{C}$ weights, and $\omega = \sum \lambda_{k} df_{k}/f_{k} = \sum \lambda_{k} d \log f_{k}$ the associated logarithmic one form. If
\begin{equation} \label{eqn - acyclic local cohomology weight condition at the origin}
\sum \lambda_{k} \notin \mathbb{Z}_{\geq \min\{2, \rank(\mathscr{A})\}}
\end{equation}
then for each $0 \leq t \leq n$ the following complex is acyclic:
\[
(H_{\mathfrak{m}}^{t}(\Omega_{R}^{\bullet}(\log \mathscr{A})), \nabla_{\omega}).
\]
\end{corollary}

\begin{proof}
\emph{Case 1:} $\rank(\mathscr{A}) \geq 2$. 

By Proposition \ref{prop - homogeneous subcomplex quasi iso}, we know that the inclusion of $-\iota_{E}(\omega)$-homogeneous subcomplex of $(H_{\mathfrak{m}}^{t}(\Omega_{R}^{\bullet}(\log \mathscr{A})), \nabla_{\omega})$ is a quasi-isomorphism:
\begin{equation} \label{eqn - specific weights, homogeneous subcomplex quasi iso}
(H_{\mathfrak{m}}^{t}(\Omega_{R}^{\bullet}(\log \mathscr{A}))_{-\iota_{E}(\omega)}, \nabla_{\omega}) \xhookrightarrow{\qi} (H_{\mathfrak{m}}^{t}(\Omega_{R}^{\bullet}(\log \mathscr{A})), \nabla_{\omega}).
\end{equation}
Since $\iota_{E}(\omega) = \sum \lambda_{k}$, it is enough to show that for each $t$ and each $j$ the subspace of $H_{\mathfrak{m}}^{t}(\Omega_{R}^{j}(\log \mathscr{A}))$ degree $-\sum \lambda_{k}$ elements is zero. For then the left hand side of \eqref{eqn - specific weights, homogeneous subcomplex quasi iso} is a complex of zero modules. In an equation, it suffices to show:
\begin{equation} \label{eqn - specific weights, subcomplex vanishes, degree condition using reg}
H_{\mathfrak{m}}^{t}(\Omega_{R}^{j}(\log f))_{- \sum \lambda_{k}} = 0.
\end{equation}

First of all, if $\sum \lambda_{k} \notin \mathbb{Z}$, the criterion \eqref{eqn - specific weights, subcomplex vanishes, degree condition using reg} trivially holds. So we may assume $\sum \lambda_{k} \in \mathbb{Z}$. And, because the $j$-logarithmic forms are reflexive (cf. subsection 2.2 of \cite{DenhamSchulzeComplexesDualityChern}), they have depth at least two. By standard facts about local cohomology, $H_{\mathfrak{m}}^{0}(\Omega_{R}^{j}(\log \mathscr{A}))$ and $H_{\mathfrak{m}}^{1}(\Omega_{R}^{j}(\log \mathscr{A}))$ then vanish. Certainly \eqref{eqn - specific weights, subcomplex vanishes, degree condition using reg} then holds for $t=0,1$.

We must deal with $t \geq 2$, where we use the regularity estimate of Theorem \ref{thm - computing CM regularity of log forms}: $\reg(\Omega_{R}^{j}(\log \mathscr{A})) \leq 0.$ Combining this with the definition of regularity in terms of local cohomology modules yields
\[
\max \deg (H_{\mathfrak{m}}^{t}(\Omega_{R}^{j}(\log \mathscr{A}))) + t \leq 0.
\]
Therefore, for all $t \geq 2$ and for all $j$, 
\[
H_{\mathfrak{m}}^{t}(\Omega_{R}^{j}(\log \mathscr{A}))_{\ell} = 0 \text{ whenever } \ell \notin \{-2, - 3, -4, \dots \}.
\]
Then our hypothesis \eqref{eqn - acyclic local cohomology weight condition at the origin} ensures \eqref{eqn - specific weights, subcomplex vanishes, degree condition using reg} holds for $t \geq 2$, completing Case 1.

\emph{Case 2:} $\rank(\mathscr{A}) = 1$.

We argue in the same way, except noting that because $\Omega_{R}^{j}(\log \mathscr{A})$ is automatically free and $R$ has depth one, that $H_{\mathfrak{m}}^{0}(\Omega_{R}^{j}(\mathscr{A}))$ vanishes. Thus we need to show that 
\[
H^{1}_{\mathfrak{m}}(\Omega_{R}^{j})_{\ell} = 0 \text{ whenever } \ell \notin \{-1, -2, -3, \dots \}.
\]
for $j = 0$ and $j = 1$. Now use Example \ref{ex - CM regularity for SNC}.
\end{proof}

\subsection{Acyclic Complexes of Analytic \v{C}ech Cohomology} \enspace

Now we can prove an analytic version of Corollary \ref{cor - specific weights at origin, acyclic complex of local cohomology modules}, using the set-up and notation of subsection 2.3. This will be the crucial result powering the proof of Theorem \ref{thm - analytic twisted lct} (analytic Twisted Logarithmic Comparison Theorem). 

Instead of twisted complexes of local cohomology modules, we have twisted complexes of \v{C}ech cohomology, cf. Definition \ref{def - analytic twisted complexes of Cech cohomology}, and the role of the homogeneous subcomplex in the proof of Corollary \ref{cor - specific weights at origin, acyclic complex of local cohomology modules} is played by the subcomplex of homogeneous forms in the Laurent series sense of subsection 2.3, cf. Definition \ref{def - homogeneous subcomplex of analytic twisted complexes of Cech cohomology}. The argument uses GAGA machinery to enrich the arguments of the previous subsection; consequently it still utilizes the bound on Castelnuovo--Mumford regularity provided in Theorem \ref{thm - computing CM regularity of log forms} to employ Lemma \ref{lemma - cohomological surjections of analytic homogeneous twisted Cech complex to total twisted Cech complex}. Within the proof, we use $(-)^{\anal}$ (resp. $(-)_{\alg}$) to denote the analytic (resp. algebraic) interpreation of $(-)$. 

\begin{corollary} \label{cor - specific weights at origin, acyclic complex of analytic Cech cohomology}
Let $f = f_{1} \cdots f_{d} \in R = \mathbb{C}[x_{1}, \dots, x_{n}]$ be a central, essential, and reduced hyperplane arrangement, $\lambda_{1}, \dots, \lambda_{d} \in \mathbb{C}$ weights, and $\omega = \sum \lambda_{k} df_{k}/f_{k} = \sum \lambda_{k} d \log f_{k}$ the associated logarithmic one form. If
\begin{equation} \label{eqn - acyclic analytic Cech cohomology complex, weight condition at the origin}
\sum \lambda_{k} \notin \mathbb{Z}_{\geq \min\{2, \rank(\mathscr{A})\}}
\end{equation}
and if $n \geq 2$, then for each $1 \leq t \leq n$ the following complex is acyclic:
\[
(H_{\Cech}^{t}(\Omega_{X}^{\bullet}(\log \mathscr{A})), \nabla_{\omega}).
\]
\end{corollary}

\begin{proof}
By Lemma \ref{lemma - cohomological surjections of analytic homogeneous twisted Cech complex to total twisted Cech complex}, it suffices to show that for $0 \leq j \leq n$, the \v{C}ech cohomology $H_{\Cech}^{t}(\Omega_{X^{\anal}}^{j}(\log \mathscr{A})_{-\iota_{E}(\omega)})$ vanishes whenever $ 1 \leq t \leq n$. If $-\iota_{E}(\omega) \notin \mathbb{Z}$ the claim is immediate: there are no homogeneous (in the Laurent series sense) elements of non-integral weight. So we may assume $-\iota_{E}(\omega) \in \mathbb{Z}$ and will proceed by a GAGA argument.

Consider the graded $R$-module $\Omega_{R}^{j}(\log \mathscr{A})(-\iota_{E}(\omega))$. The sheafification $\widetilde{\Omega_{R}^{j}(\log \mathscr{A})(-\iota_{E}(\omega)})$ of this graded module gives a coherent algebraic sheaf of $\mathscr{O}_{\mathbb{P}_{\alg}^{n-1}}$-modules on $\mathbb{P}_{\alg}^{n-1}$. It is well known that for all $t \geq 1$ we have the following identification of sheaf cohomology with local cohomology:
\begin{equation} \label{eqn - local cohomology certain degree equals projective sheaf cohomology}
H_{\mathfrak{m}}^{t+1}(\Omega_{R}^{j}(\log \mathscr{A}))_{-\iota_{E}(\omega)} = H^{t}(\widetilde{\Omega_{R}^{j}(\log \mathscr{A})(-\iota_{E}(\omega)}))
\end{equation}
On one hand, the condition \eqref{eqn - acyclic analytic Cech cohomology complex, weight condition at the origin} on the residue of $\omega$ implies that the LHS of \eqref{eqn - local cohomology certain degree equals projective sheaf cohomology} vanishes. Indeed this follows from Castelnuovo--Mumford regularity estimate of Theorem \ref{thm - computing CM regularity of log forms} as in Corollary \ref{cor - specific weights at origin, acyclic complex of local cohomology modules}. On the other hand, let $K^{j, \anal}$ be the analytification of $\widetilde{\Omega_{R}^{j}(\log \mathscr{A})(-\iota_{E}(\omega)})$. By GAGA, $K^{j, \anal}$ is a coherent analytic sheaf of $\mathscr{O}_{\mathbb{P}^{n-1, \anal}}$-modules on $\mathbb{P}^{n-1, \anal}$. Again by GAGA, the sheaf cohomology of $K^{j,\anal}$ equals the right hand side of \eqref{eqn - local cohomology certain degree equals projective sheaf cohomology}. Altogether this gives $0 = H^{t}(K^{j, \anal})$ for all $t \geq 1$.

Now consider the \v{C}ech complex $C^{\bullet}(\{D^{+}(x_{i})\}, K^{j, \anal})$ attached to the open cover $\{D^{+}(x_{i})\}_{1 \leq i \leq n}$ of $\mathbb{P}^{n-1, \anal}$. Since $D^{+}(x_{i})$ and any intersection of the $\{D^{+}(x_{i})\}$ is Stein, and since higher sheaf cohomology (of coherent sheaves) vanishes on Steins, this \v{C}ech complex computes the sheaf cohomology $K^{j, \anal}$. So for all $t \geq 1$ we have
\begin{equation} \label{eqn - vanishing of analytification projective sheaf, also equal to Cech complex of analytification projective sheaf}
0 = H^{t}(K^{j, \anal}) = H^{t}(C^{\bullet}(\{D^{+}(x_{i})\}, K^{j, \anal})).
\end{equation}

Recall $K^{j, \anal}$ is the analytification of $\widetilde{\Omega_{R}^{j}(\log \mathscr{A})(-\iota_{E}(\omega)})$. Because membership in $\Omega_{R}^{j}(\log \mathscr{A})[x_{I}^{-1}]$ is determined by the kernel of a map defined wholly in terms of $\partial_{1} \bullet f, \dots, \partial_{n} \bullet f, f$, all of which are algebraic polynomials, and because analytification is exact, the analytification of $\Gamma(D(x_{I}), \widetilde{\Omega_{R}^{j}(\log \mathscr{A})})$ is $\Gamma(D(x_{I}), \Omega_{X^{\anal}}(\log \mathscr{A}))$. (This is an affine statement: $D(x_{I}) \subseteq \mathbb{C}^{n}$.) Because $\Omega_{R}^{j}(\log \mathscr{A})$ is generated by homogeneous terms (as $\mathscr{A}$ is central), we deduce
\[
\Gamma(D^{+}(x_{I}), K^{j, \anal}) = \Gamma(D(x_{I}), \Omega_{X^{\anal}}^{j}(\log \mathscr{A}))_{-\iota_{E}(\omega)}
\]
where by $(-)_{-\iota_{E}(\omega)}$ we mean the homogeneous elements in the sense of Laurent series. This means
\begin{align*}
H^{t}(C^{\bullet}(\{D^{+}(x_{i})\}, K^{j, \anal})) 
    &= H^{t}(C^{\bullet}(\{D(x_{i})\}, \Omega_{X^{\anal}}^{j}(\log \mathscr{A}))_{-\iota_{E}(\omega)}) \\
    &= H_{\Cech}^{t}(\Omega_{X^{\anal}}^{j}(\log \mathscr{A}))_{-\iota_{E}(\omega)}.
\end{align*}
By \eqref{eqn - vanishing of analytification projective sheaf, also equal to Cech complex of analytification projective sheaf}, we conclude $0 = H_{\Cech}^{t}(\Omega_{X^{\anal}}^{j}(\log \mathscr{A}))_{-\iota_{E}(\omega)}$, as required.
\end{proof}

\section{The (un)Twisted Logarithmic Comparison Theorem}

This section is devoted to proving the Twisted Logarithmic Comparison Theorem for reduced hyperplane arrangements, subject to a relatively mild restriction on the weights in both the analytic (Theorem \ref{thm - analytic twisted lct}) and algebraic (Theorem \ref{thm - global algebraic twisted lct with homogeneous subcomplex q.i}) cases. The latter requires centrality, the former does not. Consequences are (all not necessarily central): the global analytic Twisted Logarithmic Comparison Theorem (Corollary \ref{cor - global twisted analytic lct}); the analytic (untwisted) Logarithmic Comparison Theorem (Corollary \ref{cor - analytic untwisted lct}); the algebraic (untwisted) Logarithmic Comparison Theorem (Corollary \ref{cor - analytic untwisted lct}). The last result positively answers Terao's conjecture (Conjecture \ref{conjecture -  algebraic lct for arrangements}). In sum, these result show that the cohomology of the complement in an arbitrary rank one local system can be completely understood by an appropriately twisted logarithmic de Rham complex.

The analytic Twisted Logarithmic Comparison Theorem is the main result insofar as the others follow with significantly less difficulty. The argument synthesizes analytic and algebraic techniques.  We utilize the analytic paradigm of Theorem 1.1 of \cite{CohomologyComplementFree}: we repeat their inductive set up (either using the rank of $\mathscr{A}$ as we do here or using Saito-holonomic induction as in loc. cit.) and exploit the relationships between the four consequent spectral sequences, some of which incorporate the inductive hypothesis. The argument hinges on showing that on the second page of one of these spectral sequences all but one column vanishes. In loc.\ cit.\ (or similarly in \cite{YuzvinskyWiensLTCTameArrangements}) homological criterion on the logarithmic differential forms greatly simplifies the first page. With no homological assumptions, we have no control over the first page's complexity. It is here that we utilize algebraic techniques and our bound on the Castelnuovo--Mumford regularity of logarithmic forms: using Corollary \ref{cor - specific weights at origin, acyclic complex of analytic Cech cohomology} and subsection 2.3 we can translate this data into our analytic situation, furnishing adequate control over vanishing on the second page. 

Before embarking note that in Proposition \ref{prop - all local systems can be computed}, we show that all rank one local systems on $U = X \setminus \mathscr{A}$ can be computed via these Twisted Logarithmic Comparison Theorems. This is not the case for the twisted Orlik--Solomon algebra, cf. \cite{CohenSuciuTangent}, \cite{SuciuTranslatedTori}, \cite{CohenTriplesofArrangements}. When $\mathscr{A}$ is central, Theorem \ref{thm - global algebraic twisted lct with homogeneous subcomplex q.i} and Remark \ref{rmk - finite dimensional lin alg for analytic stalk} imply these computations are explicit, finite dimensional linear algebra. Example \ref{ex - deleted B3 arrangement} exemplifies this: we compute the Betti numbers of $H^{\bullet}(U, \localSystem_{\boldsymbol{\beta}})$ for a $\localSystem_{\boldsymbol{\beta}}$ belonging to a translated component of the characteristic variety--precisely the components undetectable by Orlik--Solomon methods. 

\subsection{The Analytic Case} \enspace

Before we can prove Theorem \ref{thm - analytic twisted lct}, we require a significant amount of lemmas. First, the analytic Twisted Logarithmic Comparison can be checked stalk-wise. Using this, the following lemma says it is enough to check the analytic Twisted Logarithmic Comparison Theorem on sufficiently small Stein opens $V \subseteq X$. The proof is exactly the same as in Lemma 2.5 of \cite{CohomologyComplementFree} and is tantamount to the following: because higher sheaf cohomology of coherent modules over Stein opens vanishes, taking hypercohomology amounts to taking global sections.

\begin{lemma} \label{lemma - twisted lct determined by behavior at small Steins} (Lemma 2.5 \cite{CohomologyComplementFree})
Let $f = f_{1} \cdots f_{d}$ cut out a hyperplane arrangment $\mathscr{A}$ and consider the weights $\lambda_{1}, \dots, \lambda_{d}$ with their associated one-form $\omega = \sum_{k} \lambda_{k} df_{k} / f_{k}.$
Then the following are equivalent: 
\begin{enumerate}[label=(\alph*)]
    \item $(\Omega_{X}^{\bullet}(\log \mathscr{A}), \nabla_{\omega}) \xhookrightarrow{\qi} (\Omega_{X}^{\bullet}(\star \mathscr{A}), \nabla_{\omega})$;
    \item For each (sufficiently small) Stein open $V \subseteq X$ and each $p \geq 0$, the map induced by inclusion
    \[
    H^{p}(\Gamma(V, \Omega_{X}^{\bullet}(\log \mathscr{A})), \nabla_{\omega}) \xrightarrow[]{} H^{p}(\Gamma(V \setminus \mathscr{A}, \Omega_{X}^{\bullet}(\star \mathscr{A})), \nabla_{\omega})
    \]
    is an isomorphism.
\end{enumerate}
\end{lemma}
\begin{proof}
Exactly as in Lemma 2.5 of \cite{CohomologyComplementFree}.
\end{proof}

The majority of preliminary work is to recreate the inductive set-up from \cite{CohomologyComplementFree} in the twisted case. This will let information about smaller rank arrangements percolate to edges of a larger rank arrangement. Since we have $\omega \neq 0$, unlike loc.\ cit., there are several technical details that must be checked. 

First, we want to be able to study the Twisted Logarithmic Comparison at a point $\mathfrak{x}$ that only lies in a subset of the hyperplanes of $\mathscr{A}$. The next lemma says that such a point, we can throw out the summands $\lambda_{k} d \log f_{k}$ of $\omega$ corresponding to hyperplanes $\Var(f_{k})$ that do not pass through $\mathfrak{x}$. The argument uses an idea from Lemma 2.1 of \cite{KawaharaTwistedDeRham}.

\begin{lemma} \label{lemma - inductive step, removing fk terms from omega}
Suppose that $f = f_{1} \cdots f_{d}$ cuts out a reduced hyperplane arrangement $\mathscr{A}$. For weights $\lambda_{1}, \dots, \lambda_{r} \in \mathbb{C}$, we have the associated one form
\[
\omega = \sum_{k} \lambda_{k} \frac{df_{k}}{f_{k}}.
\]
For $\mathfrak{x} \in X$, let $f^{\prime} = \prod_{\{1 \leq k \leq d \mid \mathfrak{x} \in \Var(f_{k})\}} f_{k}$ and let
\[
\omega^{\prime} = \sum_{\{1 \leq k \leq d \mid \mathfrak{x} \in \Var(f_{k})\}} \lambda_{k} \frac{df_{k}}{f_{k}}.
\]
Then for small enough Stein opens $V \ni \mathfrak{x}$ we have isomorphisms
\begin{enumerate}[label=(\alph*)]
    \item $(\Omega_{V}^{\bullet}(\log f), \nabla_{\omega}) \xrightarrow[]{\simeq} (\Omega_{V}^{\bullet}(\log f^{\prime}), \nabla_{\omega^{\prime}})$;
    \item $(\Omega_{V}^{\bullet}(\star f), \nabla_{\omega}) \xrightarrow[]{\simeq} (\Omega_{V}^{\bullet}(\star f^{\prime}), \nabla_{\omega^{\prime}})$.
\end{enumerate}
These isomorphisms are compatible with $(\Omega_{V}^{\bullet}(\log f), \nabla_{\omega}) \xhookrightarrow{} (\Omega_{V}^{\bullet}(\star f), \nabla_{\omega}).$
\end{lemma}

\begin{proof}
It is enough to prove the assertion for $V$ the complement of all the hyperplanes $\Var(f_{j})$ that do not contain the point $\mathfrak{x}$. Observe:
\[
\omega^{\prime} = \omega - \sum_{\{1 \leq k \leq d \mid \mathfrak{x} \notin \Var(f_{k})\}} \lambda_{j} \frac{df_{k}}{f_{k}} = \omega - \sum_{\{1 \leq k \leq d \mid \mathfrak{x} \notin \Var(f_{k})\}} \lambda_{k} (d \log f_{k}).
\]
Rename every such $f_{j}$ not vanishing at $\mathfrak{x}$ as $u_{j}$. We construct isomorphisms
\begin{equation} \label{eqn - sequence of quasi-isomorphisms, removing lambda summands}
    (\Omega_{V}^{\bullet}(\log f), \nabla_{\omega}) \xrightarrow[]{\simeq} (\Omega_{V}^{\bullet}(\log f), \nabla_{\omega - \lambda_{1}(d \log u_{1})}) \xrightarrow[]{\simeq} \cdots \xrightarrow[]{\simeq} (\Omega_{V}^{\bullet}(\log f), \nabla_{\omega^{\prime}}).
\end{equation}
We construct the first map of \eqref{eqn - sequence of quasi-isomorphisms, removing lambda summands}. We may assume $\lambda_{1} \neq 0$. Consider the diagram
\[
\begin{tikzcd}
\Omega_{V}^{i}(\log f) \rar{u_{1}^{-\lambda_{1}} \cdot } \dar{\nabla_{\omega - \lambda_{1}(d \log u_{1})}} 
    & \Omega_{V}^{i}(\log f)  \dar{\nabla_{\omega}} \\
\Omega_{V}^{i+1}(\log f) \rar{u_{1}^{-\lambda_{1}} \cdot } 
    & \Omega_{V}^{i+1}(\log f).
\end{tikzcd}
\]
Since $u_{1}$ is a unit, $u_{1}^{-\lambda_{1}}$ is a unit as well, making the horizontal maps isomorphisms. A quick computation verifies the diagram commutes. Therefore multiplication by $u_{1}^{-\lambda_{1}}$ is an invertible chain map and we have constructed the first isomorphism of \eqref{eqn - sequence of quasi-isomorphisms, removing lambda summands}. Iterating the procedure demonstrates (a).

The same construction works for the twisted meromorphic de Rham complex, giving (b). And multiplication by $u_{1}^{\lambda}$ is compatible with including the twisted logarithmic de Rham complex into the meromorphic one.
\end{proof}

Second, when working at a point $\mathfrak{x}$ whose corresponding edge and subarrangement is of non-maximal rank, we want to be able to reduce the ambient dimension. This corresponds to, and is proved similarly, to Lemma 2.2 of \cite{CohomologyComplementFree}.

\begin{lemma} \label{lemma - divisor product, inductive step quasi-iso}
Suppose $f = f_{1} \cdots f_{d}$ cuts out a hyperplane arrangement $\mathscr{A}$. For weights $\lambda_{1}, \dots, \lambda_{d} \in \mathbb{C}$, we have the associated one form
\[
\omega = \sum_{k} \lambda_{k} \frac{df_{k}}{f_{k}}.
\]
Assume that on some Stein manifold $V \subseteq X$, the divisor $\Div(f)$ is locally a product $(V, \Div(f)) = (\mathbb{C} \times V^{\prime}, \mathbb{C} \times \Div(f^{\prime}))$ for $V^{\prime}$ a Stein manifold of smaller dimension. Let $\pi: V \to V^{\prime}$ be the projection, let $f^{\prime}$ and $f_{k}^{\prime}$ $\in \mathscr{O}_{V^{\prime}}$ satisfy $\Div(f) = \mathbb{C} \times \Div(f^{\prime})$ and $\mathbb{C} \times \Div(f_{k}^{\prime}) = \Div(f_{k}))$, and let
\[
\omega^{\prime} = \sum_{k} \lambda_{k} \frac{df^{\prime}_{k}}{f^{\prime}_{k}}.
\]
Then
\begin{enumerate}[label=(\alph*)]
    \item $\pi^{-1}(\Omega_{V^{\prime}}^{\bullet}(\log f^{\prime}), \nabla_{\omega^{\prime}}) \to (\Omega_{V}^{\bullet}(\log f), \nabla_{\omega})$ is a quasi-isomorphism;
    \item $\pi^{-1}(\Omega_{V^{\prime}}^{\bullet}(\star f^{\prime}), \nabla_{\omega^{\prime}}) \to (\Omega_{V}^{\bullet}(\star f), \nabla_{\omega})$ is a quasi-isomorphism.
\end{enumerate}
\end{lemma}
\begin{proof}
Pick local coordinates $(t, x^{\prime})$ for $V = \mathbb{C} \times V^{\prime}$. And note that $f_{k} = u_{k} f_{k}^{\prime}$ where $u_{k}$ is a unit. Arguing as in Lemma \ref{lemma - inductive step, removing fk terms from omega}, $(\Omega_{V}^{\bullet}(\log f^{\prime}), \nabla_{\omega^{\prime}}) \xrightarrow[]{\simeq} (\Omega_{V}^{\bullet}(\log f), \nabla_{\omega})$, where here $f^{\prime}$ and $\omega^{\prime}$ are viewed as functions and forms over $V$. (Note that we are removing the summands $\lambda_{k} d \log u_{k}$ from $\omega$). Now weight the local coordinates $(t, x^{\prime})$ by giving $t$ weight one, each $x_{i}^{\prime}$ weight zero, $dx_{i}^{\prime}$ weight zero, and $dt$ weight one. Let $\chi = t \partial_{t}$ be a logarithmic derivation on $f^{\prime}$. Then we can argue as in Lemma \ref{lemma - twisted analytic log qi to homogeneous subcomplex}: the twisted Lie derivative $\nabla_{\omega^{\prime}}(\iota_{\chi}) + \iota_{\chi}(\nabla_{\omega^{\prime}})$ induces a quasi-isomorphism $(\Omega_{V}^{\bullet}(\log f^{\prime})_{0}, \nabla_{\omega^{\prime}}) \xhookrightarrow{\qi} (\Omega_{V}^{\bullet}(\log f^{\prime}), \nabla_{\omega^{\prime}})$, where the weight zero subcomplex is defined by the aforementioned weighting. (Note that $\iota_{\chi}(\omega) = 0$ and the previous claim can be checked stalk-wise.) This weight zero subcomplex is $\pi^{-1}(\Omega_{V^{\prime}}^{\bullet}(\log f^{\prime}), \nabla_{\omega^{\prime}})$, giving (a). The proof of (b) is similar. 
\end{proof}

With the preparations complete, we can prove the main theorem. The infrastructure is that of Theorem 1.1 of \cite{CohomologyComplementFree} but: we have to be careful manipulating the form $\omega$; the \v{C}ech cohomology/local cohomology argument (Step 3) is much more subtle. To sketch the recipe for the latter, we: use subsection 2.3 to reduce to a ``homogeneous'' subcomplex of a twisted complex of analytic \v{C}ech cohomology modules; use the GAGA considerations of Corollary \ref{cor - specific weights at origin, acyclic complex of analytic Cech cohomology} to show the vanishing of this subcomplex on the nose is dictated by algebraic Castelnuovo--Mumford regularity bounds; use Theorem \ref{thm - computing CM regularity of log forms} to bound the regularity.

\begin{theorem} \label{thm - analytic twisted lct}
(Analytic Twisted Logarithmic Comparison Theorem) Let $f = f_{1} \cdots f_{d} \in R$ cut out a reduced hyperplane arrangement $\mathscr{A}$. Suppose that $\lambda_{1}, \dots, \lambda_{d} \in \mathbb{C}$ are weights such that, for each edge $E \in \mathscr{L}(\mathscr{A})$,
\begin{equation} \label{eqn - analytic twisted lct weight condition}
\sum_{\{1 \leq k \leq d \mid E \subseteq \Var(f_{k})\}} \lambda_{k} \notin \mathbb{Z}_{\geq \min\{2, \rank(E)\}}.
\end{equation}
Let $\omega = \sum_{k} \lambda_{k} \frac{df_{k}}{f_{k}}$ be
the logarithmic one form determined by the $\{\lambda_{k}\}$. Then the analytic Twisted Logarithmic Comparison Theorem with respect to $\{\lambda_{k}\}$ holds:
\begin{equation} \label{eqn - twisted lct statement}
(\Omega_{X}^{\bullet}(\log \mathscr{A}), \nabla_{\omega}) \xhookrightarrow{\qi} (\Omega_{X}^{\bullet}(\star \mathscr{A}), \nabla_{\omega}) \quad (= \derivedR j_{\star} \localSystem_{\Exp(\boldsymbol{\lambda})}).
\end{equation}
\end{theorem}

\begin{proof}
While the result is well known, we defer giving details on the derived category identification ``$(= \derivedR j_{\star} \localSystem_{\Exp(\boldsymbol{\lambda})})$'' until Proposition \ref{prop - twisted meromorphic de rham quasi-iso to local system}.

\emph{Step 0}: Reductions and the Inductive Scheme.

We proceed stalk-wise and induce on $\dim X$. When $\dim X = 1$ the result is well known and easy to check. So assume $\dim X = n$ and the claim holds for all reduced arrangements in $\mathbb{C}^{p}$ for $p < n$. Let $\mathfrak{x} \in X$ and $E$ the smallest edge containing $\mathfrak{x}$. If the subarrangement $f^{\prime}$ consisting of all hyperplanes containing $\mathfrak{x}$ is of rank less than $\dim X$, then the induction hypothesis applies. For if $V$ is a small Stein containing $\mathfrak{x}$ and
\[
\omega^{\prime} = \sum_{\{1 \leq k \leq d \mid E \subseteq \Var(f_{k})\}} \lambda_{k} \frac{df_{k}}{f_{k}},
\]
then Lemma \ref{lemma - inductive step, removing fk terms from omega} says \eqref{eqn - twisted lct statement} holds at $\mathfrak{x}$ is if and only if $(\Omega_{V}^{\bullet}(\log f^{\prime}), \nabla_{\omega^{\prime}}) \xhookrightarrow{\qi} (\Omega_{V}^{\bullet}(\star f^{\prime}), \nabla_{\omega^{\prime}})$ holds at $\mathfrak{x}$. Since $f^{\prime}$ is of rank less than $\dim X$ at $\mathfrak{x}$ it is a product, and we use Lemma \ref{lemma - divisor product, inductive step quasi-iso} to invoke the inductive hypothesis. Note that $\omega^{\prime}$ satisfies its version of \eqref{eqn - analytic twisted lct weight condition} since $\omega$ satisfies \eqref{eqn - analytic twisted lct weight condition} itself.

If the subarrangement $f^{\prime}$ at $\mathfrak{x}$ has rank $n$, we may change coordinates and assume that $\mathfrak{x} = 0$ and $f^{\prime}$ is central and essential. By Lemma \ref{lemma - inductive step, removing fk terms from omega}, the quasi-isomorphism \eqref{eqn - twisted lct statement} holds at $\mathfrak{x}$ exactly when the corresponding Twisted Logarithmic Comparison Theorem holds for $f^{\prime}$ and $\omega^{\prime}$ at $0$ (in these new coordinates).

Therefore it suffices to prove the following: if $\mathscr{A}$ is a central, essential, reduced hyperplane arrangement such that the \eqref{eqn - twisted lct statement} holds at all points outside of $0$, then \eqref{eqn - twisted lct statement} holds at $0$. By Lemma \ref{lemma - twisted lct determined by behavior at small Steins} we may verify this (equivalently) by checking \eqref{eqn - twisted lct statement} holds at a small open Stein $V \ni 0$.

\emph{Step 1}: Setting up the Spectral Sequences.

For the rest of the proof, we inhabit the spectral sequence infrastructure from the proof of Theorem 1.1 of \cite{CohomologyComplementFree} but refine it using the results of our previous sections. Let $V$ be small Stein open about $0$. Define the Stein opens $V_{i} = V \setminus \{x_{i} = 0\}$ and $V_{i}^{\prime} = V_{i} \setminus (V \cap \mathscr{A})$. We have Stein open covers $\{V_{i}\}$ and $\{V_{i}^{\prime}\}$ of $V \setminus 0$ and $V \setminus (V \cap \mathscr{A})$ respectively. Consider the following two double complexes with objects:
\begin{align*}
K^{p, q} &= \bigoplus_{1 \leq i_{1} \leq \cdots \leq i_{q} \leq n} \Gamma(\bigcap_{i=0}^{q} V_{i_{j}}, \Omega_{X}^{p}(\log \mathscr{A})); \\
\widetilde{K}^{p, q} &= \bigoplus_{1 \leq i_{1} \leq \cdots \leq i_{q} \leq n} \Gamma(\bigcap_{i=0}^{q} V_{i_{j}}^{\prime}, \Omega_{X}^{p}(\star \mathscr{A})).
\end{align*}
Each complex is positioned in the lattice $\mathbb{Z}^{2}$ so that, for example, $K^{p,q}$, occurs in row $p$ and column $q$. $K^{\bullet, \bullet}$ has horizontal differential the \v{C}ech differential and vertical differential $\nabla_{\omega}$; $\widetilde{K}^{\bullet, \bullet}$ has the same. And restriction commutes with the differentials, giving a map $\rho_{0}^{\bullet, \bullet}: K^{\bullet, \bullet} \to \widetilde{K}^{\bullet, \bullet}$. 

We will consider the resulting four spectral sequences arising from these double complexes using similar notation as in \cite{CohomologyComplementFree}. As shorthand, $^{\prime}(-)$ denotes the spectral sequence first taking cohomology vertically and $^{\prime \prime}(-)$ first taking cohomology horizontally. To be precise: 
\begin{align*}
    ^{\prime}E_{1}^{p,q} &= \bigoplus_{1 \leq i_{0} \leq \cdots \leq i_{q} \leq n} H^{p} (\Gamma(\bigcap_{j=0}^{q} V_{i_{j}},  (\Omega_{X}^{\bullet}(\log \mathscr{A}), \nabla_{\omega}))); \\
    ^{\prime \prime}E_{1}^{p,q} &= H^{q}(V \setminus 0, \Omega_{X}^{p}(\log \mathscr{A})); \\
    ^{\prime}\widetilde{E}_{1}^{p,q} &= \bigoplus_{1 \leq i_{0} \leq \cdots \leq i_{q} \leq n} H^{p} (\Gamma(\bigcap_{j=0}^{q} V_{i_{j}}^{\prime} , (\Omega_{X}^{\bullet}(\star \mathscr{A}), \nabla_{\omega}))); \\
    ^{\prime \prime}\widetilde{E}_{1}^{p,q} &= H^{q}(V \setminus (V \cap \mathscr{A}), \Omega_{X}^{p}(\star \mathscr{A})). 
\end{align*}

\emph{Step 2}: The $^{\prime}(-)$ Spectral Sequences.

By the inductive set up, $\eqref{eqn - twisted lct statement}$ holds at all points outside the origin. This means that the aformentioned natural map: 
\[
\rho_{1}^{\bullet, \bullet}: \enspace ^{\prime}E_{1}^{\bullet, \bullet} \xrightarrow[]{} \enspace ^{\prime} \widetilde{E}_{1}^{\bullet, \bullet}
\]
is an isomorphism. As $\rho_{0}^{\bullet, \bullet}$ commutes with the differentials, this isomorphism continues on further pages. Since the spectral sequences converge we have
\begin{equation} \label{eqn- first spectral sequences isos continuing on infty page}
\rho_{\infty}^{\bullet, \bullet}:  \enspace ^{\prime}E_{\infty}^{\bullet, \bullet} \xrightarrow[]{\simeq} \enspace ^{\prime}\widetilde{E}_{\infty}^{\bullet, \bullet}.
\end{equation}

\emph{Step 3}: The $^{\prime \prime}(-)$ Spectral Sequences. 

We first consider $^{\prime \prime}\widetilde{E}^{\bullet, \bullet}$. Because $V \setminus (V \cap \mathscr{A})$ is Stein and $\Omega_{X}^{p}(\star f) |_{U}$ is coherent, the higher sheaf cohomology $^{\prime \prime}\widetilde{E}_{1}^{p,q}$ vanishes for $q > 0$. So this page's nonzero terms lie on the $y$-axis, the spectral sequence converges on the second page where the nonzero terms still lie on the $y$-axis, and
\begin{equation} \label{eqn - second page, meromorphic spectral sequence only one column}
    ^{\prime \prime}\widetilde{E}_{2}^{p,0} = H^{p}(\Gamma(V \setminus (V \cap \mathscr{A}), (\Omega_{X}^{\bullet}(\star \mathscr{A}), \nabla_{\omega}))).
\end{equation}

We claim $^{\prime \prime}E^{\bullet, \bullet}$ has the same vanishing behavior. Combine the long exact sequence of cohomology supported on $\{0\}$
\[
\to H^{q}_{0}(V, \Omega_{X}^{p}(\log \mathscr{A})) \to 
    H^{q}(V, \Omega_{X}^{p}(\log \mathscr{A})) \to 
        H^{q}(V \setminus 0, \Omega_{X}^{p}(\log \mathscr{A})) \to
\]
with the higher cohomological vanishing of $H^{q}(V, \Omega_{X}^{p}(\log \mathscr{A}))$ for $q \geq 1$, to get a natural isomorphism
\[
^{\prime \prime}E_{1}^{p,q} \simeq H_{0}^{q+1}(V,\Omega_{X}^{p}(\log \mathscr{A}))
\]
for all $q \geq 1$. By excision, 
\[
H_{0}^{q+1}(V,\Omega_{X}^{p}(\log \mathscr{A})) \simeq H_{0}^{q+1}(X, \Omega_{X}^{p}(\log \mathscr{A})).
\]
Again use the long exact sequence of cohomology supported on $\{0\}$ to get
\[
H^{q}(X \setminus 0, \Omega_{X}^{p}(\log \mathscr{A})) \simeq H_{0}^{q+1}(X, \Omega_{X}^{p}(\log \mathscr{A}))
\]
for all $q \geq 1$. Since $\{D(x_{i})\}_{1 \leq i \leq n}$ is an open cover of $X \setminus 0$, and since every $D(x_{i})$ along with every intersection of elements of the $\{D(x_{i})\}$ is a Stein manifold, the \v{C}ech complex $C^{\bullet}(\{D(x_{i})\}, \Omega_{X}^{p}(\log \mathscr{A}))$ computes the sheaf cohomology $H^{\bullet}(X \setminus 0, \Omega_{X}^{p}(\log \mathscr{A}))$. 

Combining these equivalences and using the notation of subsection 2.3 yields
\[
^{\prime \prime}E_{1}^{p,q} \simeq H^{q}(X \setminus 0, \Omega_{X}^{p}(\log \mathscr{A})) \simeq H_{\Cech}^{q}(\Omega_{X}^{p}(\log \mathscr{A}))
\]
for all $q \geq 1$. The second page $^{\prime \prime}E_{2}^{\bullet, \bullet}$ arises from the first by taking cohomology vertically with respect to the differential $\nabla_{\omega}$. So
\begin{equation} \label{eqn - hard second page, converted into twisted cohomology of Cech cohomology}
^{\prime \prime}E_{2}^{p,q} \simeq H^{p}(H_{\Cech}^{q}(\Omega_{X}^{\bullet}(\log \mathscr{A})), \nabla_{\omega})
\end{equation}
for all $q \geq 1.$ Now our condition on the weights \eqref{eqn - analytic twisted lct weight condition} considered at the origin become $\sum \lambda_{k} \notin \mathbb{Z}_{\geq \min\{2, \rank(\mathscr{A})\}}$. So we may apply Corollary \ref{cor - specific weights at origin, acyclic complex of analytic Cech cohomology} to \eqref{eqn - hard second page, converted into twisted cohomology of Cech cohomology}, deducing $^{\prime \prime}E_{2}^{p,q}$ vanishes for $q \geq 1$. Therefore the nonzero terms of the second page $^{\prime \prime}E_{2}^{\bullet, \bullet}$ all lie on the $y$-axis, the spectral sequence converges on this second page, and 
\[
^{\prime \prime}E_{2}^{p, 0} = H^{p}(\Gamma(V \setminus 0, (\Omega_{X}^{\bullet}(\log \mathscr{A}), \nabla_{\omega}))) .
\]
Since $X = \mathbb{C}^{n}$ with $n \geq 2$ by assumption, the reflexivity of $\Omega_{X}^{j}(\log \mathscr{A})$ (see, for example, Proposition 1.5 \cite{MondLogarithmicDifferentialFormsAnd}) implies that the global sections of $V \setminus 0$ are the global sections of $V$, meaning the only nonzero terms of the second page $^{\prime \prime}E_{2}^{\bullet, \bullet}$ are
\begin{equation} \label{eqn - second page, log spectral sequence only one column}
^{\prime \prime}E_{2}^{p, 0} = H^{p}(\Gamma(V, (\Omega_{X}^{\bullet}(\log \mathscr{A}), \nabla_{\omega}))) .
\end{equation}

\emph{Step 4}: The End.

Because $^{\prime \prime}E_{2}^{\bullet,\bullet}$ consists of only one column on the $y$-axis, convergence to $H^{\bullet}\text{Tot}(K^{\bullet,\bullet})$ happens on the second page. And because there is only one column, the induced filtration on $H^{\bullet}\text{Tot}(K^{\bullet,\bullet})$ is trivial. Using \eqref{eqn - second page, log spectral sequence only one column} we have
\[
H^{p}(\Gamma(V, (\Omega_{X}^{\bullet}(\log \mathscr{A}), \nabla_{\omega}))) = \enspace ^{\prime \prime}E_{2}^{p, 0} \enspace = H^{p}\text{Tot}(K^{\bullet,\bullet}).
\]
Similarly, $^{\prime \prime}\widetilde{E}_{2}^{\bullet,\bullet}$ has only one column on the $y$-axis, so convergence to $H^{\bullet} \text{Tot}(\widetilde{K}^{\bullet, \bullet})$ happens on the second page and the isomorphism has no graded data. Using \eqref{eqn - second page, meromorphic spectral sequence only one column} we have
\[
H^{p}(\Gamma(V \setminus (V \cap \mathscr{A}), (\Omega_{X}^{\bullet}(\star \mathscr{A}), \nabla_{\omega}))) = \enspace ^{\prime \prime}\widetilde{E}_{2}^{p, 0} \enspace = H^{p}\text{Tot}(\widetilde{K}^{\bullet,\bullet}).
\]

On the other hand, the isomorphism $\enspace ^{\prime}E_{\infty}^{\bullet, \bullet} \xrightarrow[]{\simeq} \enspace ^{\prime}\widetilde{E}_{\infty}^{\bullet, \bullet}$ from \eqref{eqn- first spectral sequences isos continuing on infty page} gives (ignoring the extra graded data) the natural isomorphism 
\[
H^{p}\text{Tot}(K^{\bullet,\bullet}) \simeq H^{p}\text{Tot}(\widetilde{K}^{\bullet,\bullet}).
\]

Putting the three displayed equations of this step together produces
\[
H^{p}(\Gamma(V, (\Omega_{X}^{\bullet}(\log \mathscr{A}), \nabla_{\omega}))) \simeq H^{p}(\Gamma(V \setminus (V \cap \mathscr{A}), (\Omega_{X}^{\bullet}(\star \mathscr{A}), \nabla_{\omega}))),
\]
which is what we required.
\end{proof}

\begin{remark}
The base case of the induction, $\mathscr{A} = \Var(x)$, requires $\mathbb{Z}_{\geq 1}$ in \eqref{eqn - analytic twisted lct weight condition}; the base case fails if $\mathbb{Z}_{\geq 2}$ is used in \eqref{eqn - analytic twisted lct weight condition} instead.
\end{remark}

Let us record two immediate corollaries. As it is of independent interest, we state the untwisted version:

\begin{corollary} \label{cor - analytic untwisted lct} (Analytic Logarithmic Comparison Theorem)
For a reduced hyperplane arrangement $\mathscr{A}$, the analytic Logarithmic Comparison Theorem holds:
\[
(\Omega_{X}^{\bullet}(\log \mathscr{A}), d) \xhookrightarrow{\qi} (\Omega_{X}^{\bullet}(\star \mathscr{A}), d) \quad (= \derivedR j_{\star} \mathbb{C}_{U})
\]
\end{corollary}

\begin{proof}
Use Theorem \ref{thm - analytic twisted lct} and set all the weights $\{\lambda_{k}\}$ to zero.
\end{proof}

We also obtain a global version as in Corollary 1.4 of \cite{CohomologyComplementFree}. We continue to not require centrality.

\begin{corollary} \label{cor - global twisted analytic lct}
Let $f = f_{1} \cdots f_{d} \in R$ be a reduced hyperplane arrangement $\mathscr{A}$ equipped with a factorization into linear forms. Suppose that $\lambda_{1}, \dots, \lambda_{d} \in \mathbb{C}$ are weights such that, for each edge $E \in \mathscr{L}(\mathscr{A})$,
\begin{equation*} 
\sum_{\{1 \leq k \leq d \mid E \subseteq \Var(f_{k})\}} \lambda_{k} \notin \mathbb{Z}_{\geq \min\{2, \rank(E)\}}.
\end{equation*}
Let $\omega = \sum_{k} \lambda_{k} \frac{df_{k}}{f_{k}}$ be
the logarithmic one form determined by the $\{\lambda_{k}\}$. Then the global analytic Twisted Logarithmic Comparison Theorem holds:
\[
(\Gamma(X, \Omega_{X}^{\bullet}(\log \mathscr{A})), \nabla_{\omega}) \xhookrightarrow{\qi} (\Gamma(X, \Omega_{X}^{\bullet}(\star \mathscr{A})), \nabla_{\omega}) \quad (\simeq H^{\bullet}(U, \localSystem_{\Exp(\boldsymbol{\lambda})}),
\]
where ``$(\simeq  H^{\bullet}(U, \localSystem_{\exp(\boldsymbol{\lambda})}))$'' means there are cohomological isomorphisms.
\end{corollary}

\begin{proof}
The quasi-isomorphism induced by inclusion in Theorem \ref{thm - analytic twisted lct} gives isomorphisms in hypercohomology:
\begin{equation} \label{eqn - global analytic twisted lct, taking hypercohomology}
\textbf{H}^{t}(\Omega_{X}^{\bullet}(\log \mathscr{A}), \nabla_{\omega}) \xrightarrow[]{\simeq} \textbf{H}^{t}(\Omega_{X}^{\bullet}(\star \mathscr{A}), \nabla_{\omega}).
\end{equation}
Since $X$ and $U = X \setminus \mathscr{A}$ are Stein and $\Omega_{X}^{j}(\log \mathscr{A})$ and $\Omega_{U}^{j}$ are coherent, their higher sheaf cohomology vanishes. So the hypercohomology in \eqref{eqn - global analytic twisted lct, taking hypercohomology} amounts to taking global sections and then taking cohomology. This gives the promised quasi-isomorphism. As for the claim involving ``$(\simeq  H^{\bullet}(U, \localSystem_{\exp(\boldsymbol{\lambda})}))$'', it is well known that $\textbf{H}^{t} (\derivedR j_{\star} \localSystem_{\Exp(\boldsymbol{\lambda})}) = H^{t}(U, \localSystem_{\Exp(\boldsymbol{\lambda})})$.
\end{proof}

\subsection{The Algebraic Case} \enspace

Now we turn to the problem of the algebraic Twisted Logarithmic Comparison Theorem. We cannot always repeat the proof of Theorem \ref{thm - analytic twisted lct} because of Lemma \ref{lemma - inductive step, removing fk terms from omega}: this construction requires multiplying by $f_{k}^{\lambda_{k}}$ at places where $f_{k}$ is a unit. Because $\lambda_{k}$ may not be integral, in the algebraic category this may be nonsensical.

Nevertheless, when $\mathscr{A}$ is central we can deduce the algebraic Twisted Logarithmic Comparison Theorem from the analytic one. We highlight that the cohomology of the complement $U$ in the attached local system is computed by a particular homogeneous subcomplex of the twisted logarithmic de Rham complex. So computing the complement's cohomology is a finite dimensional linear algebra problem.

\begin{theorem} \label{thm - global algebraic twisted lct with homogeneous subcomplex q.i} (Algebraic Twisted Logarithmic Comparison Theorem)
Let $f = f_{1} \cdots f_{d}$ cut out a central, reduced hyperplane arrangement $\mathscr{A}$ and let $\lambda_{1}, \dots, \lambda_{d} \in \mathbb{C}$ be weights such that, for each edge $E$,
\[
\sum_{\{1 \leq k \leq d \mid E \subseteq \Var(f_{k})\}} \lambda_{k} \notin \mathbb{Z}_{\geq \min\{2, \rank(E)\}}.
\]
Furthermore, let $\omega = \sum_{k} \lambda_{k} df_{k} / f_{k}$. Then the algebraic Twisted Logarithmic Comparison Theorem holds along with an additional quasi-isomorphism:
\[
(\Omega_{R}^{\bullet}(\log \mathscr{A})_{-\iota_{E}(\omega)}, \nabla_{\omega}) \xhookrightarrow{\qi} (\Omega_{R}^{\bullet}(\log \mathscr{A}), \nabla_{\omega}) \xhookrightarrow{\qi} (\Omega_{R}^{\bullet}(\star \mathscr{A}), \nabla_{\omega}) \quad (\simeq  H^{\bullet}(U, \localSystem_{\Exp(\boldsymbol{\lambda})})).
\]
Here $(\simeq  H^{\bullet}(U, \localSystem_{\Exp(\boldsymbol{\lambda})}))$ means there are isomorphisms on the level of cohomology and $(\Omega_{R}^{\bullet}(\log \mathscr{A})_{-\iota_{E}(\omega)}, \nabla_{\omega})$ is the homogeneous subcomplex of degree $- \iota_{E}(\omega) = -( \lambda_{1} + \cdots + \lambda_{d})$, which is a complex of finite dimensional $\mathbb{C}$-vector spaces.
\end{theorem}
\begin{proof}
The claim involving $(\simeq  H^{\bullet}(U, \localSystem_{\exp(\boldsymbol{\lambda})}))$ follows from the algebraic de Rham Theorems of Deligne and Grothendieck, cf. \cite{OrlikHypergeometricIntegralsAndArrangements}. That, when $\mathscr{A}$ is central, the map $(\Omega_{R}^{\bullet}(\log \mathscr{A})_{-\iota_{E}(\omega)}, \nabla_{\omega}) \xhookrightarrow{} (\Omega_{R}^{\bullet}(\log \mathscr{A}), \nabla_{\omega})$ is a quasi-isomorphism is the first part of Proposition \ref{prop - homogeneous subcomplex quasi iso}. As for the finite dimensional claim: since $R$ and $\Omega_{R}^{j}(\log \mathscr{A})$ do not have elements of arbitrary large negative degree, the $-\iota_{E}(\omega)$-homogeneous subspace $\Omega_{R}^{j}(\log \mathscr{A})_{-\iota_{E}(\omega)}$ is finite dimensional. 

It remains to validate the purported quasi-isomorphism $(\Omega_{R}^{\bullet}(\log \mathscr{A}), \nabla_{\omega}) \xhookrightarrow{\qi} (\Omega_{R}^{\bullet}(\star \mathscr{A}), \nabla_{\omega})$. Consider the following commutative diagram:
\begin{equation} \label{eqn - comm diagram analytic to algebraic, homogeneous subcomplexes}
\begin{tikzcd}
(\Omega_{X, 0}^{j}(\log \mathscr{A}), \nabla_{\omega}) \rar[hookrightarrow]
    & (\Omega_{X, 0}^{j}(\star \mathscr{A}), \nabla_{\omega}) \\
(\Omega_{X, 0}^{j}(\log \mathscr{A})_{-\iota_{E}(\omega)}, \nabla_{\omega}) \rar[hookrightarrow] \uar[hookrightarrow]{\qi}
    & (\Omega_{X, 0}^{j}(\star \mathscr{A})_{-\iota_{E}(\omega)}, \nabla_{\omega}) \uar[hookrightarrow]{\qi} \\
(\Omega_{R}^{\bullet}(\log \mathscr{A})_{-\iota_{E}(\omega)}, \nabla_{\omega}) \rar[hookrightarrow] \uar[equal] \dar[hookrightarrow]{\qi}
    & (\Omega_{R}^{\bullet}(\star \mathscr{A})_{-\iota_{E}(\omega)}, \nabla_{\omega}) \uar[equal] \dar[hookrightarrow]{\qi} \\
(\Omega_{R}^{\bullet}(\log \mathscr{A}), \nabla_{\omega}) \rar[hookrightarrow]
    & (\Omega_{R}^{\bullet}(\star \mathscr{A}), \nabla_{\omega}).
\end{tikzcd}
\end{equation}
The displayed quasi-isomorphisms are justified by Lemma \ref{lemma - twisted analytic log qi to homogeneous subcomplex} and Proposition \ref{prop - homogeneous subcomplex quasi iso}. As for displayed equalities, let $\eta = \sum_{p} \eta_{p} \in \Omega_{X,0}^{j}(\log \mathscr{A})$ be the homogeneous decomposition of $\eta$ as in subsection 2.3, though here this arises from, up to a pole along $f$, convergent formal power series in $\mathbb{C}[[x_{1}, \dots, x_{n}]]$ as opposed to a Laurent series. If $\eta$ is homogeneous of weight $-\iota_{E}(\omega)$, then $\eta = \eta_{-\iota_{E}(\omega)}$, the latter of which is, up to its pole along $\mathscr{A}$, an algebraic $j$-form. So $\eta_{-\iota_{E}(\omega)} \in \Omega_{R}^{j}(\log \mathscr{A})_{-\iota_{E}(\omega)}$. The same story applies for the non-logarithmic displayed equality. So this diagram concludes the proof: Theorem \ref{thm - analytic twisted lct} implies the topmost horizontal map is a quasi-isomorphism, whence the bottom-most horizontal map is a quasi-isomorphism.
\end{proof}

\begin{remark} \label{rmk-whyCentralityForReferee}
    The proof of the Theorem \ref{thm - global algebraic twisted lct with homogeneous subcomplex q.i} nontrivially exploits the centrality assumption. Without centrality $\Omega_R^j(\log \mathscr{A})$ may not have a decomposition into a direct sum of homogeneous components, as the homogeneous terms of a logarithmic $j$-form may not be logarithmic. Consequently, $(\Omega_{R}^{\bullet}(\log \mathscr{A})_{-\iota_{E}(\omega)}, \nabla_{\omega})$ may be relatively vacuous, preventing any form of \eqref{eqn - comm diagram analytic to algebraic, homogeneous subcomplexes} holding and in turn obstructing our method of deducing Theorem \ref{thm - global algebraic twisted lct with homogeneous subcomplex q.i} from Theorem \ref{thm - analytic twisted lct}. For similar reasons, one should not expect the quasi-isomorphism $(\Omega_{R}^{\bullet}(\log \mathscr{A})_{-\iota_{E}(\omega)}, \nabla_{\omega}) \xhookrightarrow{\qi} (\Omega_{R}^{\bullet}(\log \mathscr{A}), \nabla_{\omega})$ of Theorem \ref{thm - global algebraic twisted lct with homogeneous subcomplex q.i} to hold (or be sensible) without a centrality assumption. On the other hand, the quasi-isomorphism $(\Omega_{R}^{\bullet}(\log \mathscr{A}), \nabla_{\omega}) \xhookrightarrow{\qi} (\Omega_{R}^{\bullet}(\star \mathscr{A}), \nabla_{\omega})$ of Theorem \ref{thm - global algebraic twisted lct with homogeneous subcomplex q.i} \emph{may} hold without centrality. One could try to check this at each maximal ideal, where $\mathscr{A}$ can be assumed to be central, but outside the analytic category technical issues related to $\nabla_\omega$ and Lemma \ref{lemma - inductive step, removing fk terms from omega} appear.
\end{remark}

This paper's original objective was to solve Terao's Conjecture \ref{conjecture -  algebraic lct for arrangements}, that is, to prove the algebraic (untwisted) Logarithmic Comparison Theorem. (This evolved out of Conjecture 3.1 of \cite{TeraoLCTforArrangementsConjecture}.) We can now do this quite easily, and, for aesthetic reasons, we include Brieskorn's Theorem in its statement. Note that we do not require centrality, which forces us to do more work.

\begin{corollary} \label{cor - algebraic untwisted lct}
(Algebraic Logarithmic Comparison Theorem) Let $\mathscr{A}$ be a reduced hyperplane arrangement. Then the algebraic Logarithmic Comparison Theorem holds giving a sequence of quasi-isomorphisms
\begin{equation} \label{eqn - algebraic untwisted lct statement in cor}
A^{\bullet}(\mathscr{A}) \xhookrightarrow{\qi} (\Omega_{R}^{\bullet}(\log \mathscr{A}), d) \xhookrightarrow{\qi} (\Omega_{R}^{\bullet}( \star \mathscr{A}), d) \quad (\simeq H^{\bullet}(U, \mathbb{C}_{U})).
\end{equation}
Moreover, $A^{j}(\mathscr{A}) = H^{j}(\Omega_{R}^{\bullet}(\log \mathscr{A}), d)$.
\end{corollary}

\begin{proof}
Brieskorn's Theorem says the composition of maps in \eqref{eqn - algebraic untwisted lct statement in cor} is a quasi-isomorphism, so if we show the algebraic Logarithmic Comparison Theorem is true we can conclude the first map in \eqref{eqn - algebraic untwisted lct statement in cor} is a quasi-isomorphism. From this, the claim $A^{j}(\mathscr{A}) = H^{j}(\Omega_{R}^{\bullet}(\log \mathscr{A}), d)$ is easy: the implicit differential of exterior differentiation on $A^{\bullet}(\mathscr{A})$ is trivial.

It remains to prove the algebraic Logarithmic Comparison Theorem. When $\mathscr{A}$ is central this follows from Theorem \ref{thm - global algebraic twisted lct with homogeneous subcomplex q.i} by setting the weights to zero.

For non-central $\mathscr{A}$, first note that we can define a sheaf-theoretic version of the algebraic twisted logarithmic and rational de Rham complexes: $(\Omega_{X_{\alg}}^{\bullet}(\log \mathscr{A}), \nabla_{\omega})$ and $(\Omega_{X_{\alg}}^{\bullet}(\star \mathscr{A}), \nabla_{\omega})$ respectively. These are complexes of $\mathscr{O}_{X_{\alg}}$-modules. If we prove
\begin{equation} \label{eqn - sheaf theoretic algebraic lct in proof}
(\Omega_{X_{\alg}}^{\bullet}(\log \mathscr{A}), \nabla_{\omega}) \xhookrightarrow{\qi} (\Omega_{X_{\alg}}^{\bullet}(\star \mathscr{A}), \nabla_{\omega}),
\end{equation}
then we can deduce that taking global sections preserves the quasi-isomorphism by arguing as in Corollary \ref{cor - global twisted analytic lct}. The proof of Theorem \ref{thm - analytic twisted lct} gives a proof of \eqref{eqn - sheaf theoretic algebraic lct in proof} once you replace every occurrence of ``Stein'' with ``affine.'' Indeed, note that: Lemma \ref{lemma - twisted lct determined by behavior at small Steins} holds after replacing ``Stein'' with ``affine;'' Lemma \ref{lemma - inductive step, removing fk terms from omega} is never invoked since $\omega = 0$ in our case; we may still use Lemma \ref{lemma - divisor product, inductive step quasi-iso} since the change of coordinates is linear, and hence algebraic; in Step 3 of the proof, $E_{1}^{p,q} \simeq H_{0}^{q+1}(V, \Omega_{X}^{p}(\log \mathscr{A}))$ which is isomorphic to $H_{\mathfrak{m}}^{q+1}(\Omega_{R}^{p}(\log \mathscr{A}))$ by Exercise III.2.3, III.3.3 \cite{HartshorneAlgGeo} and so Corollary \ref{cor - specific weights at origin, acyclic complex of local cohomology modules} can be applied to $E_{2}^{p,q}$ directly (i.e. without subsection 2.3 and subsection 2.6).
\end{proof}

\begin{remark} \label{rmk - finite dimensional lin alg for analytic stalk} \enspace
If $\mathscr{A}$ is central and the combinatorial conditions are satisfied, \eqref{eqn - comm diagram analytic to algebraic, homogeneous subcomplexes} also shows that
    \[
    H^{\bullet}((\Omega_{R}^{\bullet}(\log \mathscr{A})_{-\iota_{E}(\omega)}, \nabla_{\omega})) \simeq H^{\bullet}(\Omega_{X,0}(\star \mathscr{A}), \nabla_{\omega}) \simeq H^{\bullet}((\derivedR j_{\star} \localSystem_{{\Exp(\boldsymbol{\lambda})}})_{0}).
    \]
\end{remark}

\subsection{Feasibility} \enspace

In this minisicule section we ask which rank one local systems can be studied by our assorted (un)Twisted Logarithmic Comparison Theorems. Recall that when $\mathscr{A}$ is central, this means computing $(\derivedR j_{\star} \localSystem_{\Exp(\boldsymbol{\beta})})_{0}$ or $H^{\bullet}(U, \localSystem_{\boldsymbol{\beta}})$ is a finite dimensional linear algebra problem (Theorem \ref{thm - global algebraic twisted lct with homogeneous subcomplex q.i}; Remark \ref{rmk - finite dimensional lin alg for analytic stalk}). The answer is: all such local systems are computable this way.

\begin{proposition} \label{prop - all local systems can be computed}
For an arbitrary rank one local system $\localSystem_{\boldsymbol{\beta}}$ on $U$, there exists a $\boldsymbol{\lambda} \in \Exp^{-1}(\boldsymbol{\beta})$ such that all the aforementioned Twisted Logarithmic Comparison Theorems hold with respect to the weights $\boldsymbol{\lambda} = (\lambda_{1}, \dots, \lambda_{d}).$
\end{proposition}

\begin{proof}
We must produce $\boldsymbol{\lambda} \in \Exp^{-1}(\boldsymbol{\beta})$ such that the combinatorial arithmetical restrictions 
\[
\sum_{\{1 \leq k \leq d \mid E \subseteq \Var(f_{k})\}} \lambda_{k} \notin \mathbb{Z}_{\geq \min\{2, \rank(E)\}}
\]
are satisfied at each edge $E$ of the intersection lattice. To do this, pick an arbitrary $\boldsymbol{\lambda} \in \Exp^{-1}(\boldsymbol{\beta})$ and replace $\boldsymbol{\lambda}$ with $\boldsymbol{\lambda} - \textbf{z}$, where $\textbf{z} = (z, \dots, z) \in \mathbb{Z}^{d}$ and $z \gg 0$.
\end{proof}

\begin{remark} \label{rmk - comparing to other local system reuslts} \text{ }
\begin{enumerate}[label=(\alph*)]
    \item If we require $\mathscr{A}$ to be central and let $P(\mathscr{A})$ be the associated projective arrangement, then the Leray spectral sequence of the fibration $\mathbb{C}^{\star} \to \mathbb{C}^{n} \setminus \mathscr{A} \to \mathbb{P}^{n-1} \setminus P(\mathscr{A})$ gives a dictionary between cohomology of the affine complement with a rank one local system and the cohomology of the projective complement with the corresponding local system, cf. Theorem 5.2 \cite{DimcaHyperplaneArrangements}. In the projective case, if we require the weights $\{\lambda_{k}\}$ to satisfy $\lambda_{1} + \cdots + \lambda_{d} = 0$, then it is proved in \cite{CohomologyofLocalSystemsOnTheComplementOfHyperplanes} that there are similar combinatorial arithmetic conditions on the dense edges of the intersection lattice of $\mathscr{A}$ that, when satisfied, ensure
    \[
    H^{\bullet}((A^{\bullet}(\mathscr{A}), \nabla_{\omega})) \simeq H^{\bullet}(U, \localSystem_{\Exp(\boldsymbol{\lambda})}).
    \]
    Here $(A^{\bullet}(\mathscr{A}), \nabla_{\omega})$ is the Orlik--Solomon algebra equipped with the differential $\omega \wedge$, a twisted Orlik--Solomon algebra. Alternatively, $A^{\bullet}(\mathscr{A})$ is canonically identified with the Brieskorn algebra: the algebra generated by wedge products of the one-forms $\{ df_{k} / f_{k} \}$.
    \item It is well known that not all rank one local systems can be computed by the machinery of the previous item since, in particular, $\lambda_{1} + \cdots + \lambda_{d} = 0$ may not be satisfied, cf. Example \ref{ex - deleted B3 arrangement}. The undetectable systems correspond to torsion translated subtori of the characteristic variety, cf. \cite{CohenSuciuTangent}. This is not a defect of \cite{CohomologyofLocalSystemsOnTheComplementOfHyperplanes}: the invisibility of the torsion translated subtori is intrinsic to Orlik--Solomon methods.
    \item In \cite{SaitoArrangements} (subsection 2.2), M. Saito removes the requirement $\lambda_{1} + \cdots + \lambda_{d} = 0$ at the cost of taking a resolution of singularities $(\widetilde{X}, \widetilde{\mathscr{A}}) \to (X, \mathscr{A})$ and doing the computations upstairs in $\widetilde{X}$. This is difficult to use since, even for arrangements, resolutions of singularities are unwieldy.
\end{enumerate}
\end{remark}

\begin{example} \label{ex - deleted B3 arrangement} (Deleted $B_{3}$ arrangement)
Let $f = yz(x+y)(x-y)(x+z)(x-z)(y+z)(y-z)$, let $\boldsymbol{\lambda} = (1/2,1/2, -1/2, -1/2, 1/4, 1/4, 1/4, 1/4)$ be a collection of weights, ordered by the factors listed, and let $\omega$ be the associated one-form. This is a free arrangement with exponents $\{1, 3, 4\}$. The local system $\localSystem_{\Exp(\boldsymbol{\lambda})}$ on $U = \mathbb{C}^{3} \setminus \Var(f)$ belongs to a translated component of the rank one local systems on $U$, cf. Example 4.1 of \cite{SuciuTranslatedTori} or Section 3 of \cite{CohenTriplesofArrangements}. So it is undetectable by any twisted Orlik--Solomon algebra: $H^{k}(A^{\bullet}(\mathscr{A}), \nabla_{\omega}) = 0$ for $0 \leq k \leq 3$. In loc.\ cit.\ Suciu computes $\dim_{\mathbb{C}}(H^{1}(U, \localSystem_{\Exp(\boldsymbol{\lambda})})) = 1$. The weights $\boldsymbol{\lambda}$ satisfy the conditions of our Twisted Logarithmic Comparison Theorems. Macaulay2 claims:
\begin{align*}
\dim_{\mathbb{C}} (\text{im} [ \Omega_{R}^{1}(\log \mathscr{A})_{-1} \xrightarrow[]{\nabla_{\omega}} \Omega_{R}^{2}(\log \mathscr{A})_{-1}]) = 8;\\
\dim_{\mathbb{C}} (\text{im} [ \Omega_{R}^{2}(\log \mathscr{A})_{-1} \xrightarrow[]{\nabla_{\omega}} \Omega_{R}^{3}(\log \mathscr{A})_{-1}]) = 8.
\end{align*}
Using Theorem \ref{thm - global algebraic twisted lct with homogeneous subcomplex q.i}, we compute the dimensions of the nonzero $H^{k}(U,\localSystem_{\Exp(\boldsymbol{\lambda})})$:
\begin{align*}
    \dim_{\mathbb{C}}(H^{1}(U, \localSystem_{\Exp(\boldsymbol{\lambda})})) = \dim_{\mathbb{C}} (H^{1}(\Omega_{R}^{\bullet}(\log f)_{-1}, \nabla_{\omega})) = 1; \\
     \dim_{\mathbb{C}}(H^{2}(U, \localSystem_{\Exp(\boldsymbol{\lambda})})) = \dim_{\mathbb{C}} (H^{2}(\Omega_{R}^{\bullet}(\log f)_{-1}, \nabla_{\omega})) = 8; \\
      \dim_{\mathbb{C}}(H^{3}(U, \localSystem_{\Exp(\boldsymbol{\lambda})})) = \dim_{\mathbb{C}} (H^{3}(\Omega_{R}^{\bullet}(\log f)_{-1}, \nabla_{\omega})) = 7.
\end{align*}
\end{example}

\section{Applications to $\mathscr{D}_{X}$-modules and Bernstein--Sato Ideals}

Hereafter, we will require $\mathscr{A}$ to be central. Let $\mathscr{D}_{X}$ be the analytic sheaf of $\mathbb{C}$-linear differential operators. We will mostly focus on the analytic setting, but will track statements that also work for the algebraic one. 

Because the Logarithmic Comparison Theorem (twisted or otherwise) offers a quasi-isomorphism between the logarithmic de Rham complex and the meromorphic de Rham complex and because the meromorphic de Rham complex is obtained by applying the de Rham functor to a certain $\mathscr{D}_{X}$-module, it is natural to hope for a intrinsic complex of $\mathscr{D}_{X}$-modules whose image under the de Rham functor is the logarithmic de Rham complex. There is a good candidate for free divisors and the result is known for certain such free divisors, see \cite{MorenoMacarroLogarithmic} as well as the survey \cite{NarvaezMacarroLinearityConditions}. Consequently, there is a package of $\mathscr{D}_{X}$-theoretic results divisors satisfying Logarithmic Comparison Theorems ``ought'' to possess. For certain free divisors, the most thorough treatment is \cite{MAcarroDuality}, especially Theorem 4.7 and Corollary 4.2/4.3 from which most of the ``expected'' properties discussed in \cite{TorelliLogarithmicComparisionTheoremAnd} can be deduced in this case; see also loc.\ cit.\ and \cite{NarvaezMacarroLinearityConditions} for more discussion. Despite having (un)Twisted Logarithmic Comparison Theorems for reduced hyperplane arrangements we do not have a good candidate for the ``pre''-de Rham complex of $\mathscr{D}_{X}$-modules corresponding to the logarithmic de Rham complex.

Nevertheless we are able to obtain many of the desired $\mathscr{D}_{X}$-related results for reduced hyperplane arrangements $f$. In the first subsection we study $\mathscr{O}_{X}(\star f^{\boldsymbol{\lambda}})$ for a set of weights $\boldsymbol{\lambda} \in \mathbb{C}^{d}$, showing that (Theorem \ref{thm - twisted LCT implies generating level of meromorphic specialization}), under familiar hypotheses on the weights, $f^{-\textbf{1} + \boldsymbol{\lambda}}$ generates this $\mathscr{D}_{X}$-module. This should be thought of as a multivariate generalization of fact $f^{-1}$ generates $\mathscr{O}_{X}(\star f)$; the theorem recovers this univariate statement. We also justify our previous claim $(\Omega_{X}^{\bullet}(\star \mathscr{A}), \nabla_{\omega}) = \derivedR j_{\star} \localSystem_{\Exp(\boldsymbol{\lambda})}$ in the derived category (Proposition \ref{prop - twisted meromorphic de rham quasi-iso to local system}). In the second subsection we give a sharp bound on the candidates of the codimension one components of the zero locus of the Bernstein--Sato ideal of $f$ (for any factorization). This is Theorem \ref{thm - bounding codim one components of BS zero loci} and should be thought of as a sort of generalization of M. Saito's result \cite{SaitoArrangements} that the roots of the Bernstein--Sato polynomial of $f$ live in $(-2 + 1/\deg(f), 0)$--in fact, the univariate version of Theorem \ref{thm - bounding codim one components of BS zero loci} gives a new proof of this fact. 

\subsection{Generators of Twisted Meromorphic Powers} \text{ }

First, recall that $\mathscr{O}_{X}(\star \mathscr{A})$ is naturally a $\mathscr{D}_{X}$-module: the action of a derivation on $f^{-p} $, for $p \in \mathbb{N}$ and with $f$ defining $\mathscr{A}$, is given by the chain rule. Nothing depends on the choice of defining equation for $\mathscr{A}$ and we note that the cyclic $\mathscr{D}_{X}$-module generated by $f^{-p}$ sits inside $\mathscr{O}_{X}(\star \mathscr{A})$. In particular, as $\mathscr{D}_{X}$-modules,
\[
\bigcup_{p \in \mathbb{N}} \mathscr{D}_{X} f^{-p} = \mathscr{O}_{X}(\star \mathscr{A}) = \mathscr{O}_{X}(\star f).
\]
Now let $f = f_{1} \cdots f_{d}$ be a factorization into irreducibles and take $\textbf{p} \in \mathbb{Z}^{d}$. Denote $f_{1}^{p_{1}} \cdots f_{d}^{p_{d}}$ by $f^{\textbf{p}}$. Then, as $\mathscr{D}_{X}$-modules,
\[
\bigcup_{\textbf{p} \in \mathbb{N}^{d}} \mathscr{D}_{X} f^{ -\textbf{p}} = \mathscr{O}_{X}(\star \mathscr{A}) = \mathscr{O}_{X}(\star f).
\]

We can tell the same story for nonstandard powers of the $f_{k}$. To this end, let $\boldsymbol{\lambda} = (\lambda_{1}, \dots, \lambda_{d}) \in \mathbb{C}^{d}$ be weights as before, define the $\mathscr{D}_{X}$-action on $f^{ \boldsymbol{\lambda}}$ via the chain and product rule, and make the $\mathscr{D}_{X}$-module identification/definition
\[
\bigcup_{p \in \mathbb{N}^{d}} \mathscr{D}_{X} f^{ \boldsymbol{\lambda} - \textbf{p}} = \mathscr{O}_{X}(\star f^{ \boldsymbol{\lambda}}).
\]
Repeating the construction algebraically produces a $D_{X_{\alg}}$-module $\mathscr{O}_{X_{\alg}}(\star f^{\boldsymbol{\lambda}}).$

For technical reasons, we will work with the following equivalent object:

\begin{proposition} \label{prop - technical details of funky connection}
Consider $\mathscr{O}_{X}(\star f)$ as an $\mathscr{D}_{X}$-module. Let $\textbf{p} \in \mathbb{N}^{d}$. We let $\delta \in \Der_{X}$ act on $g f^{-\textbf{p}}$ by
\[
\delta(g f^{-\textbf{p}}) = \delta(g) f^{-\textbf{p}} + \sum_{1 \leq k \leq d} g (-p_{k} + \lambda_{k}) \delta(f_{k}) f_{k}^{-1} f^{-\textbf{p}}.
\]
This defines a connection $\nabla^{\boldsymbol{\lambda}}$ on $\mathscr{O}_{X}(\star f)$ satisfying the commutative diagram of $\mathscr{D}_{X}$-maps
\[
\begin{tikzcd}
\mathscr{D}_{X}f^{-\textbf{p} + \boldsymbol{\lambda}} \rar[hookrightarrow] 
    & \mathscr{O}_{X}(\star f^{\boldsymbol{\lambda}})  \\
(\mathscr{D}_{X}f^{-\textbf{p}}, \nabla^{\boldsymbol{\lambda}}) \uar{\simeq} \rar[hookrightarrow]
    & (\mathscr{O}_{X}(\star f), \nabla^{\boldsymbol{\lambda}}) \uar{\simeq} ,
\end{tikzcd}
\]
where $(\mathscr{D}_{X} f^{-\textbf{p}}, \nabla^{\boldsymbol{\lambda}})$ denotes the module generated by $f^{-\textbf{p}}$ under this new connection $\nabla^{\boldsymbol{\lambda}}$ and the vertical isomorphisms are determined by $g f^{-\textbf{p}} \mapsto g f^{-\textbf{p} + \boldsymbol{\lambda}}$.
\end{proposition}

\begin{proof}
The vertical maps in the diagram are certainly $\mathscr{O}_{X}$-isomorphisms and it is easy to check that action of a derivation $\delta$ commutes with the vertical maps: $\psi(\nabla_{\delta}^{\boldsymbol{\lambda}} (-)) = \delta(\psi(-)).$ This fact simplifies the straightforward calculations mandatory in checking $\nabla^{\boldsymbol{\lambda}}$ is a well-defined connection (and then certifies the vertical maps are $\mathscr{D}_{X}$-isomorphisms).
\end{proof}

\begin{remark} \text{ }
\begin{enumerate}[label=(\alph*)]
    \item Note that $\nabla_{\delta}^{\boldsymbol{\lambda}}(1) = \nabla_{\delta}^{\boldsymbol{\lambda}}(f^{\boldsymbol{0}}) = \sum_{k} \delta(f_{k}) f^{-\boldsymbol{e_{k}}}$, where $\boldsymbol{e_{k}}$ are the standard unit vectors. 
    \item Proposition \ref{prop - technical details of funky connection} also holds in the algebraic setting.
\end{enumerate}
\end{remark}

Before proceeding, let us complete our proofs of the analytic (un)Twisted Logarithmic Comparison Theorems by justifying the relationship between the analytic twisted meromorphic de Rham complex and the derived direct image of a rank one local system on the the complement $U = X \setminus \mathscr{A}$. This is known to experts and follows from the explicit Riemann--Hilbert correspondence.

\begin{proposition} \label{prop - twisted meromorphic de rham quasi-iso to local system}
Let $f = f_{1} \cdots f_{d} \in R$ be a reduced hyperplane arrangement $\mathscr{A}$ equipped with a factorization into linear forms. Suppose that $\boldsymbol{\lambda} = (\lambda_{1}, \dots, \lambda_{d}) \in \mathbb{C}^{d}$ and let $\omega = \sum_{k} \lambda_{k} df_{k} / f_{k}$ the one-form determined by $\boldsymbol{\lambda}$. And let $\localSystem_{\Exp(\lambda}$ be the rank one local system on $X \setminus \mathscr{A}$ corresponding to the torus point $\Exp(\boldsymbol{\lambda})$; let $j: X \setminus \mathscr{A} \xhookrightarrow{} X$ be the inclusion. Then we have the following equality in the category of perverse sheaves:
\[
(\Omega_{X}^\bullet(\star \mathscr{A}), \nabla_{\omega}) = \derivedR j_{\star} \localSystem_{\Exp(\boldsymbol{\lambda})}.
\]
\end{proposition}

\begin{proof}
We justify the following:
\begin{align*}
\deRham((\Omega_{X}^\bullet(\star \mathscr{A}), \nabla_{\omega})) \stackrel{(1)}{=} \deRham((\mathscr{O}_{X}(\star f), \nabla_{\boldsymbol{\lambda}})) \stackrel{(2)}{=} \deRham(\mathscr{O}_{X}(\star f^{\lambda})) \stackrel{(3)}{=} \derivedR j_{\star} \localSystem_{\Exp(\boldsymbol{\lambda})}.
\end{align*}
Note that the equalities occur in the derived category. Indeed: $\stackrel{(1)}{=}$ is nothing more than a straightforward computation; since $\mathscr{O}_{X}(\star f^{\boldsymbol{\lambda}})$ is regular holonomic (cf. Theorem 5.2 \cite{BudurLocalSystems}) and since $\deRham(-)$ gives an equivalence of categories between regular holonomic $\mathscr{D}_{X}$-modules and perverse sheaves, by Proposition \ref{prop - technical details of funky connection} we have $\stackrel{(2)}{=}$ in the category of perverse sheaves; the explicit Riemann--Hilbert correspondence gives $\stackrel{(3)}{=}$ (cf. Theorem 5.2 \cite{BudurLocalSystems}).
\end{proof}

Using the identification within Proposition \ref{prop - technical details of funky connection}, we can use the Twisted Logarithmic Comparison Theorem to locate choices of $\boldsymbol{\lambda}$ such that $\mathscr{D}_{X} f^{-\textbf{1} + \boldsymbol{\lambda}} = \mathscr{O}_{X}(\star f^{\boldsymbol{\lambda}})$. The following is based off an observation of Torelli in the univariate, untwisted case (cf. Proposition 3.1 \cite{TorelliLogarithmicComparisionTheoremAnd}).

\begin{theorem} \label{thm - twisted LCT implies generating level of meromorphic specialization}
Let $f = f_{1} \cdots f_{d}$ cut out a central, reduced hyperplane arrangement and let $\boldsymbol{\lambda} = (\lambda_{1}, \dots, \lambda_{d}) \in \mathbb{C}^{d}$ be weights such that for each edge $E$
\[
\sum_{\{1 \leq k \leq d \mid E \subseteq \Var(f_{k})\}} \lambda_{k} \notin \mathbb{Z}_{\geq \min\{2, \rank(E)\}}.
\]
Then 
\[
\mathscr{D}_{X} f^{-\textbf{1} + \boldsymbol{\lambda}} = \mathscr{O}_{X}(\star f^{\boldsymbol{\lambda}}).
\]
\end{theorem}

\begin{proof}
Doing this stalk-by-stalk, with the standard reduction to a smaller rank arrangement at points not the origin, and using the Proposition \ref{prop - technical details of funky connection}, it suffices to prove that the inclusion $(\mathscr{D}_{X,0}f^{-\textbf{1}}, \nabla^{\boldsymbol{\lambda}}) \subseteq (\mathscr{O}_{X,0}(\star f), \nabla^{\boldsymbol{\lambda}})$ is an equality.

Define the $\mathscr{D}_{X,0}$-module $Q$ to be the cokernel of this inclusion:
\begin{equation} \label{eqn - ses cokernel Q}
0 \to (\mathscr{D}_{X,0}f^{-\textbf{1}}, \nabla^{\boldsymbol{\lambda}}) \to (\mathscr{O}_{X,0}(\star f), \nabla^{\boldsymbol{\lambda}}) \to Q \to 0.
\end{equation}
Apply the de Rham functor $\deRham(-)$ to this short exact sequence. By the inductive hypothesis, $Q$ is supported at the origin; by Kashiwara's equivalence it is isomorphic  to a direct sum of copies of $\mathscr{D}_{X,0}/\mathscr{D}_{X,0} \cdot \mathfrak{m}_{0} \simeq \mathbb{C}[\partial_{1}, \dots, \partial_{n}]$. The de Rham complex attached to this $\mathscr{D}_{X,0}$-module is essentially the Koszul complex of the regular sequence $\partial_{1}, \dots, \partial_{n}$ on $\mathbb{C}[\partial_{1}, \dots, \partial_{n}]$. Thus $H^{j}(\deRham(Q)) = 0 $ for all $j \neq n$. This means the long exact sequence arising from applying the de Rham functor to \eqref{eqn - ses cokernel Q} is a collection of isomorphisms and one short exact sequence:
\begin{align} \label{eqn - de Rham isomorphisms and ses in last slot}
&0 \to H^{n}(\deRham((\mathscr{D}_{X,0}f^{-\textbf{1}}, \nabla^{\boldsymbol{\lambda}})) \to H^{n}(\deRham((\mathscr{O}_{X,0}(\star f), \nabla^{\boldsymbol{\lambda}}))) \to H^{n}(\deRham(Q)) \to 0.
\end{align}
Note that the first nontrivial map is induced by $\deRham(-)$ applied to a submodule inclusion. 

On the other hand, by the definition of the de Rham functor one can easily check that we have an inclusion of complexes
\begin{equation} \label{eqn - inclusion of three de rham complexes}
(\Omega_{X, 0}^{\bullet}(\log f), \nabla_{\omega}) \xhookrightarrow{} \deRham((\mathscr{D}_{X,0}f^{-\textbf{1}}, \nabla^{\boldsymbol{\lambda}})) \xhookrightarrow{} \deRham((\mathscr{O}_{X,0}(\star f), \nabla^{\boldsymbol{\lambda}})), 
\end{equation}
where the last complex equals $(\Omega_{X,0}^{\bullet}(\star f), \nabla_{\omega})$. By Theorem \ref{thm - analytic twisted lct}, the composition of maps in \eqref{eqn - inclusion of three de rham complexes} is a quasi-isomorphism. So the second map in \eqref{eqn - inclusion of three de rham complexes} is a surjection on the level of cohomology. Specifically, for $H^{n}(-)$, this surjection is the same map as the first nontrivial map in \eqref{eqn - de Rham isomorphisms and ses in last slot}, since both are induced by submodule inclusion. We deduce $H^{n}(\deRham(Q)) = 0$ and $\deRham(Q)$ is acyclic.

Since $(\mathscr{O}_{X,0}(\star f), \nabla^{\boldsymbol{\lambda}}) \simeq \mathscr{O}_{X,0}(\star f^{\boldsymbol{\lambda}})$ is regular holonomic (Theorem 5.2 \cite{BudurLocalSystems}) and these properties are preserved under subquotients, $Q$ is regular holonomic. By the Riemann--Hilbert correspondence, there is an equivalence of categories between regular holonomic $\mathscr{D}_{X,0}$-modules and perverse sheaves. As have shown $\deRham(Q)$ corresponds to the zero perverse sheaf, we deduce $Q = 0$, completing the proof.
\end{proof}

\subsection{Codimension One Components of Bernstein--Sato Ideals} \text{ }

We conclude with a sharp estimate of the codimension one components of the Bernstein--Sato ideal attached to an arbitrary factorization of a central, reduced hyperplane arrangement. We first give a quick tour of some of the necessary concepts, with some other tools deferred to the proof.

Let $f$ cut out a central, reduced hyperplane arrangement and let $f = f_{1} \cdots f_{r}$ be an arbitrary factorization of $f$, not necessarily into linear forms. We encode this factorization by $F = (f_{1}, \dots, f_{r}).$ Formally invert $f$, add $r$ dummy variables $s_{1}, \dots, s_{r}$, and consider the $\mathscr{O}_{X,0}[\frac{1}{f}, S] = \mathscr{O}_{X,0}[\frac{1}{f}, s_{1}, \dots, s_{r}]$-module generated by $f_{1}^{s_{1}} \cdots f_{r}^{s_{r}}$; call this $\mathscr{O}_{X,0}[\frac{1}{f}, S] F^{S}$. This is a $\mathscr{D}_{X,0}[S] = \mathscr{D}_{X,0}[s_{1}, \dots, s_{r}]$-module where the action of derivation on $F^{S} = f_{1}^{s_{r}} \cdots f_{r}^{s_{r}}$ is given by formally applying the chain and product rule. The cyclic $\mathscr{D}_{X,0}[S]$-submodule generated by $F^{S}$ (resp. $F^{S+1} = f_{1}^{s_{1}+1} \cdots f_{r}^{s_{r}+1}$) is denoted by $\mathscr{D}_{X,0}[S] F^{S}$ (resp. $\mathscr{D}_{X,0}[S] F^{S+1}$). 

\begin{define} \label{def - basic BS ideal definitions}
Consider the $\mathscr{D}_{X}[S]$-module
\[
M = \frac{\mathscr{D}_{X,0}[S] F^{S}}{\mathscr{D}_{X,0}[S] F^{S+1}}.
\]
The $\mathbb{C}[S] = \mathbb{C}[s_{1}, \dots, s_{r}]$-module annihilator of this module is the \emph{Bernstein--Sato ideal}. It is denoted by
\[
B_{F,0} = \ann_{\mathbb{C}[S]} M.
\]
We call the zero locus of the Bernstein--Sato ideal $Z(B_{F,0}) \subseteq \mathbb{C}^{r}$ and we single out its codimension one components with 
\[
Z_{r-1}(B_{F,0}) = \text{ codimension one components of } Z(B_{F,0}).
\]
When working with a factorization $F = (f_{1}, \dots, f_{r})$ into more than one term (i.e. $r \geq 2$) we call this the multivariate situation; when the factorization $F = (f)$ is trivial (i.e. $r = 1$) we call this the univariate situation. In the univariate situation the \emph{Bernstein--Sato polynomial} is the monic generator of the (necessarily principal) Bernstein--Sato ideal.
\end{define}

It is well known that the Bernstein--Sato ideal (both the univariate and multivariate versions) is nonzero. We recommend the modern treatments \cite{MaisonobeFiltrationRelative, ZeroLociI, ZeroLociII} which focus on the multivariate situation, especially since we utilize many of the ideas therein. In \cite{MaisonobeFiltrationRelative}, Maisonobe realized studying the $\Ext$-modules
\begin{equation} \label{eqn - Ext modules for BS things}
\Ext_{\mathscr{D}_{X,0}[S]}^{\bullet}(M, \mathscr{D}_{X,0}[S])
\end{equation}
is crucial to understanding Bernstein--Sato ideals. In particular, he proved the first nonvanishing $\Ext$-module sits in slot $n+1$ (recall $n = \dim X$). In \cite{ZeroLociI}, the authors: (1) developed the theory of localized $\mathscr{D}_{X,0}[S]$-modules, for example, replacing the $\mathscr{D}_{X,0}[S]$-module $M$ with the $\mathscr{D}_{X,0} \otimes_{\mathbb{C}} T^{-1}\mathbb{C}[S]$-module $M \otimes_{\mathbb{C}} T^{-1} \mathbb{C}[S]$ for a multiplicatively closed subset $T \subseteq \mathbb{C}[S]$; (2) showed that when $M$ is \emph{$(n+1)$-Cohen--Macaulay}, that is, when amongst all the $\Ext$-modules of \eqref{eqn - Ext modules for BS things} only the $\Ext$-module sitting in slot $n+1$ is nonzero, $M$ has very nice properties. 

We will make use of the approach of \cite{ZeroLociI} where they show $M$ is generically $(n+1)$-Cohen--Macaulay (generic with respect to points of $\mathbb{C}^{r}$). We will also make use of the result of the previous subsection. For $\textbf{a} = (a_{1}, \dots, a_{r}) \in \mathbb{C}^{r}$, let $\mathbb{C}_{\textbf{a}}$ be the associated residue field. There is a natural surjective specialization map of $\mathscr{D}_{X,0}$-modules
\[
\mathscr{D}_{X,0}[S] F^{S} \otimes_{\mathbb{C}[S]} \mathbb{C}_{\textbf{a}} \twoheadrightarrow \mathscr{D}_{X,0} f^{\textbf{a}} = \mathscr{D}_{X,0}f_{1}^{a_{1}} \cdots f_{r}^{a_{r}}
\]
obtained by replacing all $s_{k}$ with $a_{k}$. (See \cite{OakuAlgorithm} and section 5 of \cite{BudurLocalSystems} for details.) Theorem \ref{thm - twisted LCT implies generating level of meromorphic specialization} will, with suitable caution, let us use some of the results from \cite{OakuAlgorithm} regarding this specialization. 

Because we care about arbitrary factorizations $F$ of $f$, we require some technical notation exclusive to hyperplane arrangements. 
\begin{define} \label{def - notation for factorization of arrangements along edges}
Let $f = \ell_{1} \cdots \ell_{d}$ be a factorization of a central, reduced arrangement $f$ into linear terms. For an edge $E$ in the intersection lattice $\mathscr{L}(\mathscr{A})$, let $J(E) = \{ j \in [d] \mid E \subseteq \Var(\ell_{j})\}$ so that $f_{E} = \prod_{j \in J(X)} \ell_{j}$ is the arrangement consisting of all hyperplanes containing $E$. The factorization $F = (f_{1}, \cdots, f_{r})$ amounts to a disjoint partition of $[d]$ into $r$ subsets; let $S_{k} = \{ j \in [d] \mid \Var(\ell_{j}) \subseteq \Var(f_{k})\}$. Thus $f_{k} = \prod_{j \in S_{k}} \ell_{j}$. The factorization $F$ of $f$ induces a factorization $f_{E} = f_{E, 1} \cdots f_{E, r}$ where 
\[
f_{E,k} = \prod_{j \in S_{k} \cap J(E)} \ell_{j}.
\]
Finally, let $d = \deg(f)$, $d_{E} = \deg(f_{E}) = \mid J(E) \mid$, and $d_{E,k} = \deg(f_{E,k}) = \mid J(E) \cap S_{k} \mid.$ Note that $\deg(f_{E,k})$ can equal zero.
\end{define}

Now we can state and prove the subsection's main result. What is new below is the upper bound $Q_{E}$ on the size of $v$; that all the codimension one components of $Z(B_{F,0})$ are of this form (without any restriction other than the nonnegativity of $v$) was proved in Theorem 4.18 of \cite{Bath2}. We highlight that, in the univariate case, this independently recovers (by a very different argument) M. Saito's result \cite{SaitoArrangements} that the roots of the Bernstein--Sato polynomial of a smooth, central, reduced arrangement lie in $(-2 + \frac{1}{d}, 0)$.

\begin{theorem} \label{thm - bounding codim one components of BS zero loci}
For $F = (f_{1}, \dots, f_{r})$ an arbitrary factorization of a reduced, central arrangement $f$, the codimension one components of the Bernstein--Sato ideal attached to $F$ has the following restriction:
\[
Z_{r-1}(B_{F,0}) \subseteq \bigcup_{\substack{E \in \mathscr{L}(\mathscr{A}) \\ E \text{ \normalfont dense}}} \bigcup_{v=0}^{Q_{E}} \left\{ \sum\limits_{\{1 \leq k \leq r \mid E \subseteq \Var(f_{k})\}} d_{E,k} s_{k} + \rank(E) + v = 0 \right\}
\]
where $d_{E,k}$ is as in Definition \ref{def - notation for factorization of arrangements along edges} and 
\[
Q_{E} = 
\begin{cases}
2 d_{E} - \rank(E) - \min\{2, \rank(E)\} \text{ for } F \text{ a factorization into linears}; \\
2 d_{E} - \rank(E) - \min\{2, \rank(E)\} \text{ for $F$ any factorization } \& \text{ } E = \{0\}; \\
2 d_{E} - \rank(E) - 1 \text{ for } F \text{ not a factorization into linears }\& \text{ } E \neq \{0\}.
\end{cases}
\]
In particular, when $F = (f)$ is the trivial factorization and $f$ is not smooth, the roots of the Bernstein--Sato polynomial are contained in $(-2 + 1/d, 0).$
\end{theorem}

\begin{proof}
\emph{Case 1:} We first assume $r = d$ (i.e. $F$ is a factorization into linear terms) and establish notation and a plan that will be used by the other cases. 

\emph{General Notation:} Fix an arbitrary dense edge $E$ and then fix an integer $\ell \in [0, d_{E})$. Choose the integer $m \in \mathbb{Z}$ by the mandate
\begin{align} \label{eqn - choice of m; F a complete factorization}
m &= \min\{ p \in \{ \rank(E) + \ell + z \cdot d_{E} \mid z \in \mathbb{Z} \} \mid -p \notin \mathbb{Z}_{\geq \min\{2, \rank(E)\}} \} \\
    &\in [-\min\{2, \rank(E)\} + 1, - \min\{2, \rank(E)\} + d_{E}] \nonumber
\end{align}
and consider the prime ideal in $\mathbb{C}[S] = \mathbb{C}[s_{1}, \dots, s_{r}]$
\begin{equation} \label{eqn - def of prime ideal with min choice of m}
\mathfrak{q} = \mathbb{C}[S] \cdot \sum_{\{1 \leq k \leq r \mid E \subseteq \Var(f_{k})\}} d_{E,k} s_{k} + m.
\end{equation}
Furthermore, let $\tau:  \mathbb{C}[S] \to  \mathbb{C}[S]$ be the ring isomorphism defined by $s_{k} \mapsto s_{k}+1$ for all $k$. This induces a map $\tau^{\sharp}: \Spec \mathbb{C}[S] \mapsfrom \Spec \mathbb{C}[S]$ sending $ (a_{1}, \dots, a_{r}) \mapsfrom (a_{1} - 1, \dots, a_{r} - 1).$ In particular, $(\tau^{\sharp})^{j} : \Var(\tau^{j}(\mathfrak{q})) \to \Var(\mathfrak{q})$ is a homeomorphism in the Zariski topology. To make this plain, if $j \in \mathbb{Z}$, then 
\[
(\tau^{\sharp})^{-j}(\textbf{a}) = (a_{1} - j, \dots, a_{r} - j).
\]

\emph{General Plan:} We know from Theorem 4.18  of \cite{Bath2} that every codimension one component of $Z(B_{F,0})$ is of the form $\Var(\sum_{\{1 \leq k \leq r \mid E \subseteq \Var(f_{k})\}} d_{E,k} s_{k} + \rank(E) + u)$ for some dense edge $E \in \mathscr{L}(\mathscr{A})$ and some nonnegative integer $u$. Therefore showing the following claim proves the theorem:
\begin{equation} \label{eqn - BS ideal zero loci estimation; the claim}
\text{\textbf{The Claim}:} \Var(\tau^{j}(\mathfrak{q})) \nsubseteq Z_{r-1}(B_{F,0}) \text{ for each } j \in \{2, 3, 4, \dots \}.
\end{equation}
Indeed, in \eqref{eqn - choice of m; F a complete factorization} we selected $ - \min\{2, \rank(E)\} + 1 \leq m \leq - \min\{2, \rank(E)\} + d_{E}$. So if we prove \textbf{The Claim} \eqref{eqn - BS ideal zero loci estimation; the claim} we have proved every codimension one component of $Z(B_{F,0})$ associated to the dense edge $E$ is of of the form $\Var(\sum_{\{1 \leq k \leq r \mid E \subseteq \Var(f_{k}) \}} d_{k} s_{k} + v)$ where $- \min\{2, \rank(E)\} + 1 \leq v \leq - \min\{2, \rank(E)\} + 2 d_{E}$. Combining this with the first sentence of \emph{General Plan} means the codimension one components of $Z(B_{F,0})$ associated to the dense edge $E$ are of the form
\[
\Var(\sum_{\{1 \leq k \leq r \mid E \subseteq \Var(f_{k})\}} d_{E,k} s_{k} + \rank(E) + w)
\]
where 
\[
0 \leq w \leq 2 d_{E} -\rank(E) -\min\{2, \rank(E)\}.
\]

To prove \textbf{The Claim} we find a nonempty dense set $\Gamma \subseteq \Var(\mathfrak{q})$ such that for all $j \in \{2, 3, \dots, \}$ the translate $(\tau^{\sharp})^{-j}(\Gamma)$ and the zero locus $Z(B_{F,0})$ are disjoint. Since $(\tau^{\sharp})^{-j}(\Gamma) \subseteq (\tau^{\sharp})^{-j}(\mathfrak{q}) = \Var(\tau^{j}(\mathfrak{q})$, this shows $\Var(\tau^{j}(\mathfrak{q})$ is not a codimension one component of $Z(B_{F,0})$ certifying \textbf{The Claim}. The set $\Gamma$ is constructed by intersecting the elements of two countable collections of Zariski open sets: in \emph{Step 1} we identify a countable collection of Zariski opens so that for points in their intersection $M$ satisfies the $(n+1)$-Cohen--Macaulay property ala the approach of \cite{ZeroLociI}; in \emph{Step 2} we identify a different countable collection of opens so that points in all these opens satisfy the set-up of Theorem \ref{thm - twisted LCT implies generating level of meromorphic specialization}. In \emph{Step 3} we intersect all these opens, call it $\Gamma$, and use the generic properties of \emph{Step 1} and \emph{Step 2} to study points in $(\tau^{\sharp})^{-j}(\Gamma) \cap Z(B_{F,0})$; in \emph{Step 4} we show $\Gamma$ is non-empty and dense using the Baire Category Theorem.

\emph{Step 1:} Generic Cohen--Macaulay translates.

We begin by utilizing the main technique of \cite{ZeroLociI} (see Lemma 3.5.2). For $j$ some integer, consider the localization of $M$ at the multiplicatively closed set $T_{j} = \mathbb{C}[S]\setminus \cup_{\mathfrak{m} \in \Var(\tau^{j}(\mathfrak{q}))} \mathfrak{m}$. Then $M \otimes_{\mathbb{C}[S]} T_{j}^{-1}\mathbb{C}[S]$ is relative holonomic (Definition 3.2.3 of \cite{ZeroLociI}) and has a Bernstein--Sato ideal (since $M$ satisfies both before localization); moreover, the zero locus of this Bernstein--Sato ideal sits inside $\Var(\tau^{j}(\mathfrak{q}))$. Following Lemma 3.5.2 of \cite{ZeroLociI} (or just using the result), we can consider the zero locus $Z_{p,j} \subseteq \Var(\tau^{j}(\mathfrak{q}))$ of the Bernstein--Sato ideal of 
\[
\Ext_{\mathscr{D}_{X,0}[S] \otimes_{\mathbb{C}[S]} T_{j}^{-1}\mathbb{C}[S]}^{p}(M \otimes_{\mathbb{C}[S]} T_{j}^{-1}\mathbb{C}[S], \mathscr{D}_{X,0}[S] \otimes_{\mathbb{C}[S]} T_{j}^{-1}\mathbb{C}[S]).
\]
Provided $p \geq n+2$, each $Z_{p,j} \subseteq \Var(\tau^{j}(\mathfrak{q}))$ is Zariski closed and of codimension (with respect to this containment) of at least one. Consider the nonempty Zariski open set
\[
C_{j} = \Var(\tau^{j}(\mathfrak{q})) \setminus \left( \bigcup_{p = n+2}^{p = n+r} Z_{p,j} \right) \subseteq  \Var(\tau^{j}(\mathfrak{q}))
\]
and let $D_{j} \subseteq T_{j}^{-1}\mathbb{C}[S]$ the corresponding multiplicatively closed set. Then $M \otimes_{\mathbb{C}[S]} T_{j}^{-1} \mathbb{C}[S] \otimes_{\mathbb{C}[S]} D_{j}^{-1} \mathbb{C}[S]$ is $(n+1)$-Cohen--Macaulay (over points in $C_{j})$. (See Lemma 3.5.2 of \cite{ZeroLociI} for more details.)

Let $\mathbb{C}_{\textbf{a}}$ be the residue field corresponding to the point $\textbf{a} \in \mathbb{C}^{r}$. Suppose that $\textbf{a} \in C_{j}$. Since Bernstein--Sato ideals of relative holonomic modules localize (Lemma 3.4.1 and Remark 3.2.1 of \cite{ZeroLociI}), our point $\textbf{a}$ is in the zero locus of the Bernstein--Sato ideal of $M$ if and only if $\textbf{a}$ is in the zero locus of the Bernstein--Sato ideal of $M \otimes_{\mathbb{C}[S]} T_{j}^{-1} \mathbb{C}[S] \otimes_{\mathbb{C}[S]} D_{j}^{-1} \mathbb{C}[S]$, the latter of which is $(n+1)$-Cohen--Macaulay by construction. By Proposition 3.4.3 of of \cite{ZeroLociI} and the Cohen--Macaulay-ness, $\textbf{a} \in Z(B_{F,0})$ if and only if $M \otimes_{\mathbb{C}[S]} \mathbb{C}_{\textbf{a}} \neq 0.$ 

Recall that $(\tau^{\sharp})^{j}: \Var(\tau^{j}(\mathfrak{q})) \to \Var(\mathfrak{q})$ is a homeomorphism. Consider the collection of nonemtpy Zariski opens
\begin{equation} \label{eqn - collection of opens, generically CM translates}
\mathscr{C} = \{ C_{0}, (\tau^{\sharp})^{1}(C_{1}), (\tau^{\sharp})^{2}(C_{2}), \dots \},
\end{equation}
and note that all elements of $\mathscr{C}$ not only lie in $\Var(\mathfrak{q})$ but are dense therein. Moreover, if $\textbf{a} \in \bigcap_{U \in \mathscr{C}} U$, then for all $j \in \mathbb{Z}_{\geq 0}$ we have that $\textbf{a} \in (\tau^{\sharp})^{j}(C_{j})$. Alternatively, $(\tau^{\sharp})^{-j}(\textbf{a}) \in C_{j}$. This puts $(\tau^{\sharp})^{-j}(\textbf{a})$ within the $(n+1)$-Cohen--Macaulay locus of $\Var(\tau^{j}(\mathfrak{q}))$ and hence $(\tau^{\sharp})^{-j}(\textbf{a}) \in Z(B_{F,0})$ if and only if $M \otimes_{\mathbb{C}[S]} \mathbb{C}_{(\tau^{\sharp})^{-j}(\textbf{a})} \neq 0.$ This gives us a criterion for verifying that none of the translates $\{ (\tau^{\sharp})^{-j}(\textbf{a})\}$ of $\textbf{a} \in \bigcap_{U \in \mathscr{C}} U$ lie in the zero locus of the Bernstein--Sato ideal of $M$: no translate lies in this zero locus provided $M \otimes_{\mathbb{C}[S]} \mathbb{C}_{(\tau^{\sharp})^{-j}(\textbf{a})} = 0$ for all $j \in \mathbb{Z}_{\geq 0}$.

\emph{Step 2:} Generically generating $\mathscr{O}_{X}(\star f^{\lambda})$.

For each edge $E^{\prime} \in \mathscr{L}(\mathscr{A})$ and each $p \in \mathbb{Z}_{\geq \min\{2, \rank(E)\}}$, consider the nonempty distinguished Zariski open set
\[
G_{E^{\prime}, p} = \{ \boldsymbol{\lambda} \in \mathbb{C}^{d} \mid \sum_{\{1 \leq j \leq d \mid E^{\prime} \subseteq \Var(f_{j}) \}} \lambda_{j} \neq p) \} \subseteq \mathbb{C}^{d}.
\]
Consider the countable collection of these opens, parameterized both by $p$ and by $E^{\prime}$:
\[
\{G_{E^{\prime}, p} \mid  E^{\prime} \in \mathscr{L}(\mathscr{A}) \text{ and } p \in \mathbb{Z}_{\geq \min\{2, \rank(E^{\prime})\}} \}.
\]
By Theorem \ref{thm - twisted LCT implies generating level of meromorphic specialization}, if $\boldsymbol{\lambda}$ lies in the intersection of all these $G_{E^{\prime},p}$, then 
\[
\mathscr{D}_{X} f^{\textbf{-1} + \boldsymbol{\lambda}} = \mathscr{O}_{X}(\star f^{\boldsymbol{\lambda}}).
\]

Let $\gamma$ be the ring map of localization at $T_{j} = \mathbb{C}[S]\setminus \cup_{\mathfrak{m} \in \Var(\tau^{j}(\mathfrak{q}))} \mathfrak{m}$:
\[
\gamma = \mathbb{C}[s_{1}, \dots, s_{d}] \xrightarrow[]{} T_{j}^{-1} \mathbb{C}[S].
\]
And let $\gamma^{\sharp}: \Var(\mathfrak{q}) \to \Spec \mathbb{C}[s_{1}, \dots, s_{d}]$ be the induced map on spectra.
By definition $\textbf{a} \in \Var(\mathfrak{q})$ equates to
\begin{align} \label{eqn - bounding sum of coordinates from point of q}
\sum_{\{1 \leq j \leq d \mid E \subseteq \Var(\ell_{j})\}}  (\gamma^{\sharp}(\textbf{a}))_{j} = -m \leq -( - \min\{2, \rank(E)\} + 1)
\end{align}
where $E$ and $m$ are as determined in \emph{General Plan} (see \eqref{eqn - choice of m; F a complete factorization}). Therefore $(\gamma^{\sharp})^{-1}(G_{E, p})$ is nonempty for all $p \in \mathbb{Z}_{\geq \min\{2, \rank(E)\}}$. 

Additionally, for edges $E^{\prime} \neq E$ and for all $p \in \mathbb{Z}_{\geq \min\{2, \rank(E)\}}$ we claim our Zariski open $(\gamma^{\sharp})^{-1}(G_{E^{\prime}, p})$ is nonempty. If there exists a $k$ such that $\Var(f_{k}) \subseteq E$ and $\Var(f_{k}) \nsubseteq E^{\prime}$ then for a point $\boldsymbol{\beta} \in G_{E^{\prime}, p}$ there is no restriction on the entry $(\boldsymbol{\beta})_{k}$. So changing this value moves $\boldsymbol{\beta}$ within $G_{E^{\prime}, p}$ and thus we can change this value such that \eqref{eqn - bounding sum of coordinates from point of q} is satisfied. Conversely, if there exists a $k$ such that $\Var(f_{k}) \subseteq E^{\prime}$ but $\Var(f_{k}) \nsubseteq E$, then we can take a point $\textbf{a} \in \Var(\mathfrak{q})$, and change the entry $(\boldsymbol{\gamma}^{\sharp}(\textbf{a}))_{k}$ so that the point moves into $G_{E^{\prime}, p}$ but \eqref{eqn - bounding sum of coordinates from point of q} is not effected.

Therefore
\begin{equation} \label{eqn - countable collection of opens, generically generating twisted power meromorphic d-module}
\mathscr{G} = \{ (\gamma^{\sharp})^{-1}(G_{E^{\prime},p}) \mid E^{\prime} \in \mathscr{L}(\mathscr{A}) \text{ and } p \in \mathbb{Z}_{\geq \min\{2, \rank(E^{\prime})\}} \} 
\end{equation}
is a countable collection of nonempty sets lying in $\Var(\mathfrak{q})$, all of which are Zariski open in $\Var(\mathfrak{q})$. And if $\textbf{a} \in \bigcap_{U \in \mathscr{G}} U$, we have that
\begin{equation} \label{eqn - generically generating twisted power meromorphic d-module}
\mathscr{D}_{X} f^{\textbf{-1} + \gamma^{\sharp}(\textbf{a})} = \mathscr{O}_{X}(\star f^{\gamma^{\sharp}(\textbf{a})}).
\end{equation}

\emph{Step 3:} Combining Generic Properties.

Define $\Gamma \subseteq \Var(\mathfrak{q})$ by
\begin{equation} \label{eqn - def of gamma, very generic points}
\Gamma = (\bigcap_{C \in \mathscr{C}} C ) \cap (\bigcap_{G \in \mathscr{G}} G)
\end{equation}
and let $\textbf{a} \in \Gamma$. We will show that if $j \in \{2, 3, \dots, \}$ then $(\tau^{\sharp})^{-j}(\Gamma)$ is disjoint from the $Z(B_{F,0})$, zero locus of the Bernstein--Sato ideal of $M$. This verifies \textbf{The Claim}.

By \cite{SabbahI} (see also Proposition 3.2 of \cite{OakuAlgorithm}), only finitely many of $\{(\tau^{\sharp})^{-j})(\textbf{a}) \mid j \in \{2, 3, \dots \} \}$ intersect $Z(B_{F,0})$. Assume, towards contradiction, that the intersection is non-empty, and fix $j \in \{2, 3, \dots \}$ to be the largest choice such that $(\tau^{\sharp})^{-j}(\textbf{a}) \in Z(B_{F,0})$. For notational ease, let $\textbf{j} = (j, \dots, j) \in \mathbb{C}^{d}$ and recall that $\gamma^{\sharp}((\tau^{\sharp})^{-j}(\textbf{a})) =  - \textbf{j} + \gamma^{\sharp}(\textbf{a}).$ Now consider the following commutative diagram of $\mathscr{D}_{X, 0}$-maps:
\begin{equation} \label{eqn - comm diagram nabla specialization}
\begin{tikzcd}
\mathscr{D}_{X, 0}[S] F^{S} \otimes_{\mathbb{C}[S]} \mathbb{C}_{(\tau^{\sharp})^{-j + 1}(\textbf{a})} \rar[twoheadrightarrow] \dar{\nabla_{(\tau^{\sharp})^{-j + 1}(\textbf{a})}}
    & \mathscr{D}_{X, 0} f^{- \textbf{j} + \textbf{1} + \gamma^{\sharp}(\textbf{a})} \dar[hookrightarrow]{=}\\
\mathscr{D}_{X, 0}[S] F^{S} \otimes_{\mathbb{C}[S]} \mathbb{C}_{(\tau^{\sharp})^{-j}(\textbf{a})} \rar[twoheadrightarrow]{\simeq}
    & \mathscr{D}_{X, 0} f^{- \textbf{j} + \gamma^{\sharp}(\textbf{a})}.
\end{tikzcd}
\end{equation}
(Recall that $\mathbb{C}_{\boldsymbol{\beta}}$ is the residue field of the point $\boldsymbol{\beta} \in \Spec \mathbb{C}[S]$.) Here the horizontal specialization maps are always surjections, the rightmost vertical map is submodule inclusion, and the leftmost vertical map $\nabla_{(\tau^{\sharp})^{-j + 1}(\textbf{a})}$ is induced by sending $s_{k} \mapsto s_{k} + 1$ for all $k$, cf. Definition 3.7 of \cite{Bath3} and/or the commutative square of subsection 5.10 of \cite{BudurLocalSystems} where $\nabla_{(\tau^{\sharp})^{-j + 1}(\textbf{a})}$ goes by $\rho_{(\tau^{\sharp})^{-j + 1}(\textbf{a})}$. Because $j$ is the largest possible value such that $(\tau^{\sharp})^{-j}(\textbf{a}) \in Z(B_{F,0})$, Proposition 3.6 of \cite{OakuAlgorithm} says the bottom horizontal map is an isomorphism. And because $-j + 1 \leq -1$ and $\textbf{a} \in \Gamma$, the rightmost vertical map is actually an equality (see \eqref{eqn - countable collection of opens, generically generating twisted power meromorphic d-module}, \eqref{eqn - generically generating twisted power meromorphic d-module}).

Therefore $\nabla_{(\tau^{\sharp})^{-j + 1}(\textbf{a})}$ is surjective. This means its cokernel vanishes:
\[
M \otimes_{\mathbb{C}[S]} \mathbb{C}_{(\tau^{\sharp})^{-j}(\textbf{a})} = 0.
\]
Since $\textbf{a} \in \Gamma$, the definition \eqref{eqn - collection of opens, generically CM translates} of $\mathscr{C}$ and the subsequent remarks imply that the vanishing of this cokernel equates to 
\[
(\tau^{\sharp})^{-j}(\textbf{a}) \notin Z(B_{F,0}). 
\]
But this contradicts our choice that $j \in \{2, 3, \dots, \}$ was the largest possible value such that $(\tau^{\sharp})^{-j}(\textbf{a}) \in Z(B_{F,0}).$ We conclude that
\begin{equation} \label{eqn - translates of very generic points are disjoint from zero loci of BS ideals}
(\tau^{\sharp})^{-j}(\Gamma) \cap Z(B_{F,0}) = \emptyset \text{ for all } j \in \{2, 3, \dots, \}.
\end{equation}

\emph{Step 4:} Baire Category Theorem.

The points $\Gamma \subseteq \Var(\mathfrak{q})$ from Step 3 are obtained by taking a countable intersection of nonempty Zariski open sets. By the Baire Category Theorem $\Gamma$ is nonempty and dense in $\Var(\mathfrak{q})$. Moreover, $(\tau^{\sharp})^{-j}(\Gamma) \subseteq (\tau^{\sharp})^{-j}(\Var(\mathfrak{q})) =  \Var(\tau^{j}(\mathfrak{q}))$ is also dense in $\Var(\tau^{j}(\mathfrak{q}))$ since $\tau^{\sharp}$ is a homeomorphism. By Step 3 and \eqref{eqn - translates of very generic points are disjoint from zero loci of BS ideals}, for all $j \in \{2, 3, \dots, \}$, the codimension one component $\Var(\tau^{j}(\mathfrak{q})) \nsubseteq Z(B_{F,0})$ lest $\emptyset \neq (\tau^{\sharp})^{-j}(\Gamma) \subseteq Z(B_{F,0})$. This proves \textbf{The Claim} and by the discussion in \emph{General Plan} completes the proof for \emph{Case 1}.

\emph{Case 2:} We assume $d > r > 1$, i.e. $F$ is neither the trivial factorization $F = (f)$ nor a factorization into linear forms.

The strategy is entirely similar to Case 1 but with some mild alterations involving \emph{Step 2}. First, if $E \neq \{0\}$ change \eqref{eqn - choice of m; F a complete factorization}, and the consequent definition of $\mathfrak{q} \in \Spec \mathbb{C}[S]$, so that 
\begin{align} \label{eqn - choice of m; F a neither complete nor trivial factorization}
m &= \min\{ p \in \{ \rank(E) + \ell + z \cdot d_{E} \mid z \in \mathbb{Z} \} \mid -p \notin \mathbb{Z}_{\geq \min\{2, \rank(E)\}} \} \\
    &\in [0, d_{E} - 1]. \nonumber
\end{align}
If $E = \{0\}$ make no such changes.

The construction of $\mathscr{C}$ (see \eqref{eqn - collection of opens, generically CM translates}) from \emph{Step 1} does not change. As for the construction of $\mathscr{G}$ from \emph{Step 2}, let $\beta: \mathbb{C}[s_{1}, \dots, s_{d}] \to \mathbb{C}[s_{1}, \dots, s_{r}]$ be the ring homomorphism induced by the factorization $F$ of $f$ (see Definition \ref{def - notation for factorization of arrangements along edges}):
\[
\beta: \mathbb{C}[s_{1}, \dots, s_{d}] \to \frac{\mathbb{C}[s_{1}, \dots, s_{d}]}{(\{ s_{t} - s_{j} \mid t, j \in S_{k}, 1 \leq k \leq r\})} = \mathbb{C}[s_{1}, \dots, s_{r}].
\]
And change $\gamma$ from \emph{Step 2} so that $\gamma$ is the composition of $\beta$ with the localization at $T_{j} = \mathbb{C}[S]\setminus \cup_{\mathfrak{m} \in \Var(\tau^{j}(\mathfrak{q}))} \mathfrak{m}$:
\[
\gamma = \mathbb{C}[s_{1}, \dots, s_{d}] \xrightarrow[]{\beta} \mathbb{C}[S] \xrightarrow[]{} T_{j}^{-1} \mathbb{C}[S].
\]
And $\gamma^{\sharp}: \Var(\mathfrak{q}) \to \Spec \mathbb{C}[s_{1}, \dots, s_{d}]$ is the induced map on spectra. As before, define 
\[
\mathscr{G} = \{ (\gamma^{\sharp})^{-1}(G_{E^{\prime},p}) \mid E^{\prime} \in \mathscr{L}(\mathscr{A}) \text{ and } p \in \mathbb{Z}_{\geq \min\{2, \rank(E^{\prime})\}} \} 
\]
and, as before, we must show each element of $\mathscr{G}$ is nonempty. 

Under these modifications, $\textbf{a} \in \Var(\mathfrak{q})$ equates to
\[
\sum_{\{1 \leq j \leq d \mid E \subseteq \Var(\ell_{j})\}}  (\gamma^{\sharp}(\textbf{a}))_{j} = -m \leq 
    \begin{cases}
        0 \text{ if } E \neq \{0\}; \\
        -1 \text{ if } E = \{0\}.
    \end{cases}
\]
There are two cases. If $-m \leq 0$, then $\mathbb{Q}_{\leq 0}^{r}$ intersects $\Var(\mathfrak{q})$ nontrivially. Certainly $\mathbb{Q}_{\leq 0}^{d} \subseteq G_{E^{\prime}, p}$ for arbitrary $E^{\prime} \in \mathscr{L}(\mathscr{A})$ and $p \in \mathbb{Z}_{\geq \min\{2, \rank(E^{\prime})\}}.$ So for all such $E^{\prime}$ and $p$ we know $\emptyset \neq \gamma^{\sharp}(\mathbb{Q}_{\leq 0}^{r} \cap \Var(\mathfrak{q})) \subseteq \mathbb{Q}_{\leq 0}^{d} \subseteq G_{E^{\prime}, p}$ for arbitrary $E^{\prime} \in \mathscr{L}(\mathscr{A})$ and $p \in \mathbb{Z}_{\geq \min\{2, \rank(E^{\prime})\}}.$ Thus when $-m \leq 0$ every member of $\mathscr{G}$ is nonempty and Zariski open in $\Var(\mathfrak{q})$. If $-m > 0$, then by \eqref{eqn - choice of m; F a complete factorization} and \eqref{eqn - choice of m; F a neither complete nor trivial factorization} we conclude $-m = 1$ and $E = \{0\}$. Let $\textbf{1/d} = (1/d, \dots, 1/d) \in \mathbb{C}^{d}$ and note that $\textbf{1/d} \in G_{E^{\prime}, p}$ for all $E^{\prime} \neq \{0\}$ and $p \in \mathbb{Z}_{\geq \min\{2, \rank(E^{\prime})\}}$ (this amounts to $d_{E^{\prime}} < d$). Certainly $(\beta^{\sharp})^{-1}(\textbf{1/d}) \in \Var(\mathfrak{q})$ so in this case each relevant $(\beta^{\sharp})^{-1}(G_{E^{\prime}, p}) \neq \emptyset$.

With these alterations and arguments, the construction of $\mathscr{G}$ from \emph{Step 2} has the same properties as before. So the argument in \emph{Case 1} applies to \emph{Case 2}. Note that the change in $Q_{E}$ for edges not $\{0\}$ matches the change in the lower bound for $m$.

\emph{Case 3:} We assume $r=1$, i.e. $F$ is the trivial factorization $F = (f)$. 

We argue similarly to \emph{Case 2}. First, note that the changes to $\mathscr{G}$ therein apply here as well which, because $\Spec \mathbb{C}[S] = \mathbb{C}$, means $\mathscr{G}$ consists of copies of $\Var(\mathfrak{q})$. And because $r=1$, it is easy to check $M$ is automatically $(n+1)$-Cohen--Macaulay (e.g. Lemma 3.3.3 of \cite{ZeroLociI}). So $\mathscr{C}$ from \emph{Step 1} is simply copies of $\Var(\mathfrak{q})$. This means $\Gamma = \Var(\mathfrak{q})$ and the rest of the argument (i.e. \emph{Step 3}) holds. (There is no need to appeal to the Baire Category Theorem nor dense subset arguments since \emph{Step 3} applies to all of $\Gamma = \Var(\mathfrak{q})$.) Finally, the containment of the roots of the Bernstein--Sato polynomial in $(-2 + 1/d, 0)$ is immediate from the bound of codimension one components (i.e. roots) provided.
\end{proof}

\begin{remark} \label{rmk - on BS ideal codim one bound} \text{ }
\begin{enumerate}[label=(\alph*)]
    \item The bound in Theorem \ref{thm - bounding codim one components of BS zero loci} is sharp: the formula Theorem 3.23 of \cite{Bath3} for the Bernstein--Sato ideal of a generic arrangement exhibits this.
    \item Theorem \ref{thm - bounding codim one components of BS zero loci} is a significant improvement to the bound of Corollary 3.21 of \cite{Bath3} as there tameness was required in order to lift M. Saito's univariate's bound (cf. \cite{SaitoArrangements}) to the multivariate setting.
    \item Because the proof of Theorem \ref{thm - bounding codim one components of BS zero loci} depends on generic properties it yields no concrete evidence that $\Var(\tau^{j}(\mathfrak{q}))$ is disjoint from $Z(B_{F,0})$ for $j \in \{2, 3, \dots \}$: all we show is that the complement of their intersection is dense. In the tame case the intersection is disjoint because $Z(B_{F,0})$ is purely codimension one, see Corollary 3.5 \cite{Bath3}.
\end{enumerate}
\end{remark}

\bibliographystyle{abbrv}
\bibliography{refs}

\end{document}